%% file: avoid_inessential_edges.tex
\documentclass[12pt]{amsart}

\input{header_basic.tex}

\input{header_article.tex}
\input{header_subtle.tex}

\newcommand{\Tet}{\operatorname{\mathsf{T}}}
\newcommand{\ideal}{\operatorname{ideal}}
\newcommand{\Fix}{\operatorname{Fix}}

\title[Connecting essential triangulations I]{Connecting essential triangulations I:\\ via 2-3 and 0-2 moves}
\author{Tejas Kalelkar, Saul Schleimer, and Henry Segerman} 
\date{\today}

\begin{document}

\begin{abstract}
Suppose that $M$ is a compact, connected three-manifold with boundary.
We show that if the universal cover has infinitely many boundary components then $M$ has an ideal triangulation which is essential: no edge can be homotoped into the boundary.
Under the same hypotheses, we show that the set of essential triangulations of $M$ is connected via 2-3, 3-2, 0-2, and 2-0 moves. 

The above results are special cases of our general theory.
We introduce $L$--essential triangulations: boundary components of the universal cover receive labels and no edge has the same label at both ends.
As an application, under mild conditions on a representation, we construct an ideal triangulation for which a solution to Thurston's gluing equations recovers the given representation.

Our results also imply that such triangulations are connected via 2-3, 3-2, 0-2, and 2-0 moves.
Together with results of Pandey and Wong, this proves that Dimofte and Garoufalidis' 1-loop invariant is independent of the choice of essential triangulation.
\end{abstract}



\maketitle

\section{Introduction}

\subsection*{Paths of triangulations}

The study of combinatorial triangulations, and equivalences between them, dates back to at least the 1920's.
We refer to Lickorish~\cite{Lickorish99} for an overview of the history.
Matveev~\cite[Theorem~1.2.5]{Matveev07} and Piergallini~\cite[Theorem~1.2]{Piergallini88} independently proved the following: see \refsec{BistellarMoves} for definitions.

\begin{theorem}
\label{Thm:MatveevPiergallini}
Suppose that $M$ is a closed three-manifold.
Suppose that $\calT$ and $\calT'$ are triangulations of $M$ each with exactly one vertex and at least two tetrahedra.
Then $\calT$ and $\calT'$ are connected by a sequence of 2-3 and 3-2 moves.
\end{theorem}

Amendola~\cite[Theorem~2.1]{Amendola05} improves this to deal with ideal triangulations of three-manifolds with boundary.
(Here we allow two-sphere boundary components.)
Together, the above results show that the set of triangulations of $M$ (having a fixed number of material vertices and with at least two tetrahedra) is connected via 2-3 and 3-2 moves.
This implies that the Turaev-Viro invariant~\cite{TuraevViro92} is independent of the triangulation appearing in its definition.
See also~\cite{BarrettWestbury96}.

Other three-manifold invariants require more: 
for example, the 1-loop invariant of Dimofte and Garoufalidis~\cite[Definition~1.2]{DimofteGaroufalidis13} requires its defining ideal triangulation admit a solution to Thurston's gluing equations~\cite[Chapter~4]{Thurston80} corresponding to the complete hyperbolic structure.
Such a solution exists if and only if the triangulation is \emph{essential}~\cite[Theorem~1]{SegermanTillmann11}.
Our \refthm{ConnectivityWith0-2} implies the following.

\begin{corollary}
\label{Cor:ConnectivityWith0-2}
Suppose that $M$ is a compact, connected three-manifold with boundary.
Suppose that the universal cover $\cover{M}$ has infinitely many boundary components.
Then the set of essential ideal triangulations of $M$ is connected via 2-3, 3-2, 0-2, and 2-0 moves. \qed
\end{corollary}

Dimofte and Garoufalidis~\cite[Theorem~1.4]{DimofteGaroufalidis13} check that the 1-loop invariant is unchanged by 2-3 moves between essential triangulations; Pandey and Wong~\cite[Proposition~5.1]{PandeyWong24}
check that it is unchanged by 0-2 moves between essential triangulations.
Thus we have the following, answering Question 1.6 of \cite{DimofteGaroufalidis13} in the affirmative.

\begin{corollary}
\label{Cor:1-loop}
The 1-loop invariant (for a discrete and faithful representation) is independent of the choice of essential triangulation.  \qed 
\end{corollary}

\refcor{ConnectivityWith0-2} also gives a direct, geometric proof that the Bloch invariant is independent of its defining triangulation, reproving~\cite[Theorem~1.1]{NeumannYang99}.

Suppose that $\calT$ is an essential triangulation of a hyperbolic manifold $M$.
Let $S(\calT)$ be the variety of solutions to Thurston's gluing equations on $\calT$;
let $S_0(\calT)$ be the component containing the complete hyperbolic structure.
Another consequence of \refcor{ConnectivityWith0-2} is that if $\calT$ and $\calT'$ are two such essential triangulations then $S_0(\calT)$ and $S_0(\calT')$ are birationally equivalent~\cite[Proposition~2.22]{PandeyWong24}.
Furthermore, each has a smooth point at (the shapes giving) the complete hyperbolic structure~\cite[Proposition~2.23]{PandeyWong24}.

\subsection*{Essential triangulations}

Here is another purely combinatorial consequence of our work.

\begin{corollary}
\label{Cor:InfiniteIFFEssential}
Suppose that $M$ is a compact, connected three-manifold with boundary.
Suppose that $\pi_1(M)$ is infinite. 
Then the universal cover of $M$ has infinitely many boundary components if and only if $M$ admits an essential triangulation.
\end{corollary}

\begin{proof}
The forward direction is a special case of \refthm{ExistsTriangulation} (with $L$ the identity labelling, as in \refexa{Essential}). 
The backwards direction is a special case of \refcor{ManyRegions}.
\end{proof}

Hodgson, Rubinstein, Tillmann, and the third author~\cite[page~1105]{HodgsonRubinsteinSegermanTillmann15} construct essential triangulations given various hypotheses.
\refcor{InfiniteIFFEssential} recovers their second, third, and fourth constructions.
However, our result is more general; for example, we also deal with various reducible manifolds and the complements of almost all links.

\subsection*{$L$-essential triangulations}

Our results also apply more generally.
Suppose that $M$ is a compact, connected three-manifold with boundary.
Let $\cover{M}$ be the universal cover of $M$.
Suppose that $\calT$ is an ideal triangulation of $M$.
Let $\cover{\calT}$ be the induced triangulation of $\cover{M}$.

Let $\Delta_M$ be the set of boundary components of $\cover{M}$.
Given any set of \emph{labels} $\calL$, equipped with an action of $\pi_1(M)$, 
we say that a \emph{labelling} is a $\pi_1(M)$--equivariant function $L$ from $\Delta_M$ to $\calL$. 
We say that $\calT$ is \emph{$L$--essential} if no edge of $\cover{\calT}$ has the same label at both ends.
With the above notation, our main results are as follows.

\begin{restate}{Theorem}{Thm:ExistsTriangulation}
Suppose that $L$ is a labelling of $\Delta_M$ with infinite image.
Then there is an $L$--essential ideal triangulation of $M$.
\end{restate}

\begin{restate}{Theorem}{Thm:ConnectivityWith0-2}
Suppose that $L$ is a labelling of $\Delta_M$ with infinite image.
Then the set of $L$--essential ideal triangulations of $M$ is connected via 2-3, 3-2, 0-2, and 2-0 moves.
\end{restate}

If we take $\calL = \Delta_M$ and choose $L$ to be the identity function, then a triangulation is $L$--essential if and only if it is essential.

Coset spaces for $\pi_1(M)$ give more examples of labellings.
Applying \refthm{ExistsTriangulation} to these yields a triangulation with no edges in the given subgroup (see \refcor{SubgroupAvoiding}).
For example, we construct triangulations with no null-homologous edges (\refcor{NullHomologous}) and triangulations of Seifert fibred spaces with no edges lying in the fibration (\refcor{Seifert}).
Furthermore, by \refthm{ConnectivityWith0-2}, such sets of triangulations are connected via 2-3, 3-2, 0-2, and 2-0 moves.

Suppose that $\rho \from \pi_1(M) \to \PSL(2,\CC)$ is a representation.
A labelling with $\calL = \bdy_\infty\HH^3$, and equivariant with respect to $\rho$, is called an \emph{anchoring} of $\rho$ (\refdef{Anchoring}).
Anchorings relate to Thurston's gluing equations as follows. 

\begin{restate}{Lemma}{Lem:RecoverRho}
Suppose that $\rho$ is a representation and $\calT$ is an ideal triangulation.
Then $\calT$ is $L$--essential for some anchoring $L$ if and only if $\calT$ is \emph{$\rho$--regular} in the sense of \cite[Definition~4.2]{DimofteGaroufalidis13}. 
(That is, Thurston's gluing equations, over $\calT$, admit solutions.
Moreover, one of these has holonomy representation conjugate to $\rho$.)
\end{restate}

Not every representation $\rho$ has an anchoring $L$ which admits an $L$--essential triangulation.
However, this is the case under a mild side hypothesis, namely that $\rho$ is \emph{infinitely anchorable}: 
see Definitions~\ref{Def:Anchorable} and~\ref{Def:InfinitelyAnchorable} and \refcor{RepTriangulation}.
In \refthm{InfinitelyAnchorableZariskiNeighbourhood} we prove that representations sufficiently close to a discrete and faithful representation are infinitely anchorable.
Our main results give the following.

\begin{corollary}
\label{Cor:InfinitelyAnchorableExistsConnected}
Suppose that $M$ is a compact, connected three-manifold with boundary.
Suppose that $\rho \from \pi_1(M) \to \PSL(2,\CC)$ is an infinitely anchorable representation.
Then there is an anchoring $L$ of $\rho$ 
such that the set of $L$--essential ideal triangulations of $M$ is non-empty and is connected via 2-3, 3-2, 0-2, and 2-0 moves. \qed
\end{corollary}

This allows us to generalise \refcor{1-loop} to representations $\rho$ which are not discrete and faithful.
Suppose that $L$ is an anchoring of $\rho$.
(Near, but not at, a discrete and faithful representation, this requires choosing one of the two fixed points of each peripheral group as our ``anchor''. 
Dunfield discusses this in the special case of a once-cusped manifold~\cite[Section~10.1]{BRVD05}.)

We now restrict attention to $L$--essential triangulations: 
Dimofte and
Garoufalidis~\cite[Theorem~4.1]{DimofteGaroufalidis13} prove that the 1-loop invariant is unchanged by 2-3 moves; 
Pandey and Wong~\cite[Proposition~5.1]{PandeyWong24} prove that it is unchanged by 0-2 moves.
\refcor{InfinitelyAnchorableExistsConnected} then gives the following.

\begin{corollary}
\label{Cor:1-loopAnchorable}
Suppose that $\rho$ is an infinitely anchorable representation. 
Then there is an anchoring $L$ of $\rho$ such that
the 1-loop invariant for $\rho$ and $L$ is independent of the choice of $L$--essential triangulation appearing in its definition.  \qed 
\end{corollary}

In \refprop{HMP} we follow \cite{HowieMathewsPurcell21} and use our techniques to build an $L$--essential triangulation suitable for Dehn filling.
Pandey and Wong use this and \refthm{InfinitelyAnchorableZariskiNeighbourhood} in their proof of the following~\cite[Corollary~1.16]{PandeyWong24}.

\begin{theorem}[Pandey-Wong]
\label{Thm:PW}
The 1-loop invariant is equal to the adjoint twisted Reidemeister torsion for all hyperbolic three-manifolds obtained by doing sufficiently long Dehn-fillings on the boundary components of any fundamental shadow link complement.
\end{theorem}

The \emph{1-loop conjecture}~\cite[Equation~1-4]{DimofteGaroufalidis13} states that the conclusion holds for all complete non-compact hyperbolic three-manifolds.
\refthm{PW} verifies the conjecture for a large new class of manifolds.

\subsection*{Future directions}

\refthm{MatveevPiergallini} only requires 2-3 and 3-2 moves to connect triangulations together.
Our \refthm{ConnectivityWith0-2} also requires the use of 0-2 and 2-0 moves.
In the second paper in this series~\cite{KSS24b}, we will give the side-hypotheses needed to connect $L$--essential triangulations together using only 2-3 and 3-2 moves.

In future work, the second two authors will apply this connectivity result to prove the following.

\begin{theorem}
\label{Thm:Veering}
The figure-eight knot complement admits a unique veering triangulation.
\end{theorem}

In previous work~\cite[Theorem~5.1]{FSS19} we show that a veering triangulation induces a unique ``compatible'' circular order on $\Delta_M$.
The proof of \refthm{Veering} relies on the ability to carry these circular orderings through a sequence of essential (but in general, non-veering) triangulations. 

A more speculative application is the following.

\begin{conjecture}
\label{Conj:APoly}
Suppose that $M$ is a complete finite volume once-cusped hyperbolic three-manifold. 
Then there exists an ideal triangulation $\calT$ that ``yields'' all factors of the $A$--polynomial.
\end{conjecture}


Marc Culler~\cite{Culler05, KnotInfo24} gives a list of $A$--polynomials of knot complements, computed using the shape variety for some ideal triangulation $\calT$.
If $\calT$ is not $L$--essential, for any anchoring of any representation $\rho$ in some component of the character variety, then Culler's algorithm will miss the corresponding factor of the $A$--polynomial.
Experimentally, Culler found that in at least one example it is possible to retriangulate to move from a triangulation that does not see all factors to one that does.
See~\cite[Remark~11.5]{Segerman12}.
It seems likely that for a given manifold (perhaps with extra hypotheses), a labelling could be constructed that when used in \refthm{ExistsTriangulation} answers \refconj{APoly}.

\subsection*{Outline}

We give the necessary background in \refsec{Definitions}.
We prove \refthm{ExistsTriangulation} in \refsec{BuildLEssential}.
The idea is as follows.
We start with any triangulation.
We then take the first barycentric subdivision.
This breaks apart any $L$--inessential edges at the cost of introducing many material vertices.
We then remove each material vertex by ``identifying'' it with an ideal vertex. 
The tool we use to do this is a \emph{snake}: see \refsec{SnakePaths}.
These must be carefully chosen (using \refalg{CreateSnakePath}) to avoid creating any new $L$--inessential edges.

The strategy to prove \refthm{ConnectivityWith0-2} is to repeatedly improve a path of triangulations connecting two given triangulations. 
In \refsec{SimplicialConnectivity} we apply a result of Casali~\cite{Casali95} to construct a path of $L$--essential triangulations using 2-3 moves, bubble moves, and their inverses (\refcor{ConnectivitySimplicial}).
Bubble moves increase the number of material vertices in the triangulation. 
In Sections~\ref{Sec:SnakeHandling} and~\ref{Sec:TraversingWithSnakes} we fix this, replacing the use of bubble moves with 0-2 moves.
This is a difficult procedure.
A bubble move introduces a new material vertex, but 2-3 and 0-2 moves do not.
So we must closely shadow the given sequence of triangulations in which the number of vertices changes.
We do this while not actually adding any material vertices or creating any $L$--inessential edges. 

The theorems are stated in terms of triangulations. 
However, the proofs mostly deal with the dual objects: special spines, following Matveev.
In this paper we use the more descriptive term \emph{foam}.
Foams and triangulations record the same combinatorial information.
However, pictures of foams are often easier to understand than the corresponding pictures for triangulations.

\subsection*{Acknowledgements}

We thank Moishe Kohan and Andy Putman for their helpful comments on the proof of \reflem{ManyRegions}.
We thank Ka Ho Wong for many enlightening conversations regarding the 1-loop invariant.
We thank Marc Culler and Peter Shalen for their helpful comments on character varieties.
We thank Abhijit Champanerkar and Matthias Goerner for useful conversations about the Bloch invariant.
The third author was supported in part by National Science Foundation grant DMS-2203993.

\section{Definitions and preliminaries}
\label{Sec:Definitions}

\subsection{Triangulations and foams}
Let 
\[
\Tet = \left\{x \in \RR^4  \thinspace \middle| \thinspace \sum x_i = 1, x_i \geq 0\right\}
\]
be the standard tetrahedron.
The vertices, edges, and faces of $\Tet$ are called \emph{model cells}.
A \emph{triangulation} $\calT$ is a collection of copies of $\Tet$, called the \emph{model tetrahedra}, as well as a collection of \emph{face pairings}.
These are affine isomorphisms of pairs of model faces.
We require that every model face appear in exactly two face pairings (once as a domain and once as a codomain).
The \emph{realisation} $|\calT|$ is the topological space obtained by taking the disjoint union of the model tetrahedra and quotienting by the face pairings.
We define the \emph{vertices} of $|\calT|$ to be the images of the model vertices.
We take a subset $\calT_{\ideal}^{(0)}$ of these to be the \emph{ideal vertices} of $\calT$.
The remaining vertices are the \emph{material vertices}.

Suppose that $M$ is a compact, connected three-manifold, possibly with boundary.
We say that $\calT$ is a \emph{partially ideal triangulation} of $M$ if 
\[
|\calT| - \calT_{\ideal}^{(0)} \homeo M - \bdy M
\]
We say that $\calT$ is an \emph{ideal triangulation} of $M$ if it has no material vertices.
We say that $\calT$ is a \emph{material triangulation} of $M$ if it has no ideal vertices.
In the latter case $\bdy M$ must be empty.

\begin{definition}
\label{Def:Foam}
Suppose that $M$ is a compact, connected three-manifold, possibly with boundary.
Suppose that $\calF$ is a finite two-dimensional piecewise-linear CW complex embedded in $M$ with the following properties.
\begin{enumerate}
\item The attaching maps are locally injective. 
\item Every zero-cell is adjacent to four ends of one-cells.
\item The intersection of $\calF$ with a small regular neighbourhood of a zero-cell is isomorphic to the cone over the one-skeleton of a tetrahedron.
\item Every one-cell is adjacent to three segments of boundaries of two-cells.
\item 
\label{Itm:Complement}
Every complementary component of $M - \calF$ is either an open three-ball or is homeomorphic to a product of a component of $\bdy M$ with $[0,1)$.
We call the former \emph{material regions} and the latter \emph{peripheral regions}.
\end{enumerate}
(See Figures~\ref{Fig:PointTypesOnFoam1},~\ref{Fig:PointTypesOnFoam2}, and~\ref{Fig:PointTypesOnFoam3} for small neighbourhoods of points of $\calF$ in $M$.)
We call $\calF$ a \emph{foam} in $M$.
\end{definition}

Example foams in $S^3$ and in $S^1 \times S^2$ are illustrated in \reffig{ExampleFoams}.
This terminology appears in \cite[page~1047]{Khovanov04} and \cite[Section~5]{Sullivan99}.
Matveev~\cite[Section~1.1.4]{Matveev07} calls $\calF$ a \emph{special spine} for $M$.

\begin{figure}[htbp]
\subfloat[]{
\includegraphics[height = 2.4cm]{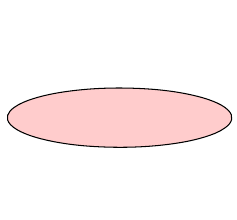}
\label{Fig:PointTypesOnFoam1}
}
\quad
\subfloat[]{
\includegraphics[height = 2.4cm]{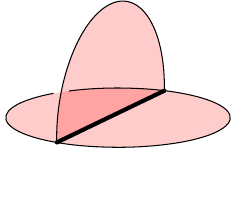}
\label{Fig:PointTypesOnFoam2}
}
\quad
\subfloat[]{
\includegraphics[height = 2.4cm]{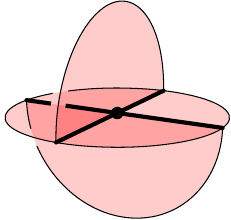}
\label{Fig:PointTypesOnFoam3}
}
\quad
\subfloat[]{
\includegraphics[height = 2.4cm]{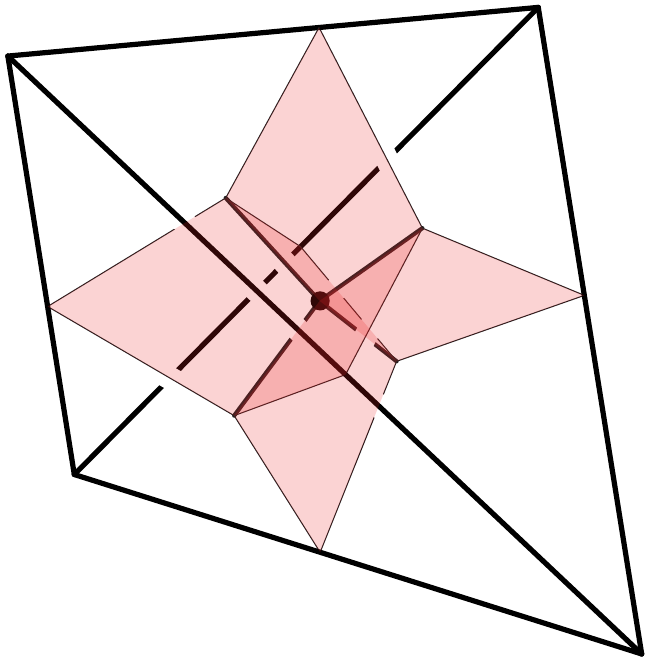}
\label{Fig:Butterfly}
}
\caption{Local pictures of foams.}
\label{Fig:Foam}
\end{figure}

\begin{remark}
\label{Rem:FoamConnected}
Note that, since $M$ is connected and $M$ (minus a point from each material region) deformation retracts to $\calF$, it follows that $\calF$ is connected.
\end{remark}

\begin{figure}[htbp]
\subfloat[A foam in $S^3$.]{
\includegraphics[height = 4.0cm]{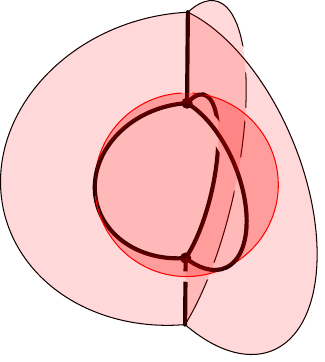}
\label{Fig:S3Example}
}
\quad
\subfloat[A foam in $S^1 \times D^2$.]{
\includegraphics[height = 4.0cm]{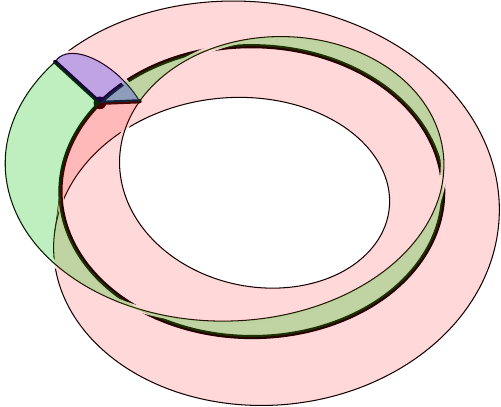}
\label{Fig:S2xS1Example}
}
\caption{Example foams. We obtain a foam in $S^1 \times S^2$ by doubling the foam in \reffig{S2xS1Example} across the boundary of the solid torus.}
\label{Fig:ExampleFoams}
\end{figure}

As discussed in~\cite[Section~2.7]{RubinsteinSegermanTillmann19}, we have the following duality (also see \reffig{Butterfly}).

\begin{lemma}
Suppose that the foam $\calF$ is dual to the triangulation $\calT$.
Then material regions of $\calF$ are dual to material vertices of $\calT$, while peripheral regions of $\calF$ are dual to ideal vertices of $\calT$. \qed
\end{lemma}


We deduce the following.
A foam for which all complementary regions are material is dual to a material triangulation.
A foam for which all complementary regions are peripheral is dual to an ideal triangulation.
A foam with both kinds of complementary region is dual to a partially ideal triangulation.

\begin{definition}
\label{Def:Simplicial}
Suppose that $\calT$ is an ideal or partially ideal triangulation. 
Suppose that every $k$--simplex of $\calT$ has $k+1$ distinct vertices.
Suppose that any set of $k+1$ vertices of $\calT$ are the vertices of at most one $k+1$--simplex of $\calT$.
Then we say that $\calT$ is \emph{simplicial}.
Suppose further that no edge of $\calT$ has both endpoints being ideal vertices.
Then we say that $\calT$ is \emph{insulated simplicial}.
\end{definition}

\subsection{Universal cover}

Suppose that $M$ is a compact, connected three-manifold, possibly with boundary.
We use the notation $\phi = \phi_M \from \cover{M} \to M$ for the universal covering of $M$.
For every boundary component $C$ of $M$, we call a component of $\phi^{-1}(C)$ an \emph{elevation} of $C$.
Note that an elevation is a lift precisely when $C$ is a sphere.

Suppose that $\calT$ is a partially ideal triangulation of $M$.
We use $\cover{\calT}$ to denote the induced partially ideal triangulation of $\cover{M}$.
For every image of a model cell $c$ in $M$, there is an induced countable collection of images of model cells in $\cover{M}$.
We call these the \emph{lifts} of $c$ to $\cover{M}$.

We use $\Delta_M$ to denote the set of components of $\bdy \cover{M}$.
Note that there is a natural bijection between $\Delta_M$ and $\cover{\calT}_{\ideal}^{(0)}$.

\subsection{Labellings and $L$--essentiality}
\label{Sec:Label}

\begin{definition}
\label{Def:Labelling}
Suppose that $\calL$ is a set, equipped with an action of $\pi_1(M)$.
We say that $\calL$ is a set of \emph{labels}.
Suppose that $L \from \Delta_M \to \calL$ is a $\pi_1(M)$--equivariant 
function.
Then we call $L$ a \emph{labelling} of $\Delta_M$.
\end{definition}

\begin{example}
\label{Exa:Essential}
We take $\calL = \Delta_M$ and we take $L$ to be the identity map.
\end{example}

\begin{example}
\label{Exa:CovEssential}
Suppose that $\psi \from M' \to M$ is a regular covering map.
Let $\phi' \from \cover{M} \to M'$ be the induced universal covering. 
Let $\calL$ be the set of components of $\bdy M'$.
Since $\psi$ is regular, the fundamental group $\pi_1(M)$ acts on $M'$, by homeomorphisms, and surjecting the deck group.
This gives the action of $\pi_1(M)$ on $\calL$.
The labelling $L$ is given by $L(c) = \phi'(c)$. 
\end{example}

\begin{definition}
\label{Def:Stab}
Suppose that $\pi_1(M)$ acts on a set $X$.
Suppose that $c \in X$.
We use the following standard notation for the \emph{stabiliser} of $c$.
\[
\Stab(c) = \{ \alpha \in \pi_1(M) \st \alpha \cdot c = c \} \qedhere
\]
\end{definition}

\begin{definition}
\label{Def:Anchorable}
Suppose that $M$ has non-empty boundary.
Suppose that $\rho \from \pi_1(M) \to \PSL(2,\CC)$ is a representation.
Let $\Fix_\rho(c) \subset \bdy_\infty \HH^3$ be the fixed points at infinity for $\rho(\Stab(c))$.
We say that $\rho$ is \emph{anchorable at} $c$ if $\Fix_\rho(c)$ is non-empty.
We say that $\rho$ is \emph{anchorable} if it is anchorable at $c$ for all $c$ in $\Delta_M$.
\end{definition}

Note that $\Fix_\rho(c)$ may be empty even when $c$ is an elevation of a torus boundary component:
see \refexa{Mattress}.

\begin{definition}
\label{Def:Anchoring}
Suppose that $\rho \from \pi_1(M) \to \PSL(2,\CC)$ is a representation.
A labelling $L$ is an \emph{anchoring} of $\rho$ if $\calL = \bdy_\infty\HH^3$ and we have $L(\gamma(c)) = \rho(\gamma)(L(c))$ for all $\gamma \in \pi_1(M)$.
\end{definition}

\begin{lemma}
\label{Lem:Anchoring}
A representation $\rho \from \pi_1(M) \to \PSL(2,\CC)$ is anchorable if and only if there exists an anchoring of $\rho$.
\end{lemma}

\begin{proof}
In the forward direction we equivariantly pick, for each $c$ in $\Delta_M$, a point $z_c \in \Fix_\rho(c)$.
Set $L(c) = z_c$.
In the backwards direction, equivariance implies that $\Fix_\rho(c)$ is non-empty.
\end{proof}


A construction very similar to the forward direction of \reflem{Anchoring} appears in \cite[Section~2.5]{Dunfield99}.
Also see the last sentence of \cite[page~216]{Thurston86}.

\subsubsection{Label-essential triangulations}

\begin{definition}
\label{Def:LEssential}
Suppose that $\calT$ is an ideal triangulation of $M$.
Suppose that $L$ is a labelling of $\Delta_M$, as in \refdef{Labelling}.
Suppose that $e$ is an edge of $\calT$ with a lift $\cover{e}$ in $\cover{\calT}$.
Suppose that $\cover{u}$ and $\cover{v}$ are the endpoints of $\cover{e}$.
If $L(\cover{u}) = L(\cover{v})$ then we say that $e$ is \emph{$L$--inessential}.
Otherwise we say that $e$ is \emph{$L$--essential}.
If all edges of $\calT$ are $L$--essential then we say that $\calT$ is \emph{$L$--essential}.
\end{definition}

\begin{remark}
\label{Rem:IdealEssential}
If $L$ is the identity labelling as in \refexa{Essential} then an ideal triangulation is $L$--essential if and only if it is \emph{essential} as in Definition~3.5 of \cite{HodgsonRubinsteinSegermanTillmann15}. 
\end{remark}

\begin{remark}
\label{Rem:MaterialEssential}
Suppose that $M$ is a closed manifold with a material triangulation $\calT$.
Let $\calN$ be a small regular neighbourhood of the vertices.
Let $M^\circ = M - \calN$.
Let $\calT^\circ$ be the resulting ideal triangulation of $M^\circ$.
Suppose that $L$ is the identity labelling of $\Delta_{M^\circ}$ as in \refexa{Essential}.
Then $\calT^\circ$ is $L$--essential if and only if $\calT$ is \emph{essential} in the sense of \cite[Definition~3.2]{HodgsonRubinsteinSegermanTillmann15} (in the one-vertex case) or in the sense of \cite[page~336]{LuoTillmannYang13} (more generally).
\end{remark}

\begin{remark}
For an $L$--essential triangulation to exist, the image of the labelling $L$ must have at least four elements.
As a corollary of this, the universal cover $\cover{M}$ must have at least four boundary components.  
We deduce that a handlebody, and similarly, a 
surface-cross-interval, cannot have an $L$--essential triangulation.
This is because their universal covers only have one and two boundary components, respectively.
\end{remark}

\begin{lemma}
\label{Lem:ManyRegions}
Suppose that $M$ is a compact, connected three-manifold with boundary.
Suppose that $\pi_1(M)$ is infinite. 
Suppose that $|\Delta_M| \geq 3$.
Then $\Delta_M$ is infinite.
\end{lemma}

\begin{proof}
Passing to a double cover, if necessary, we may assume that $M$ is oriented.
Fix now some $c \in \Delta_M$.
Let $C = \phi(c)$.
So $C$ is a compact, connected, oriented surface.
Fix a basepoint $p \in C$.
Let $\Gamma_C$ be the image of $\pi_1(C, p)$ in $\pi_1(M, p)$.

Suppose first that $\Gamma_C$ has infinite index in $\pi_1(M, p)$.
Then the full preimage of $C$ in $\cover{M}$ has infinitely many components and we are done.
Suppose instead that $\Gamma_C$ has finite index in $\pi_1(M, p)$.
In this case we replace $M$ with the finite cover with fundamental group $\Gamma_C$.
So we may assume that $\pi_1(C, p)$ surjects $\pi_1(M, p)$.
If the homomorphism induced by inclusion is an isomorphism then 
by \cite[Theorem~10.2]{Hempel04} (and 
by Perelman's resolution to the Poincar\'e conjecture), we deduce that $M$ is homeomorphic to a surface-cross-interval $C \times [0,1]$.
If instead the homomorphism has kernel then the disk theorem \cite[page~474]{Gordon99} realises $M$ as a boundary connect sum; the decomposing disk has boundary in $C$.
Induction on the (negative of the) Euler characteristic of $M$
proves that $M$ is a compression body (possibly with two-spheres in its internal boundary), with exterior boundary equal to $C$.
If $M$ is a handlebody or a product then the universal cover $\cover{M}$ has at most two boundary components, contrary to hypothesis.
We deduce that the compression body $M$ has an internal boundary component $C'$.
So $C'$ is incompressible in $M$ (or is a sphere) and has genus less than that of $C$.
We deduce that $\pi_1(C')$ has infinite index in $\pi_1(M)$.
So the full preimage of $C'$ in $\cover{M}$ has infinitely many components.
Again we are done.
\end{proof}

\begin{corollary}
\label{Cor:ManyRegions}
Suppose that $M$ is a compact and connected three-manifold with boundary.
Suppose that $\pi_1(M)$ is infinite. 
Suppose that $\calT$ is an $L$--essential ideal triangulation of $M$.
Then $\Delta_M$ is infinite.
\end{corollary}

\begin{proof}
Any ideal tetrahedron in $\cover{\calT}$ has all of its edges $L$--essential.
Thus its vertices have four distinct labels.
Thus it meets four distinct elements of $\Delta_M$. 
We can therefore apply \reflem{ManyRegions}.
\end{proof}

\subsection{Hyperbolic geometry}

Our generalisation from essential to $L$--essential triangulations has applications in hyperbolic geometry.
Given a triangulation $\calT$ we can try to solve \emph{Thurston's gluing equations}.
These give each tetrahedron $t$ of $\calT$ the shape of an ideal hyperbolic tetrahedron.
We record this shape as the cross-ratio of its four points in $\CP^1$.
Thurston's gluing equations for $\calT$ ensure that the hyperbolic tetrahedra fit together properly about each edge. 
For details, see~\cites[Chapter~4]{Thurston80}[Section~4.2]{Purcell20}.

Suppose that $Z$ is such a solution to the gluing equations.
By~\cite[Corollaries~2 and~10]{SegermanTillmann11}, associated to $Z$ we have a \emph{pseudo-developing map} $D_Z \from \cover{M} \to \HH^3$, a representation $\rho_Z \from \pi_1(M) \to \PSL(2, \CC)$, and a \emph{pseudo-developing map at infinity} $\bdy D_Z \from \Delta_M \to \bdy_\infty \HH^3$.

\begin{lemma}
\label{Lem:RecoverRho}
Suppose that $\rho \from \pi_1(M) \to \PSL(2,\CC)$ is a representation.
Suppose that $\calT$ is an ideal triangulation of $M$.
Then the following are equivalent:
\begin{itemize}
\item There is an anchoring $L$ of $\rho$ such that $\calT$ is $L$--essential.
\item There is a solution $Z$ of Thurston's gluing equations on $\calT$ so that $\rho_Z$ is conjugate to $\rho$. 
\end{itemize}
\end{lemma} 

\begin{remark}
As mentioned in the introduction, the second property exactly says that $\calT$ is \emph{$\rho$--regular} in the sense of \cite[Definition~4.2]{DimofteGaroufalidis13}. 
\end{remark}

\begin{proof}[Proof of \reflem{RecoverRho}]
Let $L$ be the given anchoring. 
Suppose that $t$ is a tetrahedron of $\calT$.
Choose a lift $\cover{t}$ of $t$ to $\cover{\calT}$.
Since $\calT$ is $L$--essential, the vertices of $\cover{t}$ map under $L$ to four distinct points of $\bdy_\infty \HH^3$.
The order of the model vertices of $t$ chooses a particular cross-ratio of these four points;
this gives a shape for $\cover{t}$.
Since $L$ is $\pi_1(M)$--equivariant, this shape is independent of the choice of lift $\cover{t}$.
Thus we get a shape for $t$ and so a collection of shapes $Z$. 

Suppose that $e$ is an edge of $\calT$.
Choose a lift $\cover{e}$ of $e$ to $\cover{\calT}$.
Let $(\cover{t}_i)$ be the cyclicly ordered list of tetrahedra of $\cover{\calT}$ around $\cover{e}$.
This gives a cyclic order to the ideal vertices of $\cover{\calT}$ incident to the tetrahedra $(\cover{t}_i)$ but not contained in $\cover{e}$.
We deduce that the product of the shapes of the $(\cover{t}_i)$ around $\cover{e}$ is $1$.
Thus the shapes $Z$ solve the gluing equations for $\calT$. 
Choosing the first three vertices correctly, the pseudo-developing map at infinity $\bdy D_Z$ equals the labelling $L$.
Likewise, $\rho_Z$ equals $\rho$. 

Now suppose that $Z$ is a solution of Thurston's gluing equations on $\calT$.
Let $D_Z$, $\rho_Z$, and $\bdy D_Z$ be the maps derived from $Z$.
After an appropriate conjugation, we may assume that $\rho_Z$ equals $\rho$.
We deduce that for each boundary component $c$ of $\Delta_M$, the point $\bdy D_Z(c)$ lies in $\Fix_\rho(c)$.
Moreover, for any $\gamma \in \pi_1(M)$ we have that $\bdy D_Z(\gamma \cdot c) = \rho(\gamma)(\bdy D_Z(c))$.
That is, $\bdy D_Z$ is an anchoring of $\rho$.
We claim that $\calT$ is $\bdy D_Z$--essential. 
This holds because every ideal tetrahedron under $D_Z$ has four distinct vertices on $\bdy_\infty \HH^3$.
\end{proof}

\subsubsection{Non-anchorable representations}

The following lemma controls the behaviour of non-anchorable representations for manifolds with torus boundary.
See also~\cite[Section~10.2]{BRVD05}.
As a piece of notation we set $K_4 = \ZZ/2\ZZ \oplus \ZZ/2\ZZ$.

\begin{lemma}
\label{Lem:NonAnchorable}
Suppose that $M$ is compact, connected, and oriented. 
Suppose that $c \in \Delta_M$ covers a torus boundary component of $M$.
Suppose that $\rho \from \pi_1(M) \to \PSL(2,\CC)$ is a representation.
If the set $\Fix_\rho(c)$ is empty then $\rho(\Stab(c)) \isom K_4$.
\end{lemma}

\begin{proof}
The subgroup $\rho(\Stab(c))$ is isomorphic to a quotient of $\ZZ^2$ and is thus abelian.
The only one of these which acts without fixed points on $\bdy_\infty\HH^3$ is $K_4$~\cite[Theorem~4.3.6]{Beardon95}.
\end{proof}



\begin{example}
\label{Exa:Mattress}
Suppose that $M$ is a $\ZZ/2\ZZ$--homology \emph{torus-cross-interval}:
that is, with coefficients in $\ZZ/2\ZZ$
\begin{itemize}
\item
the manifold $M$ has the homology of the two-torus and 
\item
the inclusion of each boundary component induces an isomorphism on homology.
\end{itemize}
So $H_1(M; \ZZ/2\ZZ)$ is isomorphic to $K_4$.  
Let $\rho \from \pi_1(M) \to \PSL(2, \CC)$ be any representation with image isomorphic to $K_4$.
The elements of order two are conjugate to rotations by $\pi$. 
Since they commute, they fix perpendicular and crossing axes.

Concretely, take $M$ equal to \texttt{m367} from the SnapPy census~\cite{SnapPy24};
this is an integral homology torus-cross-interval.
The fundamental group, using capital letters for inverses, is as follows.
\[
\pi_1(M) \isom \group{a, b}{aaaBAABaabAAAbaabAAB}
\]
The meridians and longitudes of the two cusps are as follows.
\[
\mu_0 = Ab
\quad
\lambda_0 = aaBAABaaa
\quad
\mu_1 = b
\quad
\lambda_1 = AABaaaBAA
\]
So we may take $\rho(a) = X$, take $\rho(b) = Y$, and extend to obtain the desired representation.

Other examples in the SnapPy cenus having such a representation include
\texttt{m125},
\texttt{m202},
\texttt{m203},
\texttt{m295},
\texttt{m328},
\texttt{m359}, and
\texttt{m366}.
\end{example}

\begin{remark}
It follows from \reflem{RecoverRho} that a non-anchorable representation cannot be recovered from a solution to Thurston's gluing equations for any ideal triangulation.
\end{remark}

\subsection{Partially ideal triangulations}

Our results are stated in terms of ideal triangulations. 
However, our constructions make use of partially ideal triangulations. 
We extend \refdef{LEssential} to this case as follows.

\begin{definition}
\label{Def:LEssentialPartiallyIdeal}
Suppose that $\calT$ is a partially ideal triangulation of $M$.
Suppose that $L$ is a labelling of $\Delta_M$, as in \refdef{Labelling}.
We say that $\calT$ is \emph{weakly $L$--essential} if there are no $L$--inessential edges between ideal vertices.
If in addition there are no edge loops based at material vertices then we say that $\calT$ is \emph{$L$--essential}.
\end{definition}

\begin{remark}
\label{Rem:InsulatedSimplicialLEssential}
Suppose that $\calT$ is an insulated simplicial partially ideal triangulation.
Then $\calT$ is $L$--essential for any labelling $L$.
\end{remark}

We use the notion of weak $L$--essentiality only in \refapp{HMP}.

\subsection{Essential foams}

\begin{definition}
\label{Def:FoamLabellingEssential}
Suppose that $\calT$ is a partially ideal triangulation of $M$.
Suppose that $\calF$ is the foam dual to $\calT$.
Suppose that $L \from\Delta_M \to \calL$ is a labelling of $\Delta_M$.
Then we obtain a function $L_\calF$ from peripheral regions of $\cover{M} - \cover{\calF}$ to $\calL$ as follows.
Suppose that $C$ is a peripheral region that contains the boundary component $c \in \Delta_M$.
Then we set $L_\calF(C) = L(c)$.
In what follows, we abuse notation and suppress the subscript $\calF$. 
We also say that a peripheral region $C$ \emph{has the label} $L(C)$.

Let $e$ be an edge of $\calT$.
Let $f$ be the corresponding dual face in $\calF$.
We say that $f$ is \emph{$L$--essential} or \emph{$L$--inessential} as $e$ is.
We say that $\calF$ is \emph{(weakly) $L$--essential} if $\calT$ is (weakly) $L$--essential.
\end{definition}

The following remark rephrases what it means for a face of a foam to be $L$--essential.

\begin{remark}
\label{Rem:NoLabelTouch}
With notation as in \refdef{FoamLabellingEssential},
suppose that $\cover{f}$ is a lift of $f$ to $\cover{\calF}$.
Then the face $f$ is $L$--inessential if and only if the regions of $\cover{M} - \cover{\calF}$ on either side of $\cover{f}$ have the same label. 

Similarly, an edge with both ends at a material vertex of $\calT$ corresponds to a face of $\calF$ with the same material region on both sides.
\end{remark}

\subsection{Edge loops and cyclic edges}

\begin{definition}
\label{Def:EdgeLoop}
An edge $e$ of a foam that has both ends at a single vertex is an \emph{edge loop}.
Any lift of $e$ to a cover is called a \emph{cyclic edge}.
\end{definition}  

Let $\calF$ be a foam in a manifold $M$.
Let $e$ be an edge loop in $\calF$ based at a vertex $v$.
Let $\calN$ be a small neighbourhood of $v$.
There are six disks of $(\calF \cap \calN) - \calF^{(1)}$.
See \reffig{PointTypesOnFoam3}.
Let $d$ be the disk that meets both ends of $e$.
Let $f$ be the face of $\calF$ that contains $d$.
There are now two cases:
the face $f$ does or does not contain other disks of $(\calF \cap \calN) - \calF^{(1)}$. 

Suppose that $f$ does not contain any other disks of $(\calF \cap \calN) - \calF^{(1)}$.
Then $f$ is \emph{degree-one}:
that is, $f$ has only one model edge (which is mapped to $e$).  
This is the configuration shown in \reffig{DegOne}.

Suppose instead that $f$ does contain other disks of $(\calF \cap \calN) - \calF^{(1)}$. 
If $M$ is orientable then the edge $e$ forms the core of a solid torus embedded in $M$. 
See \reffig{S2xS1Example}.
If $M$ is non-orientable then the edge $e$ forms the core of a solid torus or a solid Klein bottle embedded in $M$. 


\begin{figure}[htbp]
\includegraphics[height = 4.cm]{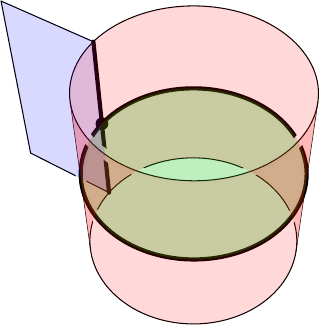}
\caption{An edge loop forming the boundary of a degree-one face.}
\label{Fig:DegOne}
\end{figure}

Our main hypothesis implies that there are no degree-one faces.

\begin{lemma}
Suppose that $\calT$ is an $L$--essential ideal triangulation.
Let $\calF$ be the dual foam.
Then $\calF$ has no degree-one faces.
\end{lemma}

\begin{proof}
Suppose for a contradiction that $f$ is a degree-one face incident to a vertex $v$.
Let $\calN$ be a small regular neighbourhood of $f$.
Thus $\calN$ is a three-ball.
See \reffig{DegOne} for a picture of $\calF \cap \calN$.
Let $g$ be the face of $\calF$ that meets $v$ along the two edges that $f$ does not. 
Consulting \reffig{DegOne}, we see that $g$ meets the same region on both sides.
Since the edge loop bounds a disk, it lifts to a loop in the universal cover.
Thus any lift of $g$ also meets the same region on both sides.
Thus $g$ is $L$--inessential for any labelling $L$.
\end{proof}

On the other hand, $L$--essentiality does not rule out edges at the cores of solid tori or solid Klein bottles.
These show up in practice, for example in foams dual to the canonical triangulations of hyperbolic knot complements.

Previous work~\cite{Segerman17} of the third author addresses the connectivity of the set of triangulations without degree-one edges.

\subsection{Basic moves on triangulations}
\label{Sec:TriangulationBasicMoves}

We use various moves to connect triangulations (and their dual foams) to each other.
We illustrate these moves on foams.
Suppose that $\calF$ is a foam in a manifold $M$.

\subsubsection{Bistellar moves} 
\label{Sec:BistellarMoves}

The three-dimensional bistellar moves~\cite{Pachner78} are the 1-4, 2-3, 3-2, and 4-1 moves.

The \emph{1-4 move} can be applied to any vertex of $\calF$ and creates a new material region.
See \reffig{1-4}.
The reverse move is the 4-1 move. 
It can be performed on a material region whose boundary in $\calF$ is combinatorially identical to the boundary of a three-simplex.

\begin{figure}[htbp]
\subfloat[Before.]{
\includegraphics[height=4.cm]{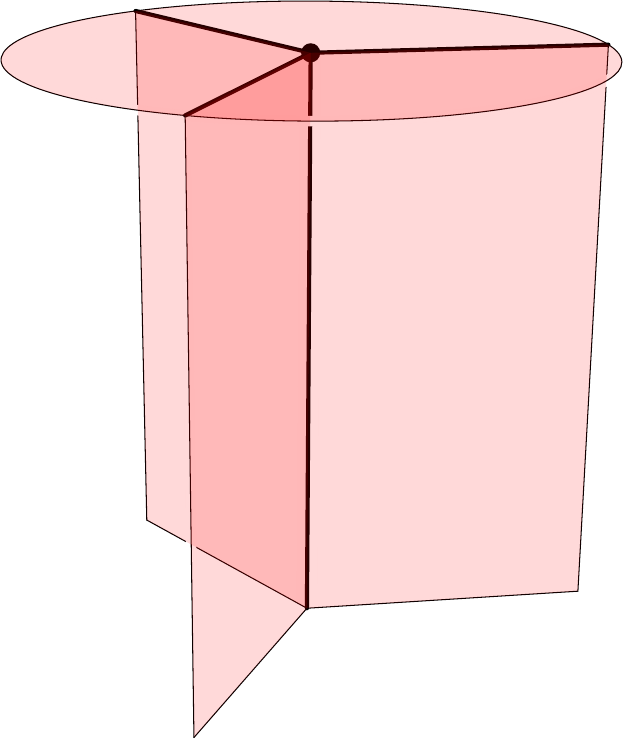}
\label{Fig:BubbleToDo1-4Step0}
}
\quad
\subfloat[After.]{
\includegraphics[height=4.cm]{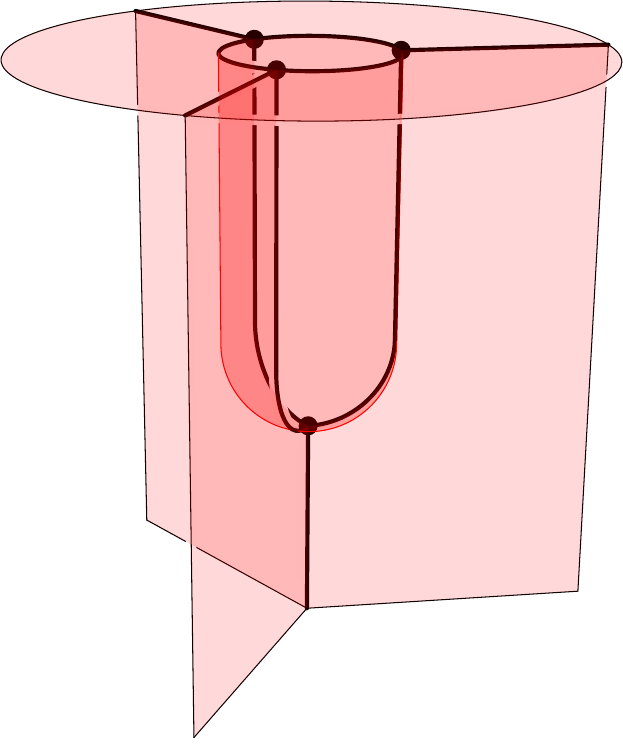}
\label{Fig:BubbleToDo1-4Step2}
}
\caption{A 1-4 move.}
\label{Fig:1-4}
\end{figure}

\begin{figure}[htbp]
\subfloat[]{
\includegraphics[height = 4.cm]{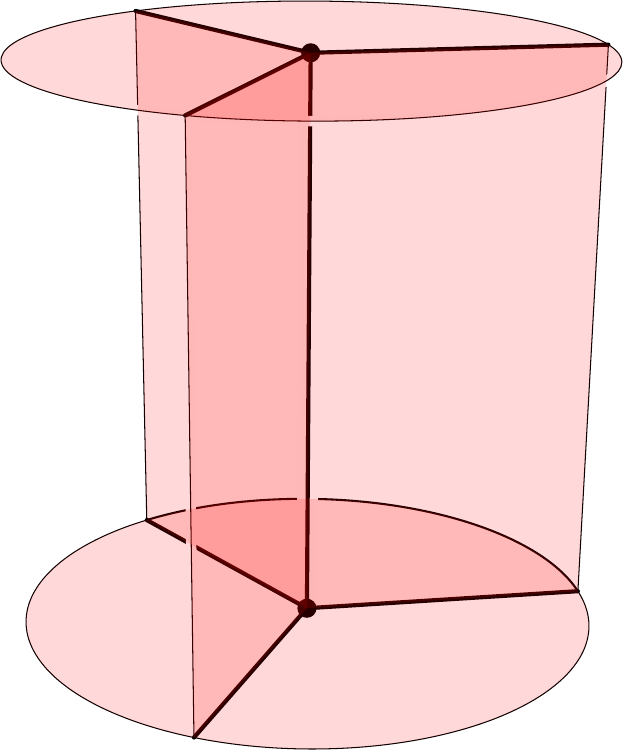}
\label{Fig:2-3Before}
}
\qquad
\subfloat[]{
\includegraphics[height = 4.cm]{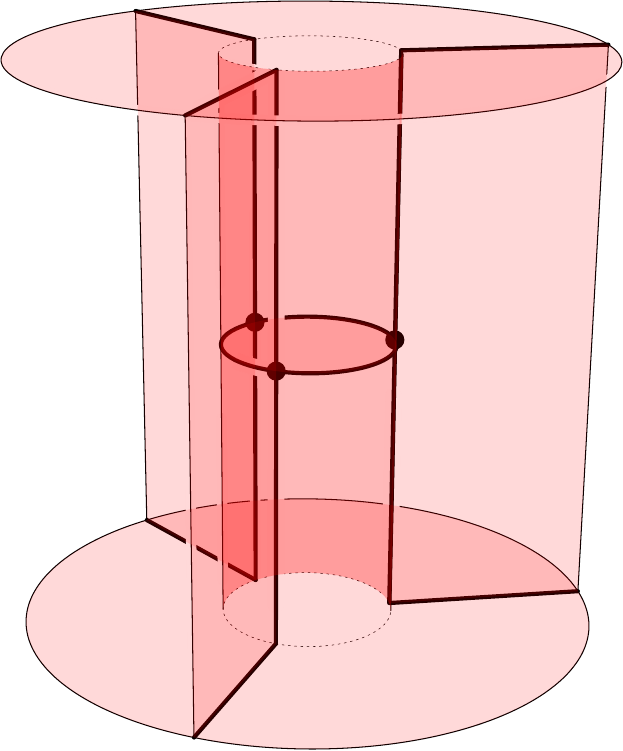}
\label{Fig:2-3After}
}
\caption{The 2-3 move.}
\label{Fig:2-3}
\end{figure}

The \emph{2-3 move} (called the \emph{T move} by Matveev~\cite[page~14]{Matveev07}) can be performed along any edge of $\calF$ that is not an edge loop. 
See \reffig{2-3}.
The reverse move is the 3-2 move. 
It can be performed on a triangular face whose closure is embedded in $\calF$.

\subsubsection{The 0-2 move}

The \emph{0-2 move} (called the \emph{ambient lune} move by Matveev~\cite[page~17]{Matveev07}) is applied along any path $\delta$ embedded in a face of $\calF$. 
It creates two new vertices and a new bigon face.
See \reffig{0-2}.
The reverse move is the 2-0 move. 
The 2-0 move produces a foam in $M$ unless the disks labelled $f$ and $f'$ in \reffig{0-2After} are part of the same face, or if the foam has only two vertices.

\begin{figure}[htbp]
\subfloat[]{
\labellist
\hair 2pt \small
\pinlabel $\delta$ [r] at 144 200
\endlabellist
\includegraphics[height = 4cm]{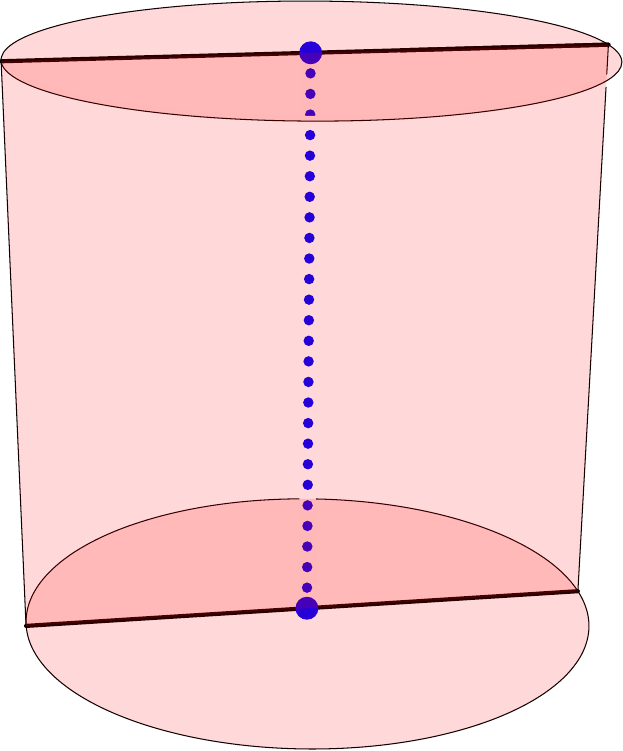}
\label{Fig:0-2Before}
}
\qquad
\subfloat[]{
\labellist
\hair 2pt \small
\pinlabel $f$ at 58 195
\pinlabel $f'$ at 242 200
\endlabellist
\includegraphics[height = 4cm]{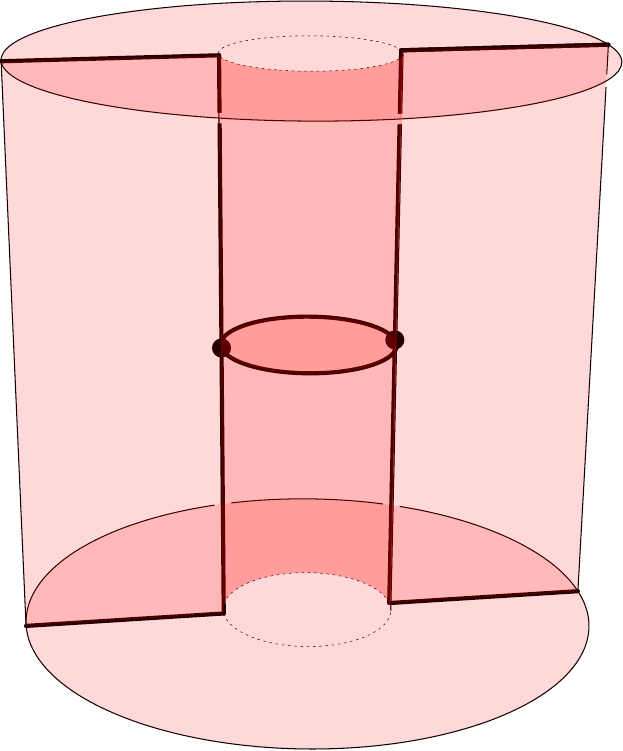}
\label{Fig:0-2After}
}
\caption{The 0-2 move. The vertical dotted arc $\delta$ in \reffig{0-2Before} indicates the path along which the 0-2 move acts.}
\label{Fig:0-2}
\end{figure}

\subsubsection{The bubble move}

The \emph{bubble move} is applied on any edge of a foam $\calF$ and creates a new material region. 
(Matveev~\cite[Section~1.2.3]{Matveev07} defines three variants of the bubble move. 
Our bubble move is equivalent to his case involving a ``triple line of the simple subpolyhedron''.)
We choose a point on the edge and a small neighbourhood $\calN$ around that point.
The result of applying the bubble move is the foam $(\calF - \calN) \cup \bdy \calN$.
See \reffig{Bubble}. 
The reverse bubble move can be performed on a material region whose boundary in $\calF$ consists of three bigons.

\begin{figure}[htbp]
\subfloat[]{
\includegraphics[height = 4.cm]{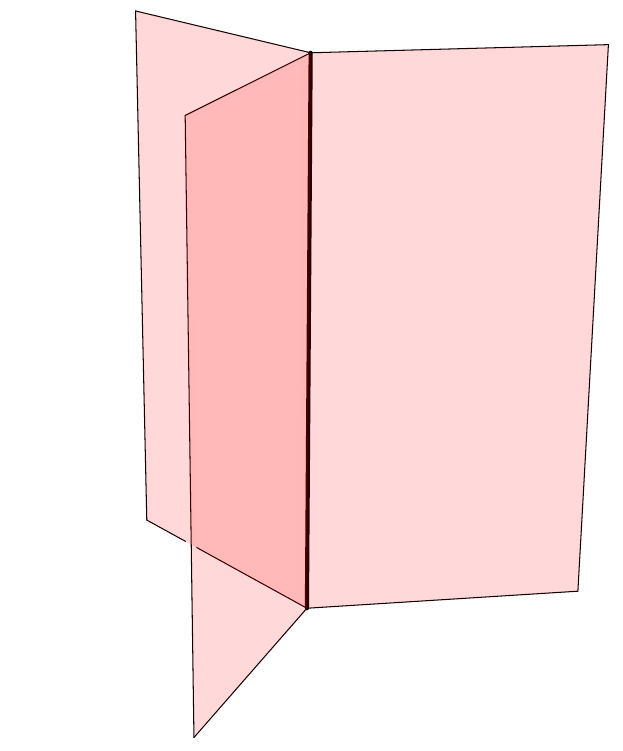}
\label{Fig:BubbleBefore}
}
\qquad
\subfloat[]{
\includegraphics[height = 4.cm]{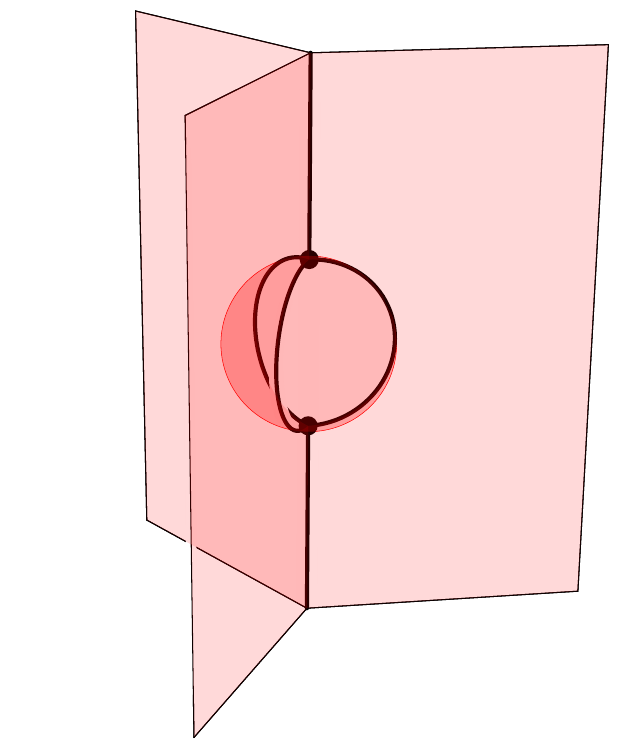}
\label{Fig:BubbleAfter}
}
\caption{The bubble move.}
\label{Fig:Bubble}
\end{figure}

\begin{lemma}
\label{Lem:1-4}
A 1-4 move may be performed by a bubble move followed by a 2-3 move.
\end{lemma}

\begin{proof}
We first perform a bubble move on one of the edges incident to the vertex, taking us from \reffig{BubbleToDo1-4Step0} to \reffig{BubbleToDo1-4}. This produces a new material region and two new vertices. We then perform a 2-3 move on the top and middle vertices to produce \reffig{BubbleToDo1-4Step2}.
\end{proof}

\begin{figure}[htbp]
\includegraphics[height=4.cm]{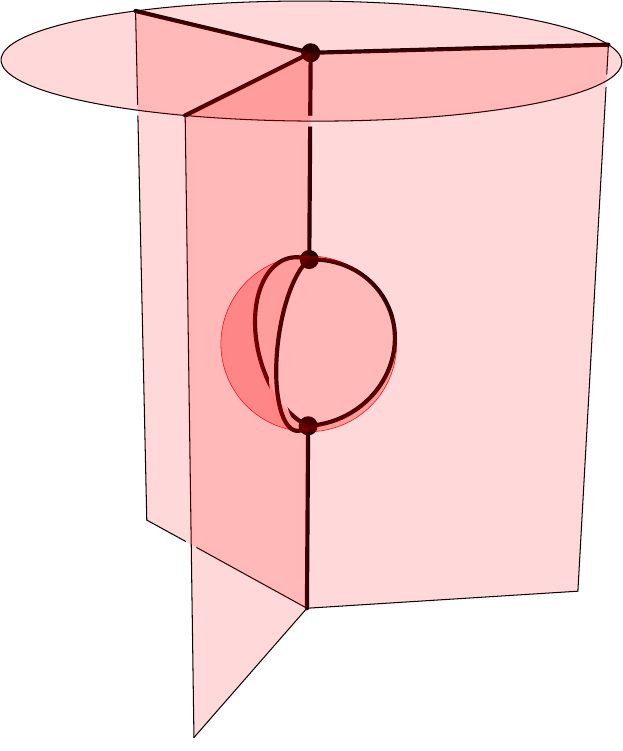}
\caption{A 1-4 move can be performed by doing a bubble move followed by a 2-3 move. Here we show the intermediate stage.}
\label{Fig:BubbleToDo1-4}
\end{figure}

\begin{definition}
\label{Def:Ancestor}
Suppose that $\calF$ is a foam in $M$.
Suppose that $e$ is an edge of $\calF$.
Suppose that $\calN(e)$ is a small regular neighbourhood of $e$.
Suppose that $\calF'$ is the result of applying a 2-3 move to $\calF$ along $e$, supported in $\calN(e)$.
Suppose that $c$ and $c'$ are open cells or open complementary regions of $\calF$ and $\calF'$ respectively.
Suppose there is a point $x$ lying in $c \cap c' - \calN(e)$.
Then we say that $c$ the \emph{ancestor} of $c'$ and $c'$ is the \emph{descendant} of $c$.

We make similar definitions for the various other moves (3-2, 1-4, 4-1, 0-2, 2-0, bubble, reverse bubble).
Finally we make the relation transitive through multiple moves.
\end{definition}

\begin{remark}
We will frequently abuse notation and use the same name for a cell (or complementary region) and its descendants.
\end{remark}

\section{Building $L$--essential triangulations}
\label{Sec:BuildLEssential}

The goal of this section is to prove the following.

\begin{theorem}
\label{Thm:ExistsTriangulation}
Suppose that $M$ is a compact, connected three-manifold with boundary.
Suppose that $L$ is a labelling of $\Delta_M$ with infinite image.
Then there is an $L$--essential ideal triangulation of $M$.
\end{theorem}

See Remarks~\ref{Rem:IdealEssential} and~\ref{Rem:MaterialEssential} for the consequences of this theorem for essential triangulations in the sense of \cite{HodgsonRubinsteinSegermanTillmann15, LuoTillmannYang13}.

\begin{remark}
Note that we do not produce \emph{strongly essential} triangulations, as given in \cite[Definitions~3.2 and~3.5]{HodgsonRubinsteinSegermanTillmann15}.
In particular, our main tool (\refdef{InflatingSnakePath}) prevents strong essentiality.
\end{remark}

\subsection{Recovering representations}
\label{Sec:AppHypGeom}

Thurston~\cite[Chapter~4]{Thurston80} was the first to generate representations from ideal triangulations.
Weeks' heuristic algorithm, implemented in SnapPy~\cite{SnapPy24}, finds triangulations that are not only essential, but also geometric. 
SnapPy is also very successful at finding solutions for representations coming from (large) hyperbolic Dehn filling.
However, not every representation can be recovered from every essential triangulation.
There may even be entire components of the representation variety that are ``missed'' by certain essential triangulations. 
See~\cite[Sections~4 and~11.3]{Segerman12}.
In this section, we show that under mild conditions on $\rho$ there exists an ideal triangulation of $M$ for which a solution to Thurston's gluing equations recovers $\rho$.

Previous work by the third author~\cite{Segerman12} and Goerner-Zickert~\cite{GoernerZickert18} used solutions to alternative sets of equations to generate representations. 

\begin{definition}
\label{Def:InfinitelyAnchorable}
Suppose that $\rho \from \pi_1(M) \to \PSL(2,\CC)$ is an anchorable representation.
Let $\Gamma = \rho(\pi_1(M))$.
Suppose that for some $c$ in $\Delta_M$ there is some $z \in \Fix_\rho(c)$ such that the orbit $\Gamma \cdot z$ is infinite.
Then we say that $\rho$ is \emph{infinitely anchorable} (at $c$).
\end{definition}

Suppose that $M$ is a three-manifold with toroidal boundary whose interior admits a finite volume hyperbolic metric.
Then any discrete and faithful representation is infinitely anchorable.

As a simple non-example, suppose that $M$ is the lens space $L(p, q)$, minus a ball.
There is a representation $\rho$ mapping the generator to an elliptic element of order $p$.
The boundary $\bdy M$ is a single sphere. 
Thus $\Delta_M$ is a collection of $p$ two-spheres.
For each $c$ in $\Delta_M$, the set $\Fix_\rho(c)$ is all of $\bdy_\infty \HH^3$.
Thus $\rho$ is anchorable, but not infinitely anchorable.

\begin{corollary}
\label{Cor:RepTriangulation}
Suppose that $M$ is a compact, connected, oriented three-manifold with non-empty boundary.
Suppose that $\rho \from \pi_1(M) \to \PSL(2,\CC)$ is infinitely anchorable.
Then there is an ideal triangulation $\calT$ of $M$ satisfying both of the following equivalent properties.
\begin{itemize}
\item There is an anchoring $L$ of $\rho$ such that $\calT$ is $L$--essential.
\item There is a solution $Z$ to Thurston gluing equations for $\calT$ so that $\rho_Z$ is conjugate to $\rho$. 
\end{itemize}
\end{corollary}


Combined with \refrem{MaterialEssential}, this corollary gives the triangulation required for, and recovers the first consequence of, Theorem~1.1 in~\cite{LuoTillmannYang13}.

\begin{proof}[Proof of \refcor{RepTriangulation}]
By hypothesis $\Delta_M$ is non-empty.
Since $\rho$ is anchorable, it admits an anchoring $L$ by \reflem{Anchoring}.
Since $\rho$ is infinitely anchorable, we may arrange matters so that the image of $L$ is infinite.
The first conclusion follows from \refthm{ExistsTriangulation}.
The second follows from \reflem{RecoverRho}.
\end{proof}

\begin{corollary}
\label{Cor:RepFilling}
Suppose that $M$ is an oriented, compact, connected three-manifold with non-empty toroidal boundary.
Suppose that $S$ is a collection of filling slopes (some of which may be $\infty$, meaning that we do not fill).
Suppose that $M(S)$, the resulting Dehn filled manifold, is hyperbolic with canonical representation $\rho_S$. 
Let $\iota \from M \to M(S)$ be the resulting inclusion.
Let $\rho = \iota^*(\rho_S)$.
Then there is an ideal triangulation $\calT$ of $M$ and a solution $Z$ of Thurston's gluing equations for $\calT$ so that $\rho_Z$ is conjugate to $\rho$. 
\end{corollary}

\begin{proof}
Suppose that $T$ is a boundary component of $M$.
Thus $T$ is a torus.
Let $c \in \Delta_M$ be an elevation of $T$.
Then $\Stab(c)$ is conjugate to $\pi_1(T)$.
Thus $\rho(\Stab(c))$ is isomorphic to $\ZZ^2$, $\ZZ$, or the trivial group. 
(No torsion can appear because $M(S)$ is a manifold rather than an orbifold.)
Therefore $\Fix_\rho(c)$ is either a point, a pair of points, or all of $\bdy_\infty \HH^3$.
Thus $\rho$ is anchorable.
Let $\Gamma = \rho(\pi_1(M))$.
Since $M(S)$ is hyperbolic, for any $z_0 \in \bdy_\infty \HH^3$, the orbit $\Gamma \cdot z_0$ is infinite.
Thus $\rho$ is infinitely anchorable.
Therefore we may apply \refcor{RepTriangulation}.
\end{proof}

In \refapp{HMP} we refine this result to provide a triangulation needed for an application to the 1-loop invariant due to Pandey and Wong~\cite{PandeyWong24}.

\subsection{Building triangulations avoiding certain edges}

One simple possibility for the set of labels is $\pi_1(M)$ itself.
In this case equivariance requires that $\bdy M$ is a single two-sphere.
The next most obvious set of labels comes from the coset space $\pi_1(M) / G$ where $G$ is a normal subgroup.

\begin{corollary}
\label{Cor:SubgroupAvoiding}
Suppose that $M$ is a closed, connected three-manifold. 
Suppose that $p$ is a point in $M$. 
Suppose that $G$ is an infinite index and normal subgroup of $\pi_1(M, p)$. 
Then there exists a one-vertex material triangulation $\calT$ of $M$ with vertex at $p$ so that no edge of $\calT$ lies in $G$.
\end{corollary}

\begin{proof}
As in \refrem{MaterialEssential}, suppose that $B$ is a small three-ball neighbourhood of $p$ in $M$.
Let $N = M - \interior(B)$.
Note that $\pi_1(N) \isom \pi_1(M)$.
Let $N_G$ be the cover of $N$ corresponding to $G$. 
Let $L$ be the labelling of $\Delta_N$ induced by the covering map from $\cover{N}$ to $N_G$, as in \refexa{CovEssential}. 
This labelling $L$ has infinite image because $G$ has infinite index. 

Apply \refthm{ExistsTriangulation} to obtain an $L$-essential ideal triangulation $\calT_N$ of $N$. 
Let $\calT$ denote the induced one-vertex triangulation of $M$. 
Suppose, to prove the contrapositive, that an edge of $\calT$ lies in $G$.
Then the corresponding edge in $\calT_N$ lifts to an edge in $N_G$ with both endpoints on the same boundary component of $N_G$. 
This contradicts the $L$-essentiality of $\calT_N$.
\end{proof}

Applying \refcor{SubgroupAvoiding} with $G$ equal to the commutator subgroup of $\pi_1(M)$ gives the following.

\begin{corollary}
\label{Cor:NullHomologous}
Suppose that $M$ is a closed connected three-manifold with infinite $H_1(M)$. 
Then there exists a one-vertex triangulation of $M$ none of whose edges are null-homologous. \qed
\end{corollary}

\begin{corollary}
\label{Cor:Seifert}
Suppose that $M$ is a closed connected orientable Seifert fibred space whose fundamental group is not virtually cyclic.
Then there exists a one-vertex triangulation of $M$ none of whose edges are homotopic to a regular fibre of $M$.
\end{corollary}

\begin{proof}
Let $G$ be the subgroup of $\pi_1(M)$ generated by a regular fibre. 
Thus $G$ is normal.
Since $\pi_1(M)$ is not virtually cyclic, the index of $G$ in $\pi_1(M)$ is infinite.
The result follows from \refcor{SubgroupAvoiding}.
\end{proof}

Burton and He~\cite[Section~5.3]{BurtonHe23} give a concrete example of an $11$-tetrahedra, one-vertex triangulation of a Seifert fibred space with no Seifert fibre edges.
They find this via a targeted search through the graph of one-vertex triangulations.

\subsection{Snake paths}
\label{Sec:SnakePaths}

The main tool that we will use to prove \refthm{ExistsTriangulation} is a \emph{snake}.
We describe this in the following sections.
Suppose that $M$ is a three-manifold.
Suppose that $\calT$ is a triangulation of $M$.
We take $\calF$ to be the dual foam.

\begin{definition}
\label{Def:SnakePath}
Suppose that $\gamma$ is a non-trivial oriented path embedded in $\calF$.
Suppose that $\gamma$ is transverse to the one-skeleton $\calF^{(1)}$ except at its endpoints which lie in the interior of edges of $\calF^{(1)}$.
See \reffig{InflateSnakePath1}.
Suppose that $c$ and $d$ are the initial and terminal endpoints of $\gamma$.
Three regions of $M - \calF$ meet $c$.
Exactly one of these, say $C$, is locally disjoint from the interior of $\gamma$.
We call $C$ the \emph{source} of $\gamma$.
We define the \emph{target} $D$ of $\gamma$ similarly, using $d$.
We require that $C$ is peripheral and that $D$ is material.
When all these properties hold we call $\gamma$ a \emph{snake path}.
\end{definition}

\begin{definition}
\label{Def:InflatingSnakePath}
Suppose that $\calF$ is a foam.
Suppose that $\gamma$ is a snake path in $\calF$.
We obtain a new \emph{inflated} foam $\calF(\gamma)$ by \emph{inflating} along the snake path as follows.
Let $\eta = \eta(\gamma)$ be a sufficiently small regular neighbourhood, chosen so that $\eta$ does not contain any vertices of $\calF$ and has a product structure away from the endpoints of $\gamma$.
Let $B$ be the collection of bigons of $\bdy \eta - \calF$.
(Note that each bigon arises from the initial or terminal point of $\gamma$.)
We set $\calF(\gamma) = (\calF - \eta) \cup (\bdy \eta - B)$.
See \reffig{InflateSnakePath}.
We call $\eta$ a \emph{snake}.
We extend the notions of ancestor and descendant (given in \refdef{Ancestor}) to inflating along snake paths.
\end{definition}

\begin{figure}[htbp]
\subfloat[$\gamma$]{
\includegraphics[width = 0.45\textwidth]{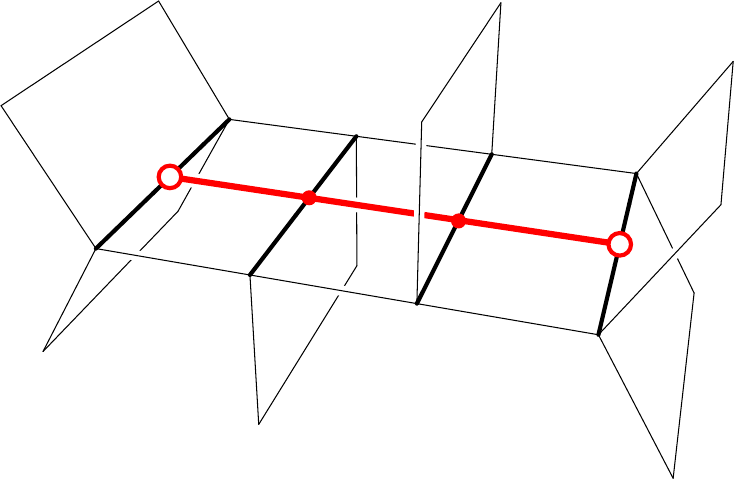}
\label{Fig:InflateSnakePath1}
}
\quad
\subfloat[$N(\gamma)$]{
\includegraphics[width = 0.45\textwidth]{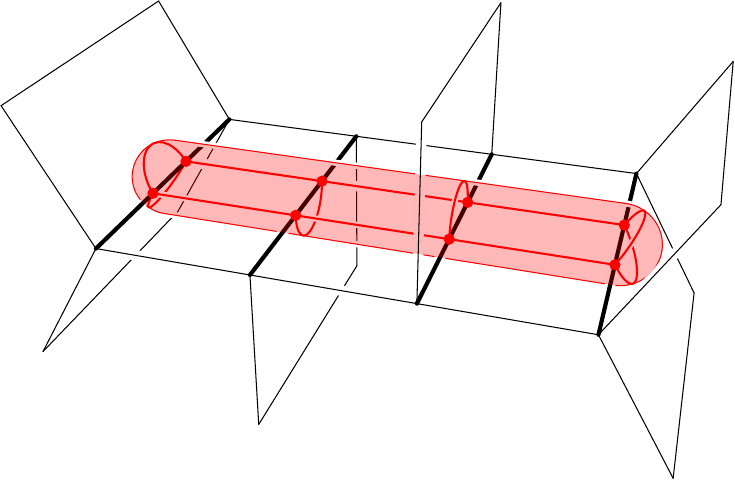}
\label{Fig:InflateSnakePath2}
}

\subfloat[$\calF(\gamma)$]{
\includegraphics[width = 0.45\textwidth]{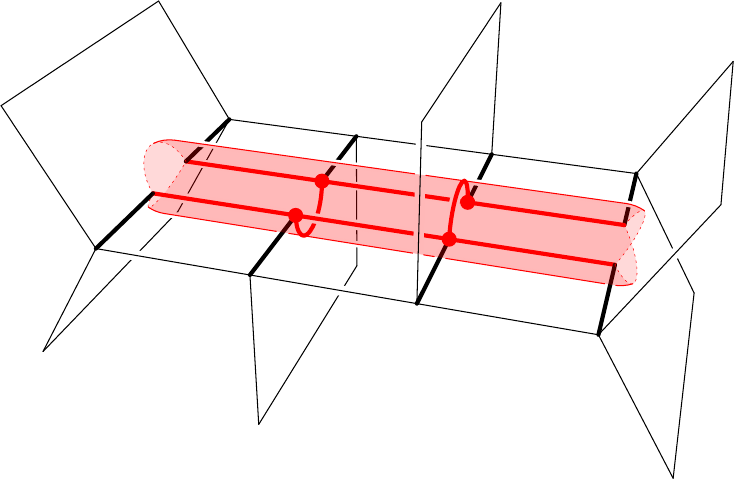}
\label{Fig:InflateSnakePath3}
}
\quad
\subfloat[$\calF(\gamma)$ may not be a foam.]{
\labellist
\hair 2pt \small
\pinlabel $f$ at 60 173
\pinlabel $f$ at 330 150
\endlabellist
\includegraphics[width = 0.45\textwidth]{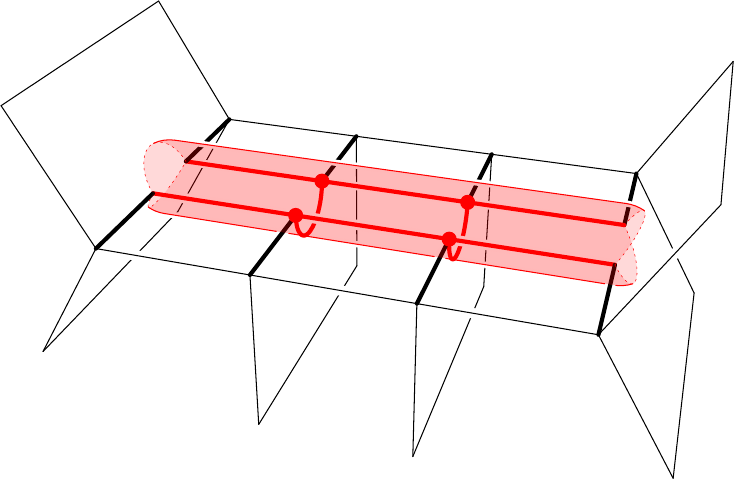}
\label{Fig:InflateSnakePathBad}
}
\caption{Inflating along a snake path $\gamma$. The result is shown in \reffig{InflateSnakePath3}.}
\label{Fig:InflateSnakePath}
\end{figure}

\begin{lemma}
Suppose that $M$ is a manifold with a foam $\calF$.
Suppose that $\gamma$ is a snake path in $\calF$.
Then the inflation $\calF(\gamma)$ is a foam in $M$.
\end{lemma}

\begin{proof}
Let $C$ and $D$ be the source and target regions for $\gamma$.
Since the inflation move occurs inside of a three-ball, all local properties of a foam are maintained.
Thus it suffices to check two properties.
First, \refdef{Foam}\refitm{Complement} holds because it holds for $\calF$ and the only change is that $C$ and $D$ are now combined into a single region.
Second, we check that $\calF(\gamma)$ is a CW complex:
that is, we check that each component of the $(k+1)$--skeleton minus the $k$--skeleton is a cell.

Faces of $\calF(\gamma)$ are either unchanged from $\calF$, or they have small parts removed, or they extend along the snake, or they are new faces running along the snake.
The only possible non-disk face $f$ would be an annulus of the form shown in \reffig{InflateSnakePathBad}.
There the snake path $\gamma$ lies entirely within the boundary of a region $E$ say.
Consulting the figure, we see that $C$ and $E$ are on opposite sides of $f$, and also that $D$ and $E$ are on opposite sides of $f$.
We deduce that $E$ equals $D$ and also that $E$ equals $C$.
However, $C$ is a peripheral region and $D$ is a material region.
Thus we reach a contradiction.

By \refrem{FoamConnected}, we have that $\calF$ is connected. 
Since inflation occurs in a small neighbourhood of an embedded arc, we deduce that $\calF(\gamma)$ is connected.
Since all faces are disks, we have that the one-skeleton of $\calF(\gamma)$ is connected.
If some component $\delta$ of the one-skeleton of $\calF(\gamma)$ is a circle without vertices then $\calF(\gamma)$ has no vertices.
Therefore $\calF$ had no vertices.
But $\calF$ is a foam and therefore a CW-complex, so it has vertices.
Thus we reach a contradiction.
\end{proof}

\subsection{Labels and snake paths}

In this section we discuss how snake paths, snakes, and labels interact.

\begin{definition}
\label{Def:LabelledSnakePath}
Suppose that $L \from \Delta_M \to \calL$ is a labelling.
Suppose that $\gamma$ is a snake path in $\calF$.
In \refdef{FoamLabellingEssential} we extended $L$ to be defined on peripheral regions of $\cover{M} - \cover\calF$.
Given $\gamma$, we further extend $L$ to a labelling scheme $L_\gamma$ as follows.

Let $\cover{\gamma}$ be any lift of $\gamma$.
Let $C$ be the source and $D$ be the target of $\cover{\gamma}$.
Set $L_\gamma(\cover{\gamma}) = L_\gamma(D) = L(C)$.
When $\gamma$ is clear from context we drop the subscript.
\end{definition}

In the classic video game~\cite[page~101]{Goggin10} a snake may never touch its own tail. 
We require a similar constraint on snake paths: a snake path $\cover{\gamma}$ cannot touch itself, its source, its target, or $\Stab(L(\cover{\gamma}))$--translates of these.

\begin{definition}
\label{Def:LSelfAvoiding}
Suppose that $L \from \Delta_M \to \calL$ is a labelling.
Suppose that $\gamma$ is a snake path in $\calF$.
Suppose that $\cover\gamma$ is a lift of $\gamma$ with source $C$ and target $D$.
Suppose that the label $L(\cover\gamma)$ is distinct from the labels of all 
regions meeting either $D$ or the interior of $\cover{\gamma}$.
Suppose that for each $\tau \in \Stab(L(\cover{\gamma}))$, there is no face with $D$ on one side and $\tau(D)$ on the other.
Then we say that $\gamma$ is \emph{$L$--self-avoiding}.
\end{definition}

The above is well-defined because $L$ is equivariant.

\begin{remark}
\label{Rem:StronglyLSelfAvoiding}
The condition in \refdef{LSelfAvoiding} involving $\Stab(L(\cover{\gamma}))$ is automatically satisfied if $\calF$ is $L$--essential (rather than only weakly $L$--essential).
\end{remark}

\begin{lemma}
\label{Lem:EssentialSnake}
Suppose that $\calF$ is a (weakly) $L$--essential foam.
Suppose that $\gamma$ is an $L$--self-avoiding snake path in $\calF$.
Then $\calF(\gamma)$ is (weakly) $L$--essential. 
\end{lemma}

\begin{proof}
Suppose that $\cover{\gamma}$ is a lift of $\gamma$ with source $C$ and target $D$.
By \refdef{LabelledSnakePath} we have that $L(C) = L(\cover{\gamma}) = L(D)$.
Suppose that $f$ is a face of $\calF(\gamma)$.
There are two cases as $f$ does or does not have an ancestor in $\calF$.

Recall that $\phi \from \cover{M} \to M$ is the covering map.
Suppose that $f$ has an ancestor $f'$ in $\calF$.
There are two subcases.
Either $f'$ lies on the boundary of $\phi(D)$ or it does not.
\begin{itemize}
\item Suppose that $f'$ lies on the boundary of $\phi(D)$.
There are two sub-subcases.
Either $f'$ meets $\phi(D)$ on only one side or on both sides.
Suppose first that $f'$ meets $\phi(D)$ on only one side.
Let $\cover{f}'$ be the unique lift of $f'$ contained in the boundary of $D$.
The two regions adjacent to $\cover{f}'$ are $D$ and $E$ say, which are distinct.
If $E$ is peripheral then, since $\gamma$ is $L$--self-avoiding, $L(E) \neq L(C)$.
If instead $E$ is material then, since $\phi(E) \neq \phi(D)$, the region $E$ has no label.
Thus $f$ is $L$--essential.

Now suppose that $f'$ meets $\phi(D)$ on both sides.
Let $\cover{f}'$ be a lift of $f'$ contained in the boundary of $D$.
The region on the other side of $\cover{f}'$ is a translate of $D$, say $\tau(D)$.
Since $\gamma$ is $L$--self-avoiding, we have that $\tau \notin \Stab(L(\cover{\gamma}))$.
Thus the regions either side of the corresponding face $\cover{f}$ have labels $L(\cover{\gamma})$ and $\tau(L(\cover{\gamma}))$, which are different.
Thus $f$ is $L$--essential.

\item Suppose that $f'$ does not lie on the boundary of $\phi(D)$.
Here the two regions adjacent to any lift of $f$ are the same as the two regions adjacent to the corresponding lift of $f'$. 
Thus $f$ is $L$--essential since $f'$ is.
\end{itemize}

Finally, suppose that $f$ does not have an ancestor in $\calF$.
Then $f$ is contained the boundary of the snake $\eta(\gamma)$.
Let $\cover{f}$ be the lift of $f$ contained in the boundary of $\eta(\cover{\gamma})$.

The region on one side of $\cover{f}$ is $\eta(\cover{\gamma})$ so has label $L(\cover{\gamma})$.
Since $\gamma$ is $L$--self-avoiding, the region on the other side of $\cover{f}$ does not have label $L(\cover{\gamma})$. 
(Note that the other side could be a translate of $D$. 
\refdef{LabelledSnakePath} gives this region a label as well, which by $L$--self avoidance is different from $L(\cover{\gamma})$.)
Thus, again, $f$ is $L$--essential.

Finally, note that inflation does not lead to any material regions coming into contact with each other, so if $\calF$ was $L$--essential then $\calF(\gamma)$ is also $L$--essential. 
\end{proof}

\subsection{Creating snake paths}

In a slight abuse of notation, if $\calA$ is a collection of subspaces of a space then we may write $\calA = \cup_{a \in \calA}{a}$.
We now describe how to produce snake paths.

\begin{algorithm}[Create snake path]
\label{Alg:CreateSnakePath} \,\\  
\noindent
\textbf{Input:} A manifold $M$ equipped with a foam $\calF$ and a labelling function $L \from \Delta_M \to \calL$ together with the following:
\begin{itemize}
\item A material region $D$ of $\cover{M} - \cover{\calF}$.
\item A label $\ell \in L(\Delta_M)$ so that
\begin{itemize}
\item $\ell$ is distinct from the labels of all regions incident to $D$ and
\item for each $\tau \in \Stab(\ell)$, there is no face of $\cover{\calF}$ with $D$ on one side and $\tau(D)$ on the other.
\end{itemize}
\item A possibly empty collection $\calA$ of either
    \begin{itemize}
    \item disjoint snake paths in $\calF$, or 
    \item faces and edges of $\calF$ such that 
        \begin{itemize}
        \item $\cover{\calF} - \phi^{-1}(\closure{\calA})$ is connected,  
        \item some edge of $\phi(D)$ not in $\calA$ is adjacent to a face not in $\calA$, and 
        \item there is a peripheral region $C'$ so that $L(C')=\ell$ and some edge of $\phi(C')$ not in $\calA$ is adjacent to a face not in $\calA$. 
        \end{itemize}
    \end{itemize}
\end{itemize}
\textbf{Output:} An $L$--self-avoiding snake path $\gamma$ in $\calF$ with the following properties.
\begin{itemize}
\item Some lift $\cover{\gamma}$ of $\gamma$ has source $C$ say, with $L(C) = \ell$, and target $D$.
\item If $\calA$ is non-empty then the snake path $\gamma$ is disjoint from $\calA$.
\end{itemize}
	
Let $e$ be an edge of $\cover{\calF}$ which is adjacent to a region with label $\ell$.
Let $d$ be an edge of $\cover{\calF}$ which is adjacent to $D$.
If $\calA$ contains edges or faces then we additionally assume that $\phi(e)$ and $\phi(d)$ are not in $\calA$ and are adjacent to faces not in $\calA$. 
In all cases we have that $\cover{\calF} - \phi^{-1}(\closure{\calA})$ is connected.
Thus we may choose an (oriented) path $\cover{\gamma}'$ starting on $e$, ending on $d$, transverse to ${\cover{\calF}}^{(1)}$, and whose interior lies in $\cover{\calF} - \phi^{-1}(\closure{\calA})$.
(Note that the endpoints of $\cover{\gamma}'$ may meet $\phi^{-1}(\closure{\calA})$.)
Let $\gamma' = \phi(\cover{\gamma}')$ be the image of $\cover{\gamma}'$ under the covering map.
We arrange matters so that $\gamma'$ is transverse to itself.

\begin{figure}[htbp]
\subfloat[]{
\labellist
\hair 2pt \small
\pinlabel $\delta$ [br] at 55 146
\pinlabel $\gamma'_{k,j}$ [tr] at 70 106
\pinlabel $\gamma'_{k,j+1}$ [b] at 165 160
\pinlabel $e$ [tl] at 170 100
\endlabellist
\includegraphics[width = 0.44\textwidth]{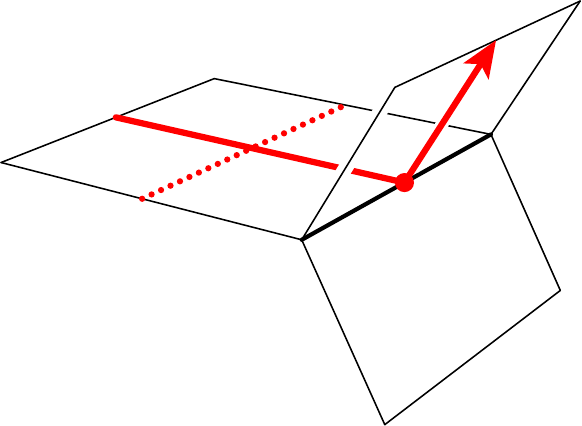}
\label{Fig:DodgeSnakePath1}
}
\quad
\subfloat[]{
\includegraphics[width = 0.44\textwidth]{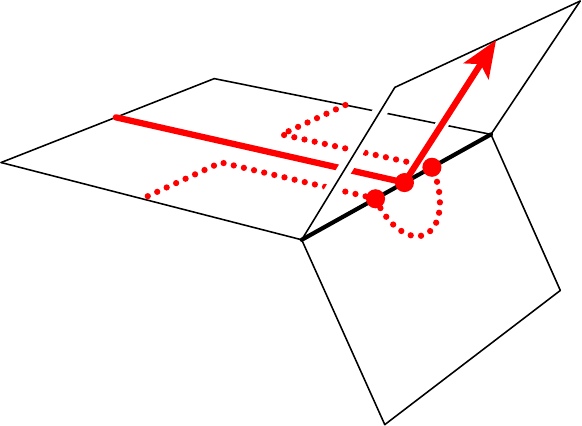}
\label{Fig:DodgeSnakePath2}
}
\caption{A \emph{detour} connects $\gamma'_{k,j}$ to $\gamma'_{k,j+1}$ while avoiding the cutting path $\delta$ as follows.  
We first remove a small regular neighbourhood of $\delta$ from each.
We then add an arc lying in the boundary of that regular neighbourhood.}
\label{Fig:DodgeSnakePath}
\end{figure}

We now construct a snake path $\gamma$ in $\calF$ which is homotopic (rel endpoints) to $\gamma'$.
We partition $\gamma'$ into segments $\{\gamma'_0, \ldots, \gamma'_N\}$ by cutting with the one-skeleton $\calF^{(1)}$.
We begin by taking $\gamma$ to be the empty path.
 Assume that we have processed $\gamma'_0, \ldots, \gamma'_{k-1}$.
 We then do the following.
\begin{enumerate}
\item 
\label{Itm:Min}
Isotope $\gamma'_k$ rel boundary in its face, remaining disjoint from $\closure{\calA}$, to meet $\gamma$ minimally. 
\item 
\label{Itm:Cut}
Partition $\gamma'_k$ into segments $\gamma'_{k,j}$ by cutting with $\gamma$.
\item Form $\gamma_k$ by connecting, for all $j$, the segment $\gamma'_{k,j}$ to $\gamma'_{k,j+1}$ using a detour around the cutting segment, $\delta$ say, as shown in \reffig{DodgeSnakePath}. 
The orientation on $\gamma$ induces an orientation on $\delta$.
We detour around the terminal point of $\delta$.
This ensures that the detour does not go around the initial point of $\gamma$, and therefore does not meet a  face of $\calA$.
(Note that the detour lies in a ball so no new self-intersections are introduced.)
\item Add $\gamma_k$ to the end of $\gamma$. 
\end{enumerate}
This completes the recursive part of the algorithm.
Let $\cover{\gamma}$ be the lift of $\gamma$ that starts at the initial point of $\cover{\gamma}'$ (and, by homotopy lifting, ends at the terminal point of $\cover{\gamma}'$).
Next, we truncate $\cover{\gamma}$ twice.
We remove a prefix of $\cover{\gamma}$ so that the remainder meets a region with label $\ell$ only at its initial point.
Thus the source for $\cover{\gamma}$, say $C$, has label $L(C) = \ell$.
Moreover, no peripheral region meeting the interior of $\cover{\gamma}$ has label $\ell$.
We remove a suffix of $\cover{\gamma}$ so that the remainder meets 
$\Stab(\ell) \cdot D$ only at its terminal point.
Finally, set $\gamma = \phi(\cover{\gamma})$ to be the image of $\cover{\gamma}$ under the covering map.
This completes \refalg{CreateSnakePath}. 
\end{algorithm}

The lift of $\gamma$ with target $D$ is a translate of $\cover{\gamma}$, but again its source has label $\ell$.
Moreover, no material region meeting the interior of $\cover{\gamma}$ has label $\ell$.
The label $\ell$ is distinct from all labels of peripheral regions meeting $D$ by hypothesis. 
Also by hypothesis, for each $\tau \in \Stab(\ell)$,  there is no face of $\cover{\calF}$ with $D$ on one side and $\tau(D)$ on the other.
Thus the resulting $\gamma$ is an $L$--self-avoiding snake path with the desired properties, proving the correctness of \refalg{CreateSnakePath}.

\begin{remark}
\label{Rem:StronglyCreateSnakePath}
The condition in \refalg{CreateSnakePath} involving $\Stab(\ell)$ is automatically satisfied if $\calF$ is $L$--essential (rather than only weakly $L$--essential).
\end{remark}

\begin{figure}[htbp]
\includegraphics[height = 5cm]{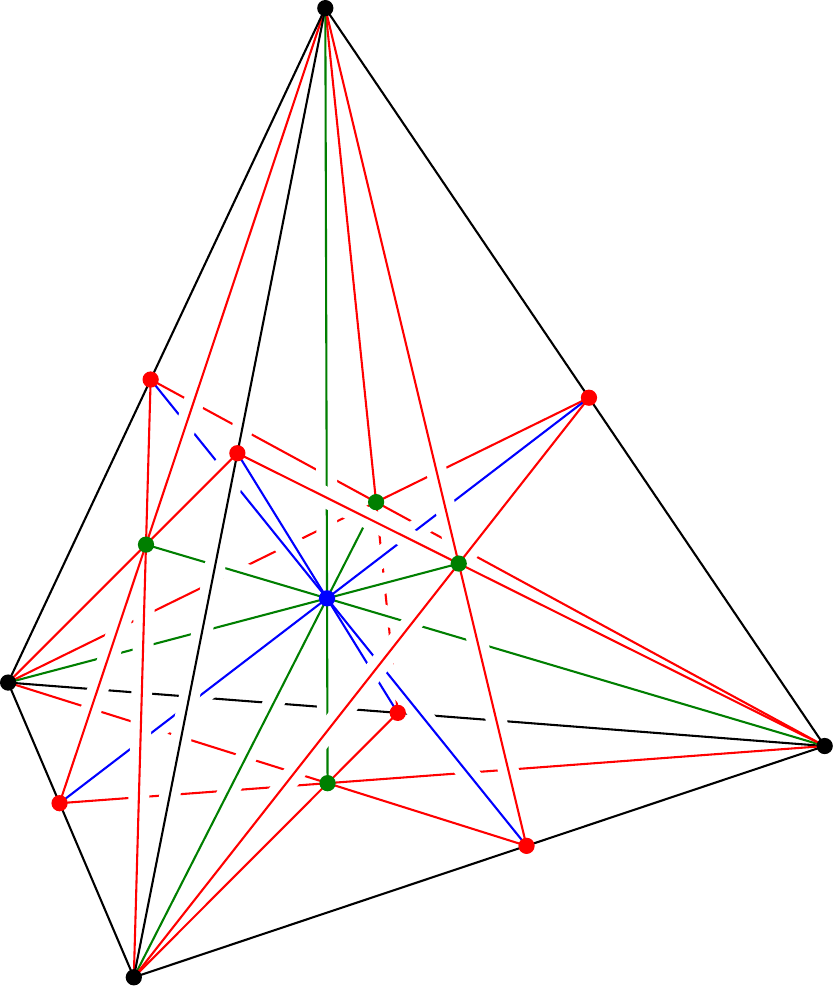}
\caption{Barycentric subdivision of a tetrahedron.}
\label{Fig:Barycentric}
\end{figure}

\begin{proof}[Proof of \refthm{ExistsTriangulation}]
Suppose that $\calS$ is an ideal triangulation of $M$.
Let $B\calS$ be the barycentric subdivision of $\calS$.
We claim that $B\calS$ is $L$--essential.
To see this, note that no pair of original vertices of $\calS$ share an edge in $B\calS$.
All new vertices are material so do not have labels, and there are no edge loops in $B\calS$.
See \reffig{Barycentric}.

We now produce a sequence of $L$--essential foams $\calF_0, \ldots, \calF_n$ where 
\begin{itemize}
\item $\calF_0$ is dual to $B\calS$, 
\item $\calF_{i+1}$ is related to $\calF_i$ by inflating a snake along an $L$--self-avoiding snake path $\gamma_i$, and 
\item $\calF_n$ has no material regions.
\end{itemize}

To produce the snake path $\gamma_i$ on $\calF_i$, we proceed as follows.
Choose a material region $D_i$ in $\cover{M} - \cover{\calF}_i$.
By hypothesis, $L(\Delta_M)$ is infinite. 
Therefore there is some label $\ell_i \in L(\Delta_M)$ that is distinct from the labels on all regions incident to $D_i$.
We can therefore apply \refalg{CreateSnakePath} with inputs $D_i$, $\ell_i$ (using \refrem{StronglyCreateSnakePath}), and $\calA = \emptyset$ to produce an $L$--self-avoiding snake path $\gamma_i$ that has a lift $\cover{\gamma}_i$ with $L(\cover{\gamma}_i) = \ell_i$ and target $D_i$.
By \reflem{EssentialSnake} we have that $\calF_{i+1} = \calF_i(\gamma_i)$ is $L$--essential.
The dual of $\calF_n$ is the desired $L$--essential triangulation.
\end{proof}

\section{Simplicial connectivity}
\label{Sec:SimplicialConnectivity}

The goal of this section is to use a result of Casali~\cite{Casali95} to prove \refcor{ConnectivitySimplicial}, which is a precursor to \refthm{ConnectivityWith0-2}.
Suppose that $M$ is a compact connected three-manifold with non-empty boundary. 
Suppose that $\calT$ is an $L$--essential ideal triangulation of $M$. 

\begin{definition}
Let $\sigma$ be an edge or a triangle of $\calT$. 
Suppose that each model tetrahedron in $\calT$ meets $\sigma$ at most once. 
Then the union of tetrahedra which contain $\sigma$ is a ball $B(\sigma)$, possibly with identifications along its boundary. 
We call this subcomplex the \emph{star} of $\sigma$.
\end{definition}

\begin{definition}
Suppose that $v$ is a vertex of $\calT$.
The \emph{link} of $v$ is the boundary of the star of $v$.
\end{definition}

\begin{definition}
A \emph{stellar move} along $\sigma$ retriangulates the interior of the star $B(\sigma)$ by replacing the existing triangulation with the following.
We choose a point in the interior of $B(\sigma)$ and cone the boundary $\bdy B(\sigma)$ to that point.
\end{definition}
 
A stellar move on $\calT$ can be realised by a sequence of bistellar moves (sometimes called Pachner moves). 
See~\cite[Section 2.4]{RubinsteinSegermanTillmann19}. 
That paper is not concerned with $L$--essential triangulations or simplicial triangulations.
However, if $\calT$ is $L$--essential, or is insulated simplicial, then the intermediate triangulations described in \cite{{RubinsteinSegermanTillmann19}} are $L$--essential, or are insulated simplicial, respectively. 
In the following section we recall the constructions and check that the two properties are preserved.

\subsection{Stellar subdivisions} 
\subsubsection{$\sigma$ is a face}
\label{Sec:StellarFace}
Suppose that the face $\sigma$ is a face of two distinct tetrahedra $t_1$ and $t_2$. 
A stellar move on $\sigma$ can be realised by a 1-4 move on $t_1$ followed by a 2-3 move along $\sigma$. 
Let $w$ be the new material vertex created by the 1-4 move. 
There are no edge loops based at $w$, and
the new edge created by the 2-3 move connects $w$ with the vertex of $t_2$ opposite $\sigma$.
Therefore an $L$--essential triangulation remains $L$--essential. 
Also all the new triangles and tetrahedra created contain $w$, so an insulated simplicial triangulation remains insulated simplicial after this stellar subdivision.

\subsubsection{$\sigma$ is an edge}
\label{Sec:StellarEdge}
Suppose that the edge $\sigma$ has degree $m$ and is an edge of $m$ distinct tetrahedra of $\calT$. 
If $\calT$ is $L$--essential then $m$ is at least two.
A stellar move on $\sigma$ can be realised as follows.
Perform a 1-4 move on one of the tetrahedra containing $\sigma$ to obtain a triangulation $\calT_0$ which has no $L$--inessential edges and no edge loops based at material vertices. 
Let $w$ be the new vertex introduced by the 1-4 move. 
The degree of $\sigma$ in $\calT_0$ is $m+1$. 
If the degree of $\sigma$ is greater than three then we perform a 2-3 move on a pair of adjacent tetrahedra both of which contain $\sigma$ and exactly one of which contains $w$. 
This reduces the degree of $\sigma$ by one and introduces a new edge between $w$ and one of the original vertices of $\calT$. 
The new triangulation $\calT_1$ obtained after this move has no $L$--inessential edges or edge loops based at material vertices.
The ball $B(\sigma)$ in $\calT_1$ has $m-1$ distinct tetrahedra. 
Repeat this process till the degree of $\sigma$ has been reduced to three. 
Finally perform a 3-2 move on $\calT_{m-2}$ to get rid of $\sigma$ and obtain its stellar subdivision.

If the triangulation $\calT$ is insulated simplicial then the initial degree $m$ is at least three, and the above moves go through, keeping all intermediate triangulations insulated simplicial.

\begin{lemma} \label{Lem:BoundaryConnectivity}
Suppose that $M$ is a compact connected three-manifold with non-empty boundary. 
Suppose that $\calT$ is an insulated simplicial ideal triangulation of $M$. 
Let $\Lambda$ be the (possibly disconnected) triangulated surface that is the link of the ideal vertex set of $\calT$. 
Let $\Lambda'$ be a simplicially triangulated surface that can be obtained from $\Lambda$ by a two-dimensional bistellar move. 
Then there is an insulated simplicial ideal triangulation $\calT'$ such that
\begin{itemize}
\item the link of the ideal vertex set of $\calT'$ is $\Lambda'$ and
\item there exists a sequence of insulated simplicial partially ideal triangulations 
connecting $\calT$ to $\calT'$
where two consecutive triangulations are related by a three-dimensional bistellar move.
\end{itemize}
\end{lemma}

\begin{proof}
Suppose that $\Lambda'$ is obtained from $\Lambda$ by a two-dimensional bistellar move on the link of an ideal vertex $v$ of $\calT$. 
Since $\calT$ is insulated simplicial, all vertices of $\Lambda$ are material.
There are three cases as the bistellar move is a 1-3, a 2-2, or a 3-1 move.
\begin{case}[1-3]
Suppose that $\Lambda'$ is obtained by a 1-3 move on a face $f$ of $\Lambda$. 
Let $t$ be the tetrahedron of $\calT$ containing $v$ and $f$. Let $\calT'$ be the triangulation obtained from $\calT$ by a 1-4 move on $t$. Then $\calT'$ is still insulated simplicial and $\Lambda'$ is the link of the ideal vertex set of $\calT'$ as required.
\end{case}

\begin{case}[2-2]
Suppose that $\Lambda'$ is obtained by a 2-2 move on a pair of adjacent faces $f_1$ and $f_2$ of $\Lambda$. Suppose that $e$ is the edge common to $f_1$ and $f_2$.
Suppose that $w_i$ is the vertex of $f_i$ which is not in $e$.  
Since $w_i$ is a vertex of $\Lambda$ it is material.
If there is an edge $e'$ in $\calT$ that connects $w_1$ and $w_2$ then since $\Lambda'$ is simplicial, the edge $e'$ does not lie in $\Lambda'$. 
We eliminate $e'$ by taking its stellar subdivision.
By \refsec{StellarEdge}, this can be realised by a sequence of bistellar moves passing through insulated simplicial triangulations.
These moves do not change $\Lambda$ because the interior of the star of $e'$ is disjoint from $\Lambda$. 
Once there are no edges that connect $w_1$ to $w_2$, the 2-2 move on $\Lambda$ can be realised by applying the 2-3 move on the pair of tetrahedra containing $f_1$, $f_2$, and $v$. 
The new edge connects the vertices $w_1$ and $w_2$.
Since these vertices are material the triangulation obtained after this 2-3 move is insulated simplicial as required.
\end{case}

\begin{case}[3-1]
Suppose that $\Lambda'$ is obtained by a 3-1 move on pairwise adjacent faces $f_1$, $f_2$ and $f_3$ of $\Lambda$. 
Let $w$ be the vertex common to $f_1$, $f_2$, and $f_3$. 
The edge joining $v$ and $w$ is of degree three. 
If there exists a face $f$ in $\calT$ whose boundary is $\partial (f_1 \cup f_2 \cup f_3)$ then since $\Lambda'$ is simplicial, $f$ does not lie in $\Lambda$. 
We eliminate $f$ by taking its stellar subdivision.
By \refsec{StellarFace}, this can be realised by a sequence of bistellar moves passing through insulated simplicial triangulations.
These moves do not change $\Lambda$ because the interior of the star of $f$ is disjoint from $\Lambda$.
Once there is no triangle whose boundary is $\partial (f_1 \cup f_2 \cup f_3)$, the 3-1 move on $\Lambda$ can be realised by applying the 3-2 move along the edge joining $v$ to $w$.
The triangulation obtained after this 3-2 move is insulated simplicial as required. \qedhere
\end{case}
\end{proof}

\begin{lemma}
\label{Lem:BarycentricSubdivision}
Suppose that $M$ is a compact connected three-manifold with non-empty boundary.
Suppose that $\calT$ is an $L$--essential triangulation of $M$.
Let $B\calT$ be the barycentric subdivision of $\calT$.
Then there is a sequence of $L$--essential triangulations 
connecting $\calT$ to $B\calT$,
where two consecutive triangulations are related by a 1-4, 2-3, 3-2, or a 4-1 move.
\end{lemma}

\begin{proof}
The construction appears in~\cite[Section 2.5]{RubinsteinSegermanTillmann19}. 
Again, that paper is not concerned with $L$--essential triangulations. 
However, if $\calT$ is $L$--essential then the intermediate triangulations described in~\cite{RubinsteinSegermanTillmann19} are also $L$--essential.
Here we recall the construction and check that the intermediate triangulations are $L$--essential. 

We first do a 1-4 move to each tetrahedron of $\calT$. 
This does not introduce any new edge loops. 
Every edge that is introduced by a 1-4 move connects an existing vertex to the material vertex that is created by the 1-4 move. 
Thus at every step the triangulation is $L$--essential.
Each face of this new triangulation lies in two distinct tetrahedra. 
So we can perform a stellar subdivision of all of the faces of $\calT$.
Using \refsec{StellarFace} we realise this via bistellar moves through $L$--essential triangulations.
Finally, each edge of $\calT$ in the new triangulation is contained in each model tetrahedron at most once. 
So we can perform a stellar subdivision of all of the edges of $\calT$.
Using \refsec{StellarEdge} we realise this via bistellar moves through $L$--essential triangulations.
This gives us the barycentric subdivision of $\calT$.
\end{proof}

\begin{theorem}
\label{Thm:ConnectivitySimplicial1-4}
Suppose that $M$ is a compact connected three-manifold with non-empty boundary.
Suppose that $\calT$ and $\calT'$ are $L$--essential ideal triangulations of $M$.
Then there is a sequence of $L$--essential (partially) ideal triangulations connecting $\calT$ to $\calT'$, where two consecutive triangulations are related by a 1-4, 2-3, 3-2, or a 4-1 move.
\end{theorem}

\begin{proof}
Let $V$ be the set of ideal vertices of $M$.  
Note that this gives a bijection between the vertices of $\calT$ and the vertices of $\calT'$. 
Let $B^2\calT$ and $B^2\calT'$ be the second barycentric subdivisions of $\calT$ and $\calT'$ respectively.
Both of these are insulated simplicial triangulations. 
Applying \reflem{BarycentricSubdivision} twice, we get from $\calT$ to $B^2\calT$ through $L$--essential triangulations.

Let $S$ be the star of the ideal vertices $V$ in $B^2\calT$. 
Let $\Lambda = \bdy S$ be the triangulated link of the cusps. 
We define $S'$ and $\Lambda'$ similarly with respect to $B^2\calT'$.  
Note that $\Lambda$ and $\Lambda'$ are two triangulations of a (possibly disconnected) surface $\Sigma$. 
 
The triangulations $\Lambda$ and $\Lambda'$ are related by a sequence of simplicial triangulations of $\Sigma$ where consecutive triangulations are related by a single two-dimensional bistellar move.
See~\cite[Theorem 5, page 64]{Moise77} and~\cite[Theorem 5.9, page 303]{Lickorish99}.
   
Apply \reflem{BoundaryConnectivity} repeatedly to get a sequence of insulated simplicial partially ideal triangulations of $M$, where consecutive triangulations are related by a single three-dimensional bistellar move, taking $B^2\calT$ to a simplicial partially ideal triangulation $\calT_0$ in which the link of $V$ is $\Lambda'$. 
Note that the star of $V$ in $\calT_0$ is equal to $S'$. 

Now consider $\calT_0 - \interior(S')$ and $B^2\calT' - \interior(S')$.
These are triangulations of a compact manifold with no ideal vertices.
The boundary of both triangulations is the triangulated surface $\Lambda'$. 
We now apply \cite[Main Theorem, page 257]{Casali95} to get from $\calT_0 - \interior(S')$ to $B^2\calT' - \interior(S')$ by a sequence of bistellar moves through simplicial triangulations. 
Gluing $\interior(S')$ onto each triangulation gives a sequence of bistellar moves from $\calT_0$ to $B^2\calT'$ through insulated simplicial triangulations.

Thus we obtain a sequence of bistellar moves taking us from $B^2\calT$ to $B^2\calT'$, through insulated simplicial triangulations.
By \refrem{InsulatedSimplicialLEssential}, these intermediate triangulations are $L$--essential.
Finally we reverse the process of taking barycentric subdivisions to get from $B^2\calT'$ back to $\calT'$.
Using \reflem{BarycentricSubdivision} again, the intermediate triangulations are $L$--essential.
\end{proof}

The combinatorics of the 1-4 move are slightly more complicated than that of the bubble move.
To simplify later arguments we replace the former with the latter.

\begin{corollary}
\label{Cor:ConnectivitySimplicial}
Suppose that $M$ is a compact connected three-manifold with non-empty boundary.
Suppose that $\calT$ and $\calT'$ are $L$--essential ideal triangulations of $M$.
Then there is a sequence of $L$--essential (partially) ideal triangulations connecting
$\calT$ to $\calT'$,
where two consecutive triangulations are related by a 2-3 move, a bubble move, or inverses of these moves.
\end{corollary}

\begin{proof}
We modify the sequence of triangulations given by \refthm{ConnectivitySimplicial1-4}.
Using \reflem{1-4} we replace each 
1-4 move with a bubble move followed by a 2-3 move. 
We similarly replace each 4-1 move with a 3-2 move followed by a reverse bubble move.
Each intermediate triangulation is $L$--essential.
\end{proof}

\section{Snake handling}
\label{Sec:SnakeHandling}

In order to prove \refthm{ConnectivityWith0-2}, we take the sequence of moves given by \refcor{ConnectivitySimplicial} and modify it to remove the bubble moves. 
As in the proof of \refthm{ExistsTriangulation} we use snakes to do this.
In this section we describe some moves and algorithms for modifying snakes and snake paths using 2-3, 3-2, 0-2, and 2-0 moves.

\subsection{Building a snake}

The first hurdle we must overcome is that we can only use 2-3, 3-2, 0-2, and 2-0 moves to build our snakes. 

\begin{definition}
\label{Def:BuildSnake}
Suppose that $\calF$ is a foam.
Suppose that $\gamma$ is a snake path in $\calF$.
We \emph{build along $\gamma$} to produce a new foam denoted by $\calF[\gamma]$ as follows.
Consider the first segment $\gamma'$ of $\gamma - \calF^{(1)}$.
We perform a 0-2 move along $\gamma'$; the new bigon is placed at the end of $\gamma'$.
If $\gamma' = \gamma$ then we are done.
If not, let $\eta(\gamma')$ be a small regular neighbourhood of $\gamma'$.
Then $\gamma - \eta(\gamma')$ is again a snake path and we recurse.
\end{definition}

This construction is the dual of~\cite[Algorithm~8.15]{Segerman12}, which was the inspiration for snakes and snake paths in this paper.
Figures~\ref{Fig:BuildSnake1} and~\ref{Fig:BuildSnake2} show one move in the sequence of 0-2 moves. 
\reffig{BuildSnakePath} shows the result of building a snake along the path $\gamma$ shown in \reffig{InflateSnakePath1} if it is oriented from left to right.

\begin{figure}[htbp]
\subfloat[]{
\includegraphics[width = 0.45\textwidth]{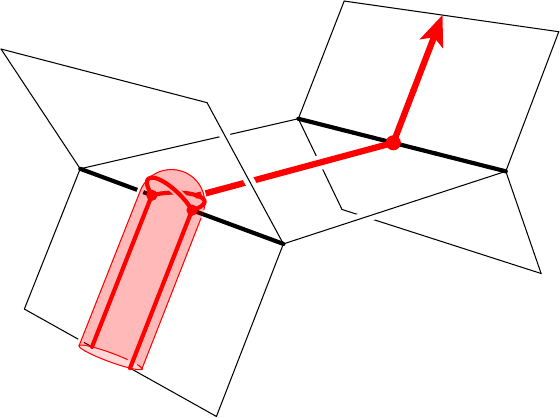}
\label{Fig:BuildSnake1}
}
\quad
\subfloat[]{
\includegraphics[width = 0.45\textwidth]{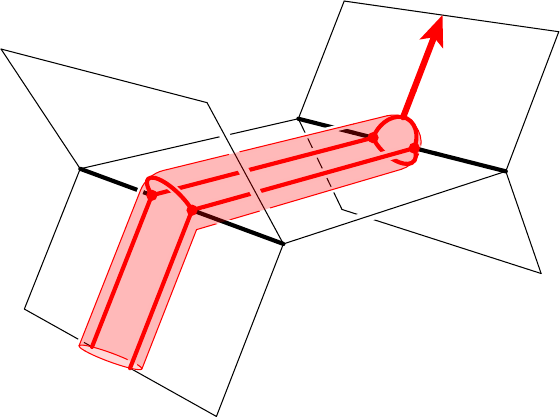}
\label{Fig:BuildSnake2}
}

\subfloat[$\calF{[}\gamma{]}$]{
\includegraphics[width = 0.55\textwidth]{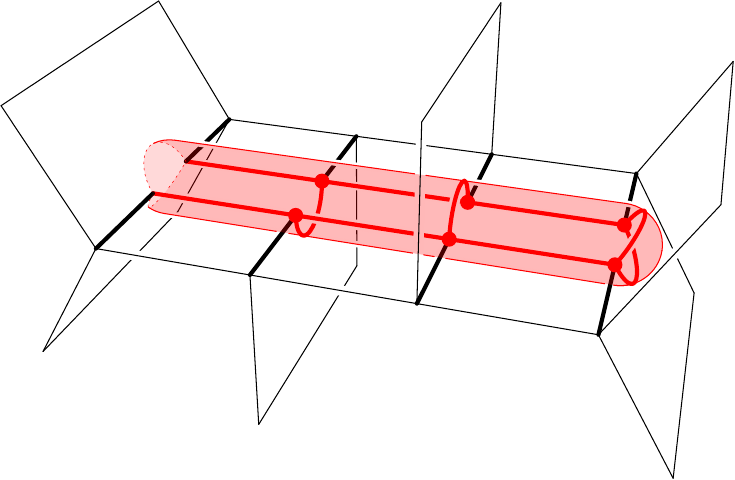}
\label{Fig:BuildSnakePath}
}
\caption{Given a snake path we build a snake using 0-2 moves. 
}
\label{Fig:BuildSnake}
\end{figure}

\begin{remark}
\label{Rem:BuildSnake}
In the construction of the inflated foam $\calF(\gamma)$ (\refdef{InflatingSnakePath}) we remove both bigons from $\bdy \eta(\gamma)$. 
The foam $\calF[\gamma]$ is identical to $\calF(\gamma)$ except that the former retains the bigon at the terminal point of $\gamma$.
\end{remark}

\begin{remark}
\label{Rem:BuildSnakeLEssential}
Note that if $\calF$ is $L$--essential and $\gamma$ is $L$--self-avoiding then the intermediate foams arising in \refdef{BuildSnake} are all $L$--essential. 
The argument is similar to the proof of \reflem{EssentialSnake}.
\end{remark}

\subsection{Basic moves on snakes}
\label{Sec:SnakeBasicMoves}

We will isotope snake paths around on foams $\calF$ in various ways.
When we do so, the corresponding inflated foams $\calF(\gamma)$ change. 
In this section we describe how to implement these modifications using 2-3, 3-2, 0-2, and 2-0 moves.

\subsubsection{Sliding over an edge (slide move A)}
\label{Sec:MoveA}
See \reffig{SnakeSlideOverEdge}. 
We can implement slide move A by doing two 0-2 moves.
\begin{figure}[htbp]
\subfloat[Before slide move A.]{
\includegraphics[width = 0.32\textwidth]{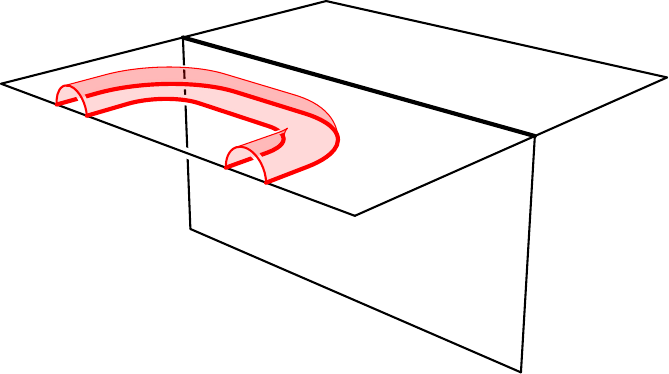}
\label{Fig:SnakeSlideOverEdge1}
}
\qquad
\subfloat[After slide move A.]{
\includegraphics[width = 0.32\textwidth]{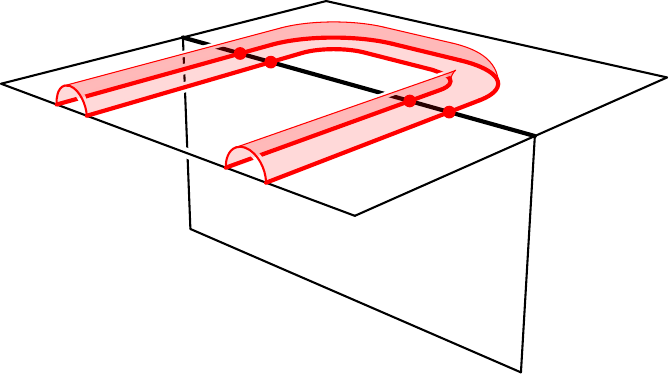}
\label{Fig:SnakeSlideOverEdge2}
}
\caption{Sliding a snake over an edge of the foam (\emph{slide move A}).}
\label{Fig:SnakeSlideOverEdge}
\end{figure}

\subsubsection{Sliding over a vertex (slide move B)}
\label{Sec:MoveB}

See \reffig{SnakeSlideOverVertex}. 
We can implement slide move B by doing two 2-3 moves.
Note that the 2-3 moves each involve one vertex of the background foam and one vertex on the snake. Thus these vertices are distinct and the 2-3 move is possible.

\begin{figure}[htbp]
\subfloat[Before slide move B.]{
\includegraphics[width = 0.32\textwidth]{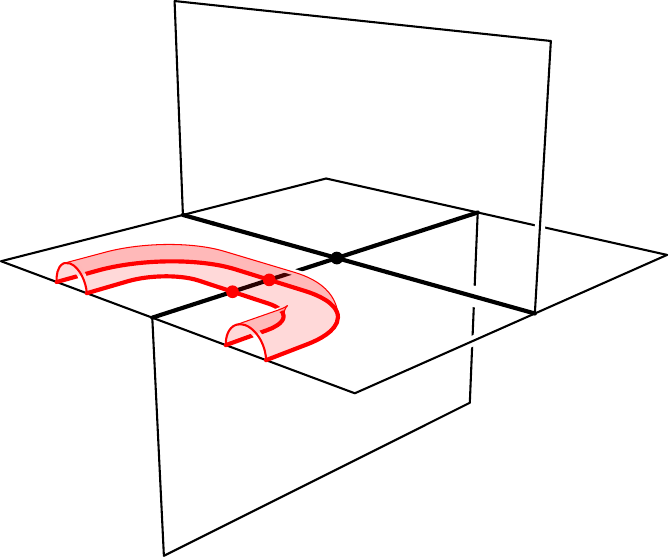}
\label{Fig:SnakeSlideOverVertex1}
}
\qquad
\subfloat[After slide move B.]{
\includegraphics[width = 0.32\textwidth]{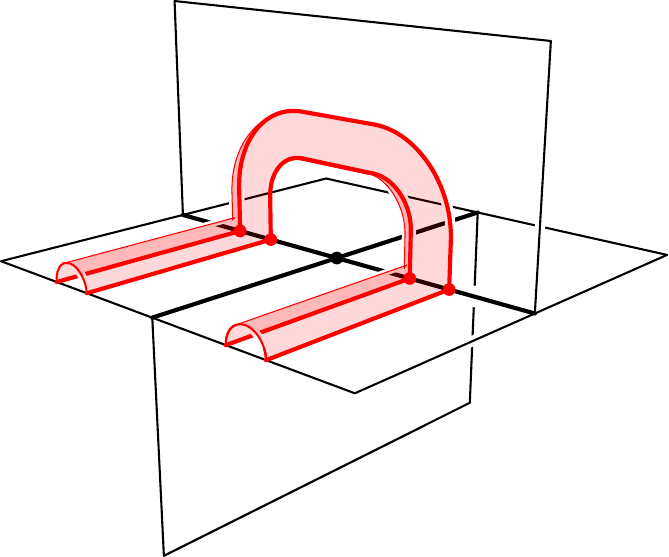}
\label{Fig:SnakeSlideOverVertex2}
}
\caption{Sliding a snake over a vertex of the foam (\emph{slide move B}).}
\label{Fig:SnakeSlideOverVertex}
\end{figure}

\subsubsection{Sliding into a face (slide move C)}
\label{Sec:MoveC}

See \reffig{SnakeSlideIntoFace}. 
We can implement slide move C by doing a 0-2 move.

\begin{figure}[htbp]
\subfloat[Before slide move C.]{
\includegraphics[width = 0.32\textwidth]{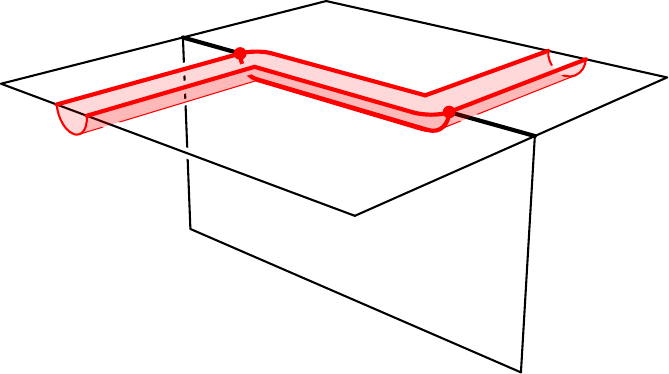}
\label{Fig:SnakeSlideIntoFace1}
}
\qquad
\subfloat[After slide move C.]{
\includegraphics[width = 0.32\textwidth]{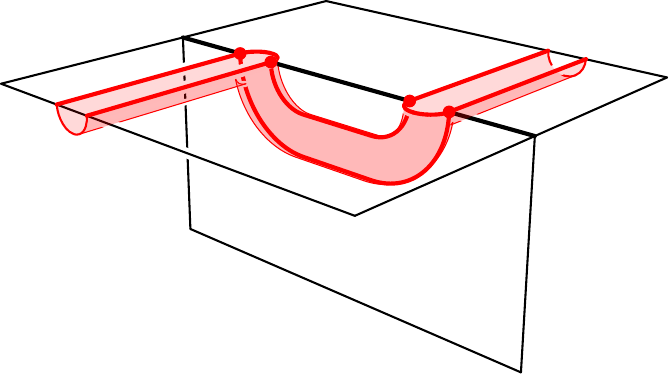}
\label{Fig:SnakeSlideIntoFace2}
}
\caption{Sliding a snake into a face of the foam (\emph{slide move C}).}
\label{Fig:SnakeSlideIntoFace}
\end{figure}

\subsubsection{Sliding a snake end (slide move D)}
\label{Sec:MoveD}

See \reffig{SnakeSlideTail}. 
We can implement slide move D by doing a 0-2 move.

\begin{figure}[htbp]
\subfloat[Before slide move D.]{
\includegraphics[width = 0.32\textwidth]{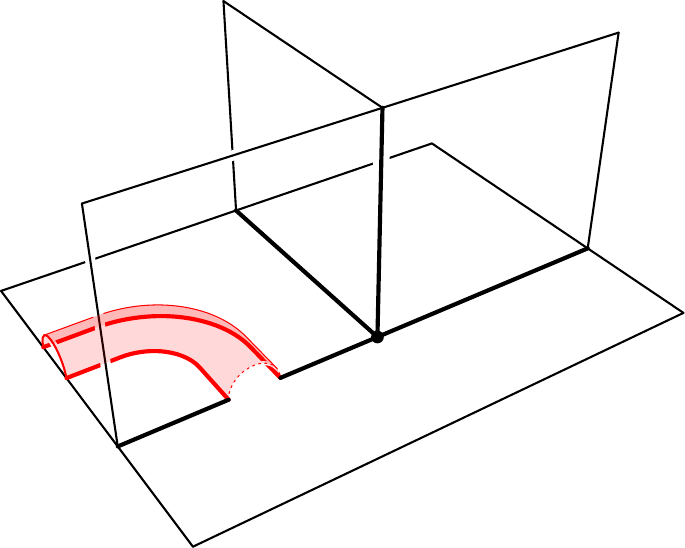}
\label{Fig:SnakeSlideTail1}
}
\subfloat[Before slide move D.]{
\includegraphics[width = 0.32\textwidth]{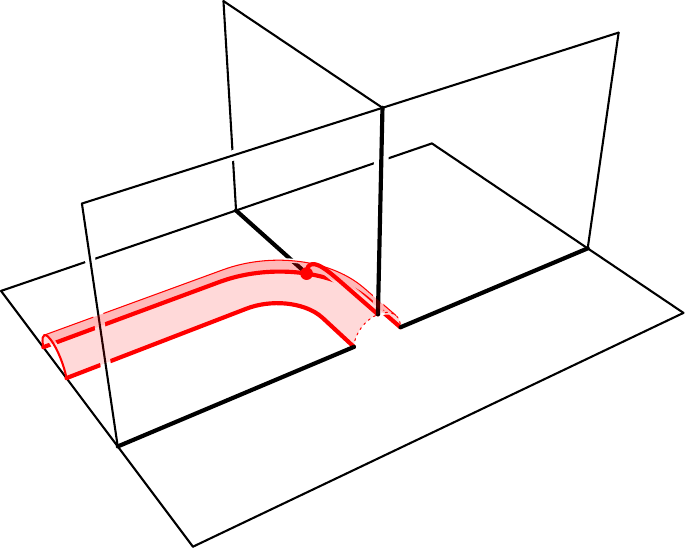}
\label{Fig:SnakeSlideTail2}
}
\subfloat[After slide move D.]{
\includegraphics[width = 0.32\textwidth]{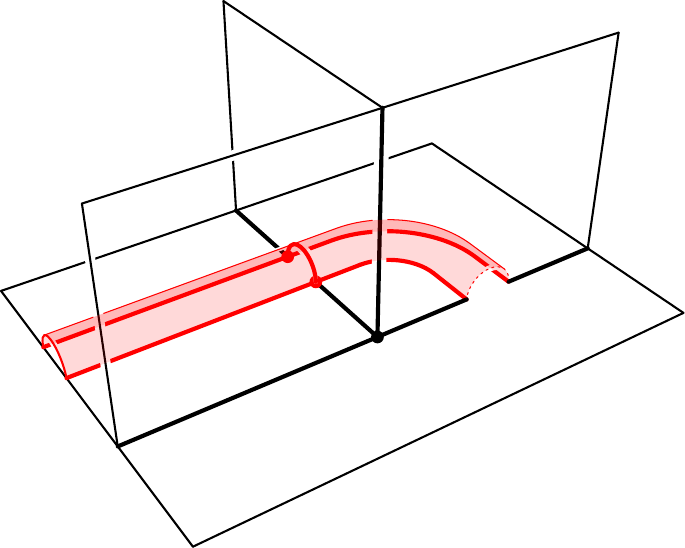}
\label{Fig:SnakeSlideTail3}
}
\caption{Sliding an end of a snake across a vertex (\emph{slide move D}). Here \reffig{SnakeSlideTail2} is an isotopy of \reffig{SnakeSlideTail1} to better see the 0-2 move between \reffig{SnakeSlideTail2} and \reffig{SnakeSlideTail3}.}
\label{Fig:SnakeSlideTail}
\end{figure}

Without any further hypotheses, it is possible that moves A and B could introduce an $L$--inessential face. 
We will deal with this possibility as it arises. 
See, for example, \refsec{TruncateSnake}.

\subsection{Finger moves}
\label{Sec:FingerMoves}

We need five \emph{finger moves}.
Each of these is applied to a foam $\calF$ containing a snake path $\gamma$.
The finger moves are shown in \reffig{FingerMoves}.
Each is an isotopy of $\calF$ in $M$ together with an isotopy of $\gamma$ in $\calF$.

\begin{figure}[htbp]
\subfloat[Before finger move A.]{
\labellist
\hair 2pt \small
\pinlabel $x$ [br] at 93 115
\endlabellist
\includegraphics[width = 0.32\textwidth]{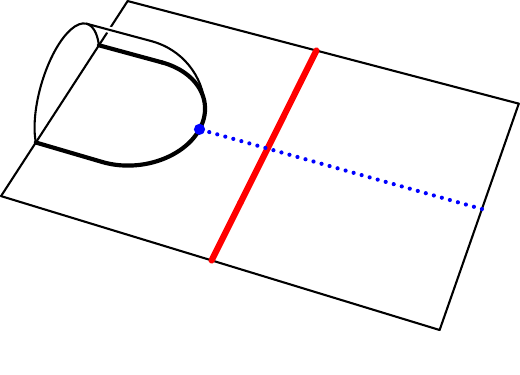}
\label{Fig:FingerA1}
}
\quad
\subfloat[After finger move A.]{
\labellist
\hair 2pt \small
\pinlabel $x$ [br] at 193 86
\endlabellist
\includegraphics[width = 0.32\textwidth]{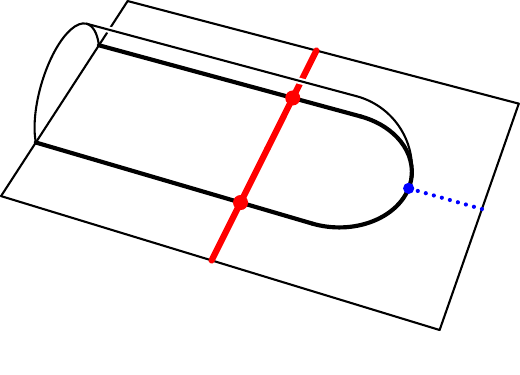}
\label{Fig:FingerA2}
}

\subfloat[Before finger move H.]{
\includegraphics[width = 0.32\textwidth]{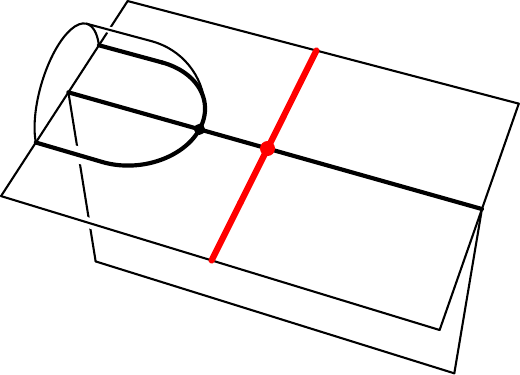}
\label{Fig:FingerHoriz1}
}
\quad
\subfloat[After finger move H.]{
\includegraphics[width = 0.32\textwidth]{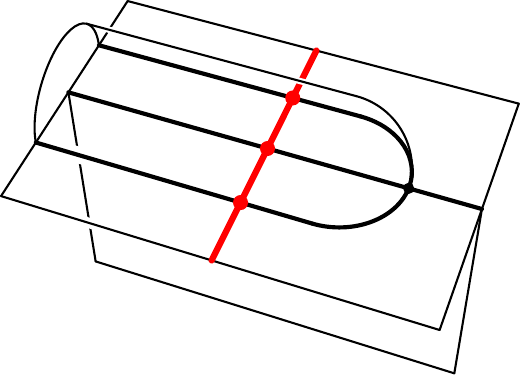}
\label{Fig:FingerHoriz2}
}

\subfloat[Before finger move V.]{
\includegraphics[width = 0.32\textwidth]{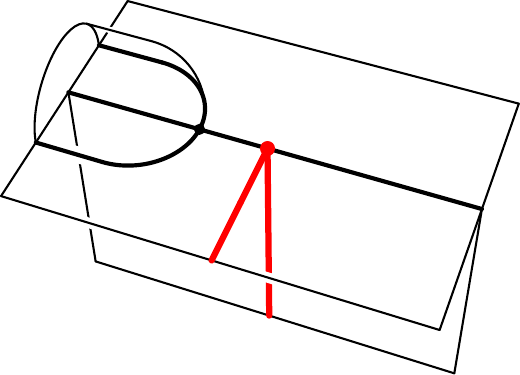}
\label{Fig:FingerVert1}
}
\quad
\subfloat[After finger move V.]{
\includegraphics[width = 0.32\textwidth]{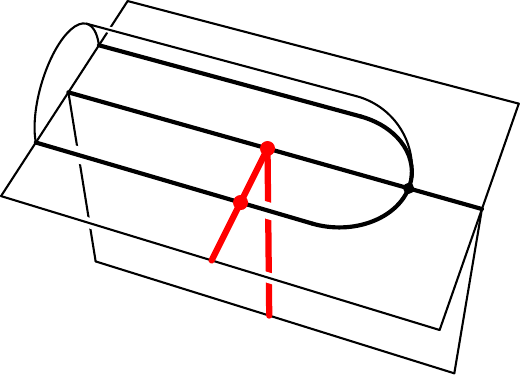}
\label{Fig:FingerVert2}
}

\subfloat[Before finger move EH.]{
\includegraphics[width = 0.32\textwidth]{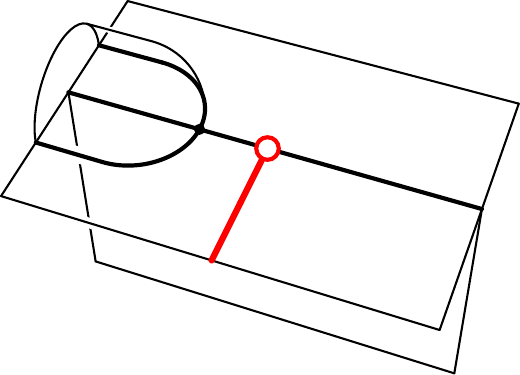}
\label{Fig:FingerHorizNeck1}
}
\quad
\subfloat[After finger move EH.]{
\includegraphics[width = 0.32\textwidth]{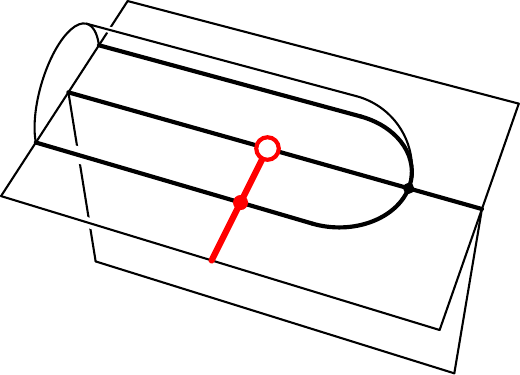}
\label{Fig:FingerHorizNeck2}
}

\subfloat[Before finger move EV.]{
\includegraphics[width = 0.32\textwidth]{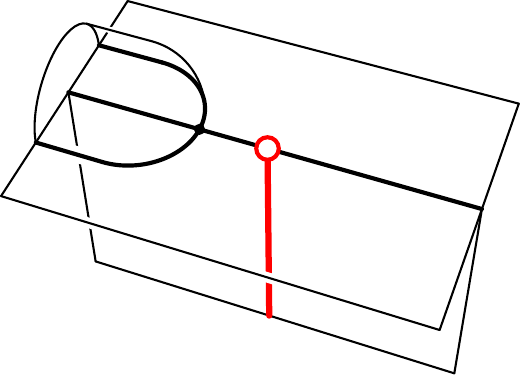}
\label{Fig:FingerVertNeck1}
}
\quad
\subfloat[After finger move EV.]{
\includegraphics[width = 0.32\textwidth]{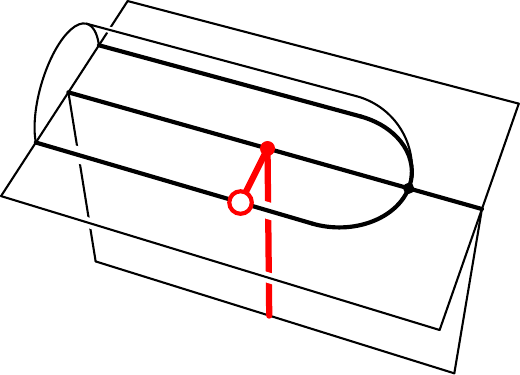}
\label{Fig:FingerVertNeck2}
}
\caption{Finger moves. The snake path is drawn in red. In Figures~\ref{Fig:FingerA1} and~\ref{Fig:FingerA2} the blue dotted line is a path along which a 0-2 move will be applied.}
\label{Fig:FingerMoves}
\end{figure}

Corresponding to a finger move on the pair $(\calF, \gamma)$ we have a move on the inflated foam $\calF(\gamma)$.
These moves on the inflated foams can be implemented, as follows, using basic moves.

\begin{itemize}
\item Finger move A (along the path a 0-2 move will occur) is implemented by a slide move A.
\item Finger move H (across a ``horizontal'' snake path) is implemented by the following sequence of moves: B, C, B, $\textrm{C}^{-1}$.
See \reffig{ImplementHMove}.
\item Finger move V (across a ``partially vertical'' snake path) is implemented by a slide move B.
\item Finger move EH (across the end of a ``horizontal'' snake path) is implemented by a slide move D.
\item Finger move EV (across the end of a ``vertical'' snake path) is implemented by a slide move D.
\end{itemize}

\begin{figure}[htbp]
\subfloat[Before.]{
\includegraphics[width = 0.27\textwidth]{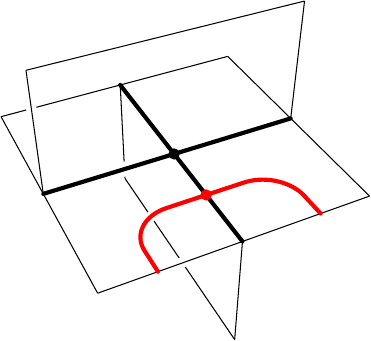}
\label{Fig:ImplementHMove0}
}
\quad
\subfloat[Apply slide move B.]{
\includegraphics[width = 0.27\textwidth]{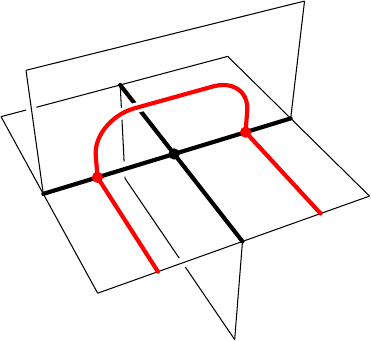}
\label{Fig:ImplementHMove1}
}
\quad
\subfloat[Apply slide move C.]{
\includegraphics[width = 0.27\textwidth]{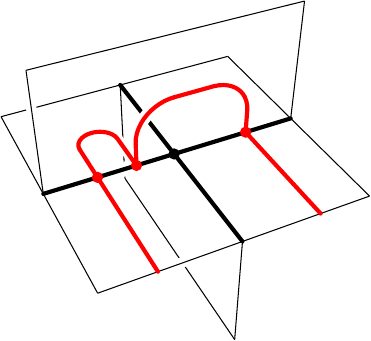}
\label{Fig:ImplementHMove2}
}

\subfloat[Apply slide move B.]{
\includegraphics[width = 0.29\textwidth]{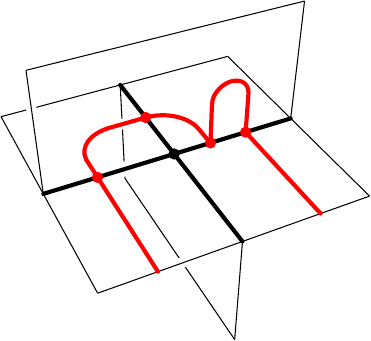}
\label{Fig:ImplementHMove3}
}
\quad
\subfloat[Apply slide move $\textrm{C}^{-1}$.]{
\includegraphics[width = 0.29\textwidth]{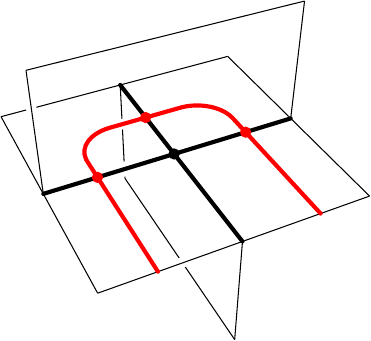} 
\label{Fig:ImplementHMove4}
}
\caption{Steps to implement an H finger move. We show only a small patch of the front of the finger.}
\label{Fig:ImplementHMove}
\end{figure}

Note that if there are no $L$--inessential faces before or after an H finger move then there are also no $L$--inessential faces created by the basic moves that realise it.

\subsection{Replacing a snake}

Given a material region and a snake path targeting it, we may need to replace the snake path in order to change the label on the material region.
The following algorithm achieves this.

\begin{algorithm}[Replace snake]
\label{Alg:ReplaceSnake}\,\\  
\noindent
\textbf{Input:} An $L$--essential foam $\calF$, a material region $D$ of $\cover{M} - \cover{\calF}$, and disjoint $L$--self-avoiding snake paths $\gamma_P$ and $\gamma_Q$ in $\calF$ with the following properties.
\begin{itemize}
\item There is a peripheral region $P$ and a lift $\cover{\gamma}_P$ of $\gamma_P$ with source $P$.
\item There is a peripheral region $Q$ and a lift $\cover{\gamma}_Q$ of $\gamma_Q$ with source $Q$.
\item The targets of $\cover{\gamma}_P$ and $\cover{\gamma}_Q$ are both $D$.
\item $L(P) \neq L(Q)$.
\end{itemize}
\textbf{Output:} A sequence of 2-3, 3-2, 0-2, and 2-0 moves that starts with the inflated foam $\calF(\gamma_P)$, produces the inflated foam $\calF(\gamma_Q)$, and such that every foam in the sequence is $L$--essential.\,\\

Starting from $\calF(\gamma_P)$, we first build a snake (as in \refdef{BuildSnake}) along $\gamma_Q$.
This gives us the foam $\calF(\gamma_P)[\gamma_Q]$.
Note that since $\gamma_Q$ is disjoint from $\gamma_P$, the snake building process for $\gamma_Q$ is unaffected by the presence of the snake $\eta(\gamma_P)$.

Our current goal is to perform 2-3, 3-2, 0-2, and 2-0 moves to get from $\calF(\gamma_P)[\gamma_Q]$ to $\calF(\gamma_Q)[\gamma_P]$, always staying $L$--essential.
In $\calF(\gamma_P)[\gamma_Q]$ the peripheral region $Q$ meets $D$ along a single bigon $b_Q$.
Note that the boundary $\bdy D$ lies in $\calF(\gamma_P)[\gamma_Q]$ apart from a bigon, $b_P$ say, where the snake $\eta({\gamma}_P)$ meets $D$.
Our plan is to push the disk $b_Q$ through $D$ to $b_P$.

Let $e_1$ and $e_2$ be the two edges of $\bdy D$ that meet $b_P$.
See \reffig{ReplaceSnakeBigons}.
Let $G$ be the graph obtained by intersecting the one-skeleton of $\calF(\gamma_P)[\gamma_Q]$ with $D$ and then subtracting the open edges $e_1$ and $e_2$.

\begin{figure}[htbp]
\subfloat[The end of the $P$ snake. The bigon $b_P$ is not a face of the foam.]{
\labellist
\hair 2pt \small
\pinlabel $b_P$ at 77 36
\pinlabel $e_1$ [t] at 32 33
\pinlabel $e_2$ [t] at 122 33
\endlabellist
\includegraphics[width = 0.32\textwidth]{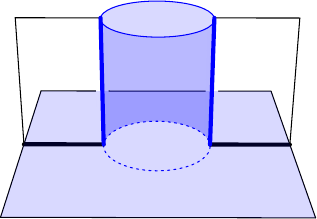}
\label{Fig:ReplaceSnake0P}
}
\quad
\subfloat[The end of the $Q$ snake. The bigon $b_Q$ is a face of the foam.]{
\labellist
\hair 2pt \small
\pinlabel $b_Q$ at 77 36
\endlabellist
\includegraphics[width = 0.32\textwidth]{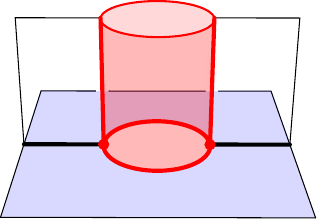}
\label{Fig:ReplaceSnake0Q}
}

\subfloat[Top down view of $b_P$.]{
\includegraphics[width = 0.32\textwidth]{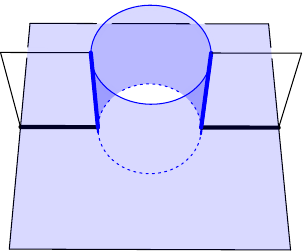}
\label{Fig:ReplaceSnake0PTop}
}
\quad
\subfloat[Top down view of $b_Q$.]{
\includegraphics[width = 0.32\textwidth]{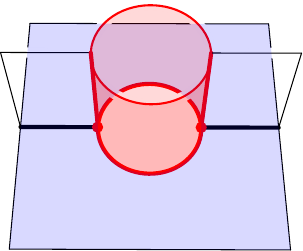}
\label{Fig:ReplaceSnake0QTop}
}
\caption{Bigons incident to a ball.  
}
\label{Fig:ReplaceSnakeBigons}
\end{figure}

\begin{lemma}
$G$ is connected.
\end{lemma}

\begin{proof}
Let $H$ be the graph formed from $G \cup \{e_1, e_2\}$ by gluing the two free endpoints of $e_1$ and $e_2$ together.
Note that $H$ is the one-skeleton of $D$ in $\calF$.
If $H$ had multiple components then some face of $\bdy D$ would not be a disk, a contradiction. 
Suppose instead that $H$ is connected but that $G$ is not. 
Then the faces of $\calF(\gamma_P)[\gamma_Q]$ incident to $e_1$ and $e_2$ that do not lie on $\bdy D$ would be $L$--inessential. 
\end{proof}

Let $T$ be a spanning tree of $G$ that contains (precisely) one of the two edges of $b_Q$.
We delete the interior of this edge from $T$ to form $T_Q$, which consists of two trees (one of which may be a single vertex).

\begin{figure}[htbp]
\subfloat[Before the first 2-3 move.]{
\includegraphics[width = 0.32\textwidth]{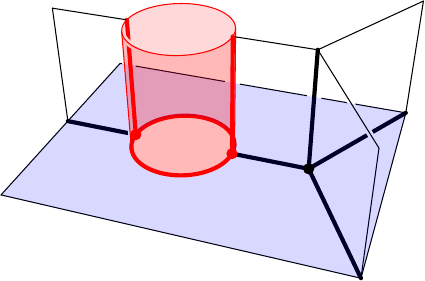}
\label{Fig:ReplaceSnake1Before}
}
\quad
\subfloat[After the first 2-3 move.]{
\includegraphics[width = 0.32\textwidth]{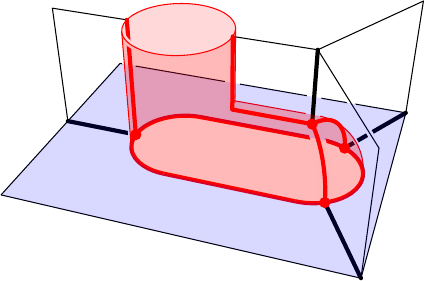}
\label{Fig:ReplaceSnake1After}
}
\caption{2-3 moves expand the $Q$ region. 
}
\label{Fig:ReplaceSnake2-3}
\end{figure}

We apply 2-3 moves to the foam, spreading out from $b_Q$ along $T_Q$.
\reffig{ReplaceSnake2-3} shows the first 2-3 move.
Note that the vertices we wish to apply the first 2-3 move to cannot be the same as each other (and so the edge we want to apply the 2-3 move to cannot be cyclic). 
This is because one of the vertices was just produced by the 0-2 move creating $b_Q$, while the other was already a vertex of $\bdy D$. 
A similar argument applies to the subsequent moves.
The $L$--self-avoiding hypothesis implies that these moves do not produce $L$--inessential faces.

Once all of the 2-3 moves have been done, the edges of $G - T_Q$ are dual to the edges of a dual tree $T^P$ made from the faces of $D$.
Here note that $e_1$ and $e_2$ are not in $G$.
Thus the two faces of $D$ that meet $b_P$ (together with $b_P$ itself) count as a single vertex of $T^P$.
We take this vertex to be the root of $T^P$.
Given a leaf of $T^P$, we may remove it by applying a 2-0 move. 
See \reffig{ReplaceSnake2-0}.

\begin{figure}[htbp]
\subfloat[Before a 2-0 move.]{
\includegraphics[width = 0.32\textwidth]{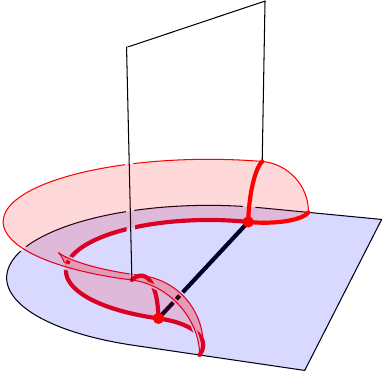}
\label{Fig:ReplaceSnake3Before}
}
\quad
\subfloat[After a 2-0 move.]{
\includegraphics[width = 0.32\textwidth]{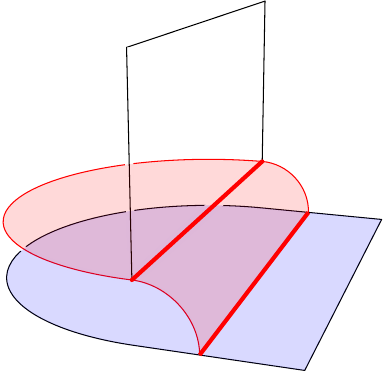}
\label{Fig:ReplaceSnake3After}
}
\caption{2-0 moves expand $Q$ further. 
}
\label{Fig:ReplaceSnake2-0}
\end{figure}

We repeatedly remove leaves of $T^P$ until all that is left is the root of $T^P$.
(Each 2-0 move removes a face of contact between regions but does not add any. 
Thus we remain $L$--essential.)
This done, the region $P$ meets $D$ only in a small neighbourhood of the bigon $b_P$ and we have reached $\calF(\gamma_Q)[\gamma_P]$. 
See \reffig{ReplaceSnakeFinal2-0}.
We retract the snake $\eta(\gamma_P)$ by a sequence of 2-0 moves, following the reverse of the process for building a snake.  
This completes \refalg{ReplaceSnake}.
\end{algorithm}

\begin{figure}[htbp]
\subfloat[Before the isotopy.]{
\includegraphics[width = 0.28\textwidth]{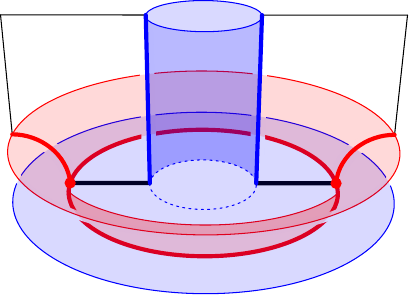}
\label{Fig:ReplaceSnake4Before}
}
\quad 
\subfloat[After the isotopy and before the 2-0 move.]{  
\includegraphics[width = 0.28\textwidth]{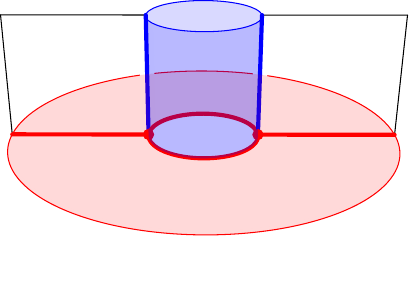}
\label{Fig:ReplaceSnake5}
}
\quad
\subfloat[After the 2-0 move.]{  
\includegraphics[width = 0.28\textwidth]{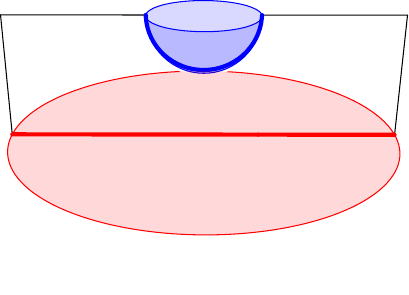}
\label{Fig:ReplaceSnake6}
}
\caption{When only the root of $T^P$ is left, the region $Q$ covers almost all of $\bdy D$. 
\reffig{ReplaceSnake4Before} shows the last part still uncovered.
An isotopy shrinks the intersection of $P$ and the ball $D$ to a small neighbourhood of the bigon $b_P$ (see \reffig{ReplaceSnake5}).
A single 2-0 move removes the last contact between $P$ and the ball $D$.
The result is shown in \reffig{ReplaceSnake6}.
}
\label{Fig:ReplaceSnakeFinal2-0}
\end{figure}

Note that the combinatorics of $\bdy D$ are identical before and after replacing one snake with the other, except possibly for where the snakes meet $D$. 

\section{Traversing foams with snakes}
\label{Sec:TraversingWithSnakes}

The goal of this section is to prove the following.

\begin{theorem}
\label{Thm:ConnectivityWith0-2}
Suppose that $M$ is a compact, connected three-manifold.
Suppose that $L$ is a labelling of $\Delta_M$ with infinite image.
Then the set of $L$--essential ideal triangulations of $M$ is connected via 2-3, 3-2, 0-2, and 2-0 moves.
\end{theorem}

Here is the overall plan.
Suppose that $\calT$ and $\calT'$ are $L$--essential ideal triangulations of $M$.
We take the sequence of triangulations $(\calT_i)_{i=0,\ldots n}$ given by \refcor{ConnectivitySimplicial} and modify it.
Let $(\calF_i)_{i=0,\ldots n}$ be the sequence of foams dual to $(\calT_i)$.
We produce a new sequence $(\calF'_{i'})_{i'=0,\ldots n'}$ of $L$--essential foams with the following properties.

\begin{itemize}
\item $\calF'_0 = \calF_0$,
\item $\calF'_{n'} = \calF_n$, 
\item $\calF'_{i'+1}$ is obtained from $\calF'_{i'}$ by a 2-3, 3-2, 0-2, 2-0, bubble, or reverse bubble move, and
\item The sequence $(\calF'_{i'})$ has one fewer bubble move than the sequence $(\calF_i)$.
\end{itemize}

If $(\calF'_{i'})$ has no bubble moves then we are done. 
Otherwise we run the above algorithm again starting with $(\calF'_{i'})$.
This eventually gives a sequence starting and ending at the same foams as $(\calF_i)$ and which uses no bubble moves, proving \refthm{ConnectivityWith0-2}.

In order to produce $(\calF'_{i'})$ from $(\calF_i)$, we choose any index $p$ so that $\calF_{p}$ is obtained from $\calF_{p-1}$ by applying a bubble move.
Suppose that the material region $Z$ created by this bubble move is removed by the reverse bubble move between $\calF_q$ and $\calF_{q+1}$.
(The index $q$ is well-defined since the final foam $\calF_n$ has no material regions.)
For $i < p$ we take $\calF'_i = \calF_i$.
In this section we construct the middle of $(\calF'_{i'})$, corresponding to the sequence $(\calF_i)_{i=p,\ldots,q}$.
We take the suffix of $(\calF'_{i'})$ to be equal to the corresponding suffix of $(\calF_i)$, with appropriate reindexing.

\subsection{Notation}

To keep track of the middle section of $(\calF'_{i'})$ we introduce the following notation.
Suppose that we have a snake path $(\gamma_{i,j}) \subset \calF_i^{(2)}$.
Here $i$ ranges between $p$ and $q$ and $j$ ranges between $0$ and some $J_i$.
Let $\calF_{i,j} = \calF_i(\gamma_{i,j})$ be the result of inflating $\calF_i$ along $\gamma_{i,j}$.

In \refsec{BubbleMove} we show how to move from $\calF'_{p-1} = \calF_{p-1}$ to $\calF'_{p'} = \calF_{p,0}$ (for some appropriate index $p'$) using only 2-3, 3-2, 0-2, and 2-0 moves.
(In this section, we also deal with any bubble move that does not create the material region $Z$ we are removing in this iteration of the algorithm.)
In Sections~\ref{Sec:0-2Move} through~\ref{Sec:3-2Move}, we show how to move from  $\calF_{i,j}$ to $\calF_{i,j+1}$, again using only 2-3, 3-2, 0-2, and 2-0 moves.
We similarly show how to move from $\calF_{i,J_i}$ to $\calF_{i+1,0}$.
(Note that everything is contained in $M$ so it makes sense to talk about snake paths living in $\calF_i$ also living in $\calF_{i+1}$.)
Finally, in \refsec{ReverseBubbleMove} we show how to move from $\calF_{q,0}$ to $\calF_{q+1}$.
(In this section, we also deal with any reverse bubble moves that do not destroy the material region $Z$.)

\subsection{Bubble move}
\label{Sec:BubbleMove}
Suppose that the bubble move taking $\calF_{i}$ to $\calF_{i+1}$ occurs on the edge $e$.
First assume that $i = p - 1$.
Here we must create the snake path $\gamma_{p,0}$.
Choose a lift $\cover{e}$ of $e$. 
By hypothesis, $L(\Delta_M)$ is infinite, so we may choose a label $\ell \in L(\Delta_M)$ different from any label on the three complementary regions that meet the interior of $\cover{e}$.

Let $D$ be the material region in the complement of $\cover\calF_{p}$ that is created by the bubble move and splits $\cover{e}$ in two.
Note that no translate of $D$ shares a face with $D$.
We now apply \refalg{CreateSnakePath} with foam $\calF_{p}$, with material region $D$, label $\ell$ (using \refrem{StronglyCreateSnakePath}), and with $\calA$ the three faces of the image $\phi(D)$ under the covering map $\phi$.

Let $\gamma$ be the resulting $L$--self-avoiding snake path in $\calF_{p}$.
Since $\gamma$ is disjoint from $\phi(D)$ other than at its terminal point, $\gamma$ lies in $\calF_{p-1}$.
Let $\gamma'$ be the $L$--self-avoiding snake path in $\calF_{p-1}$ formed by adding to $\gamma$ a small segment that connects within $\phi(D)$ from the end of $\gamma$ to a point of $e$.
See \reffig{BubbleTruncate}.
We take $\gamma_{p,0} = \gamma$.
(Note that for the bubble move we do not use the $j$ index.)

To get from $\calF_{p-1}$ to $\calF_{p,0} = \calF_p(\gamma_{p,0}) = \calF_{p-1}[\gamma']$ we build along the snake path $\gamma'$ as in \refdef{BuildSnake}.
\reffig{BuildSnakeBubble} shows the result of the last 0-2 move at the end of $\gamma'$.
Here we draw it in a different style, showing it as the result of inflating $\calF_{p}$ along $\gamma_{p,0}$, with the bubble region shaded in blue.
By \refrem{BuildSnakeLEssential} every foam generated in our sequence is $L$--essential.

\begin{figure}[htbp]
\subfloat[We extend $\gamma$ by a small segment inside of the bubble to make $\gamma'$.]{
\includegraphics[width = 0.45\textwidth]{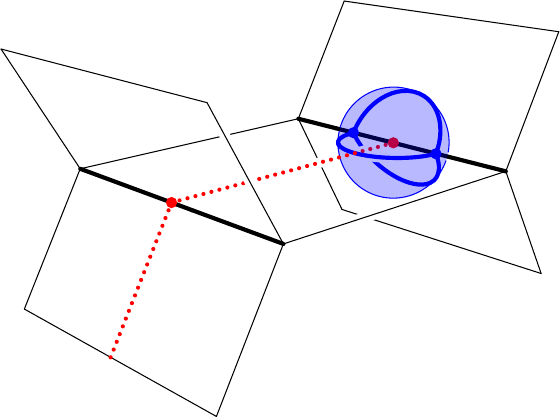} 
\label{Fig:BubbleTruncate}
}
\quad
\subfloat[The last step of building a snake, interpreted as creating a bubble.]{
\includegraphics[width = 0.45\textwidth]{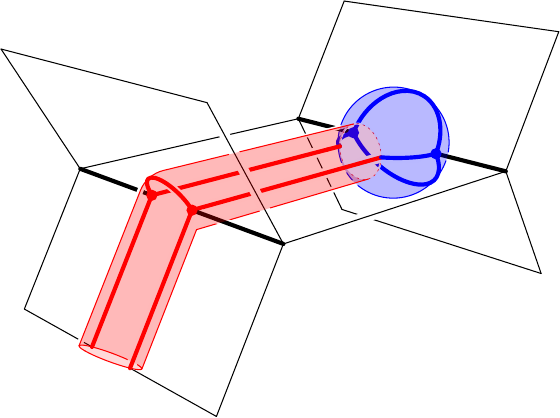} 
\label{Fig:BuildSnakeBubble}
}
\caption{}
\label{Fig:BuildSnake3}
\end{figure}

Next we deal with the case that $p \leq i < q$. 
Thus the bubble move, applied to $\calF_i$, introduces a material region in this iteration of the algorithm.
By an isotopy of the bubble in $\calF_{i+1}$ we may assume that the bubble does not intersect $\gamma_{i,0}$.
This allows us to set $\gamma_{i+1,0} = \gamma_{i,0}$. 
The bubble move taking $\calF_{i}$ to $\calF_{i+1}$ also takes $\calF_{i,0}$ to $\calF_{i+1,0}$.

\subsection{0-2 move}
\label{Sec:0-2Move}

Suppose that applying a 0-2 move along a path $\delta$ in a face $f$ of $\calF_i$ produces $\calF_{i+1}$.
(Note that the original sequence provided by \refcor{ConnectivitySimplicial} has no 0-2 or 2-0 moves, but later sequences in our recursion will.)
Let $\cover{\delta}$ be a lift of $\delta$ to $\cover{M}$.
Let $\cover{f}$ be the lift of $f$ containing $\cover{\delta}$.
See \reffig{0-2}.
There are four regions of $\cover{M} - \cover{\calF}_i$ incident to $\cover{\delta}$: two meet its interior and two meet only its endpoints.
Let $N$ and $S$ be the regions at the endpoints.
Note that they come into contact in the $L$--essential foam $\calF_{i+1}$.
If $N$ and $S$ are both peripheral then we deduce that $L(N) \neq L(S)$.
If $N$ and $S$ are both material then, by $L$--essentiality, $\phi(N) \neq \phi(S)$.
Only a material region that projects to $Z$ has a label.
Now suppose (breaking symmetry) that $N$ is material and $S$ is peripheral.
Here, it is possible that $L_{\gamma_{i,0}}(N) = L(S)$ (recall the notation given in \refdef{LabelledSnakePath}).
They are equal if and only if we have the following:
\begin{itemize}
\item the image $\phi(N) = Z$ is (in $\calF_{i}$) the target for $\gamma_{i,0}$ and
\item the lift $\cover{\gamma}_{i,0}$ that has $N$ as its target has $L_{\gamma_{i,0}}(\cover{\gamma}_{i,0}) = L(S)$.
\end{itemize}

Suppose further that $\gamma_{i,0}$ is disjoint from $\delta$. 
(As we will see, this can be arranged by appropriately modifying $\gamma_{i,0}$.)
Then bringing $N$ and $S$ together with the 0-2 move would create an $L$--inessential face in the inflated foam $\calF_{i+1}(\gamma_{i,0})$. 
To avoid this we ``relabel'' $N$ as follows. 

\subsubsection{Replacing a snake path}
\label{Sec:ReplaceSnakePath}

By hypothesis, $L(\Delta_M)$ is infinite, so we may choose a label $\ell \in L(\Delta_M)$ different from $L(\cover{\gamma}_{i,0})$ and different from the label of any complementary region that meets $N$.

Note that no translate of $N$ shares a face with $N$.
We apply \refalg{CreateSnakePath} with foam $\calF_i$, with material region $N$, with label $\ell$ (using \refrem{StronglyCreateSnakePath}), and with $\calA = \{\gamma_{i,0}\}$.
This returns an $L$--self-avoiding snake path $\gamma_\ell$ which is disjoint from $\gamma_{i,0}$.
We take $\gamma_{i,1} = \gamma_\ell$. 
Replacing $\gamma_{i,0}$ with $\gamma_{i,1}$ has the effect of changing the label on $N$ from $L_{\gamma_{i,0}}(N) = L_{\gamma_{i,0}}(\cover{\gamma}_{i,0}) = L(S)$ to $\ell \neq L(S)$. 
\refalg{ReplaceSnake} then describes how to connect $\calF(\gamma_{i,0})$ to $\calF(\gamma_{i,1})$.

\subsubsection{Clearing $\delta$}
\label{Sec:ClearPath}

Recall that $\delta$ is the path on which the 0-2 move is applied.
Recall that $\cover{\delta}$ is its chosen lift.
Recall that $N$ and $S$ are the two regions meeting the endpoints of $\cover{\delta}$.
By \refsec{ReplaceSnakePath} we may assume that the labels of $N$ and $S$ (if they exist) are distinct.
If at this point the snake path $\gamma$ (either $\gamma_{i,0}$ or $\gamma_{i,1}$) does not meet $\delta$ then we can apply the 0-2 move.
If not then we recursively clear $\gamma$ off of $\delta$ as follows.

Choose $X$ to be one of $N$ or $S$.
Suppose that $x$ is the endpoint of $\cover{\delta}$ meeting $X$.
See \reffig{FingerA1}.
We choose a small neighbourhood of $x$ in one of the two faces of $X$ containing $x$.
We push the small neighbourhood along $\cover{\delta}$.
Let $\cover{\gamma}$ be the lift of the snake path that meets $\cover{\delta}$ closest to $x$.
If the label of $X$ (if it exists) and the label of $\cover{\gamma}$ are different then 
we may perform the finger move, obtain \reffig{FingerA2}, and $\gamma$ remains $L$--self-avoiding. 
On the other hand, if $L(X) = L(\cover{\gamma})$ then we truncate the snake path $\gamma$ as follows.


\subsubsection{Truncating a snake path}
\label{Sec:TruncateSnake}

\begin{figure}[htbp]
\subfloat[]{
\labellist
\hair 2pt \small
\pinlabel $X$ [r] at 20 148
\pinlabel $x$ [br] at 93 115
\pinlabel $\cover{\gamma}$ [tl] at 119 85
\pinlabel $\cover{\delta}$ [bl] at 209 85
\endlabellist
\includegraphics[width = 0.31\textwidth]{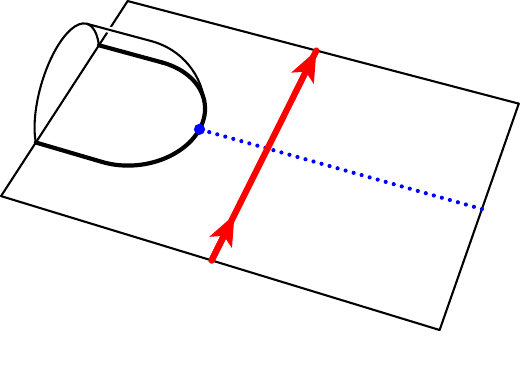}
\label{Fig:Reroute0-2Before}
}
\subfloat[End $\cover{\gamma}$ early.]{
\labellist
\hair 2pt \small
\pinlabel $X$ [r] at 20 148
\pinlabel $x$ [br] at 193 86
\pinlabel $\cover{\gamma}_-$ [tr] at 113 52
\pinlabel $\cover{\delta}$ [bl] at 209 85
\endlabellist
\includegraphics[width = 0.31\textwidth]{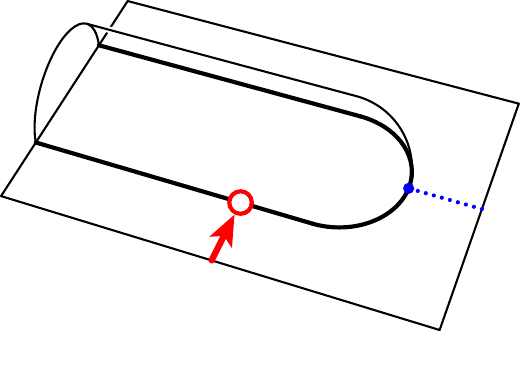}
\label{Fig:Reroute0-2_1}
}
\subfloat[Start $\cover{\gamma}$ late.]{
\labellist
\hair 2pt \small
\pinlabel $X$ [r] at 20 148
\pinlabel $x$ [br] at 193 86
\pinlabel $\cover{\gamma}_+$ [bl] at 145 159
\pinlabel $\cover{\delta}$ [bl] at 209 85
\endlabellist
\includegraphics[width = 0.31\textwidth]{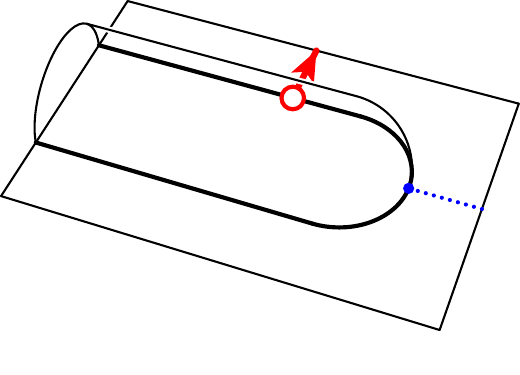}
\label{Fig:Reroute0-2_2}
}
\caption{Ways to truncate a snake path $\cover{\gamma}$ instead of having it meet a region with the same label as itself.}
\label{Fig:Reroute0-2}
\end{figure}
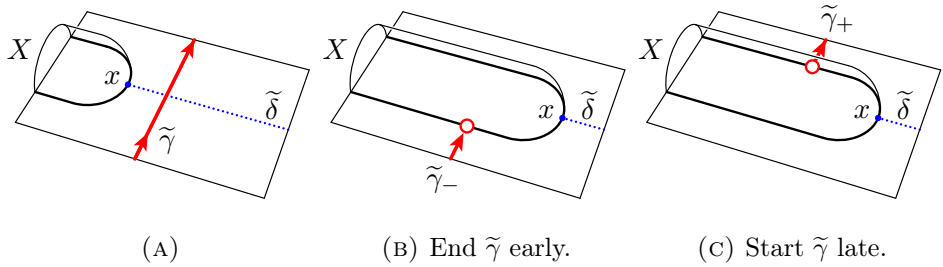

There are two cases as $X$ is either material or peripheral.

\begin{itemize}
\item
Suppose that $X$ is material.
We replace $\cover\gamma$ essentially with a prefix of itself.
This is illustrated in \reffig{Reroute0-2_1}; here are the details.
Let $\cover\gamma'$ be the snake path produced by applying the $A$ move to slide the finger onto $\cover\gamma$. 
Let $\cover\gamma_-$ be the prefix of $\cover\gamma'$ whose lift ends where it meets $X$. 
By isotoping $\cover\gamma_-$ slightly to one side of $\cover\gamma$ we make them disjoint.
The snake path $\gamma_- = \phi(\cover\gamma_-)$ is, up to a small isotopy, a prefix of the $L$-self-avoiding snake path $\gamma$. 
Thus $L(\cover\gamma_-) = L(\cover{\gamma})$ is distinct from the labels of regions meeting the interior of $\cover\gamma_-$.
Moreover, this label is distinct from regions meeting $X$.
Thus $\gamma_-$ is $L$--self-avoiding.

Since $L(\Delta_M)$ is infinite we may choose a label $\ell \in L(\Delta_M)$ different from the label of $\cover\gamma$ and of any complementary region that meets $X$.
Note that no translate of $X$ shares a face with $X$.
We apply \refalg{CreateSnakePath} with the current foam, with material region $X$, with label $\ell$ (using \refrem{StronglyCreateSnakePath}), and with $\calA = \{\gamma, \gamma_-\}$.
This produces an $L$--self-avoiding snake path $\gamma_\ell$ disjoint from $\calA$.

With $\gamma_\ell$ in hand, we replace $\gamma$ by $\gamma_\ell$ using \refalg{ReplaceSnake}.
This done, we replace $\gamma_\ell$ by $\gamma_-$, again using \refalg{ReplaceSnake}.
This achieves the desired truncation.

\item 
Suppose that $X$ is peripheral.
In this case we replace $\cover\gamma$ essentially with a suffix of itself.
This suffix, $\cover\gamma_+$ say, starts on $X$ (or possibly a $\Stab(L(X))$ translate of $X$).
This is illustrated in \reffig{Reroute0-2_2}.
This case is very similar to the first.
\end{itemize}

\subsubsection{Applying the 0-2 move}
\label{Sec:Apply0-2}

In all cases, each finger or truncation move we use advances the $j$ index, giving the sequence $(\gamma_{i,j})_{j = 0,\ldots,J_i}$. 
The foams $\calF_{i,j}$ are related to each other by sequences of 2-3, 3-2, 0-2, and 2-0 moves.
As $j$ increases, the number of intersections between $\delta$ and $\gamma_{i,j}$ decreases.
After a finite number of moves we have cleared all snake paths from ${\delta}$.

At this point we implement the 0-2 move along $\delta$ that changes $\calF_i$ into $\calF_{i+1}$.
We do this in a neighbourhood of $\delta$ small enough to avoid the current snake path $\gamma_{i,J_i}$. 
This allows us to set $\gamma_{i+1,0} = \gamma_{i,J_i}$.
Finally, the same 0-2 move takes the foam $\calF_{i,J_i}$ to the foam $\calF_{i+1,0}$. 

\subsection{2-3 move}
\label{Sec:2-3Move}

Suppose that applying a 2-3 move to edge $e$ of $\calF_i$ produces $\calF_{i+1}$.
Let $\cover{e}$ be a lift of $e$ to $\cover{M}$.
See \reffig{2-3}.
There are five regions of $\cover{M} - \cover{\calF}_i$ incident to $\cover{e}$: three meet its interior and two meet only its endpoints. 
Let $N$ and $S$ be the regions at the endpoints.

As in \refsec{0-2Move}, and up to breaking symmetry, there are three cases depending on the materiality or peripherality of $N$ and $S$. 
Again, in the case that one is material and one is peripheral it is possible that they share the same label.
If so then we replace the label on the material region, exactly as we did in \refsec{ReplaceSnakePath}.

Suppose now that no snake paths meet $\cover{e}$.
Then we simply perform the 2-3 move.

Otherwise, we must first clear $\cover{e}$ of snake paths.
This is very similar to the process of clearing the path for a 0-2 move, before performing the 0-2 move, as discussed in \refsec{ClearPath}.
Again we choose $X$ to be one of $N$ or $S$.
We choose one of the three face corners incident to the $X$ end of $\cover{e}$ and push it towards the other end of $\cover{e}$.
When we meet a lift $\cover{\gamma}$ of the snake path, we use finger moves (as in \refsec{FingerMoves}) or truncate the snake (similar to the method of \refsec{TruncateSnake}).
Here, instead of the one-finger move A, there are four different finger moves, H, V, EH, and EV, depending on
\begin{itemize}
\item
whether a snake path crosses or ends on $\cover{e}$ and 
\item
how the snake path is arranged relative to the finger.
\end{itemize}
See \reffig{FingerMoves}.

If we apply move H or V then it is possible that the lift $\cover\gamma$ and $X$ share the same label.
In these cases we truncate the snake path essentially as described in \refsec{TruncateSnake}.
If we apply move EH or EV then the label on the snake path $\cover\gamma$ cannot match the label on $X$.
This is because the region connected to the end of the snake path already meets $X$ along a face of the foam and by induction the foam is $L$--essential.

This constructs the sequence of snake paths $(\gamma_{i,j})_{j = 0,\ldots,J_i}$ as before.
The remainder of the construction is the same as in \refsec{Apply0-2}.

\subsection{2-0 move}
\label{Sec:2-0Move}

Suppose that applying a 2-0 move to the bigon $b$ of $\calF_i$ produces $\calF_{i+1}$.
Let $\cover{b}$ be a lift of $b$ to $\cover{M}$.
All regions around $\cover{b}$ have distinct (or no) labels; this is because they are all in contact with each other in $\cover{\calF}_i$.
Suppose that $\gamma = \gamma_{i,0}$ is the current snake path.
If $\gamma$ is disjoint from $b$ then we can apply the 2-0 move.
If not then we must clear $\cover{b}$ of snake paths.
Let $e_1$ and $e_2$ be the edges of $b$.
We first clear $\gamma$ (incrementing $j$) off of $e_1$ using exactly the same process as in \refsec{2-3Move}.

Next, we perform finger moves (possibly with truncations) along $e_2$ as shown in \reffig{ClearBigon}.
In particular, we push the finger from a face that shares an edge with $b$.
When we perform an EV finger move (as in \refsec{FingerMoves}), we have a choice of which side of the finger to move the endpoint of $\gamma$ to.
We always push the endpoint off of $b$. 
(Note that a neighbourhood of $b$ is an embedded three-ball because we are about to perform a 2-0 move.)

\begin{figure}[htbp]
\subfloat[]{
\labellist
\hair 2pt \small
\pinlabel $e_1$ [r] at 101 110
\pinlabel $b$ at 154 105
\pinlabel $e_2$ [l] at 207 100
\endlabellist
\includegraphics[width = 0.43\textwidth]{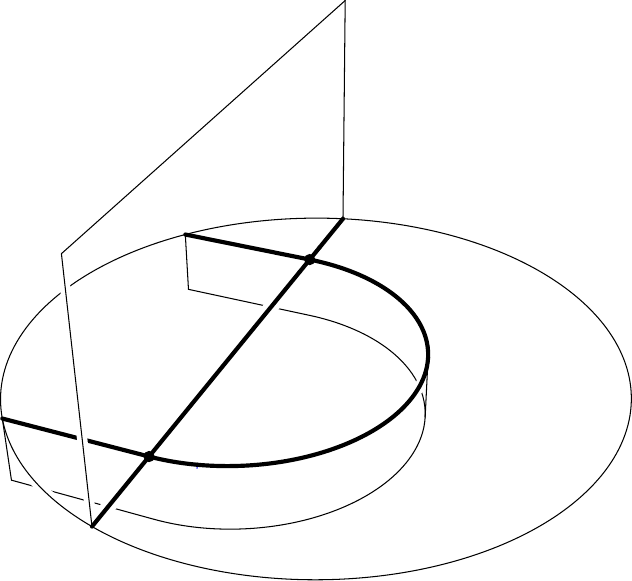}
\label{Fig:ClearBigon1}
}
\quad
\subfloat[]{
\labellist
\hair 2pt \small
\pinlabel $b$ at 149 105
\endlabellist
\includegraphics[width = 0.43\textwidth]{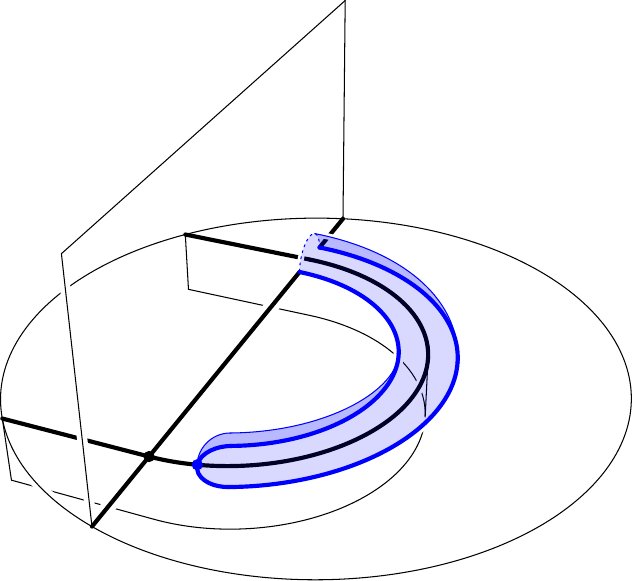}
\label{Fig:ClearBigon2}
}
\caption{We use a sequence of finger moves to clear a bigon.}
\label{Fig:ClearBigon}
\end{figure}

We continue doing finger moves (perhaps with truncations) until we have passed all intersections of $\gamma$ with $e_2$.
\begin{claim*}
At this point neither endpoint of $\gamma$ is on $b$.
\end{claim*}

\begin{proof}
By our choice of face to push the finger from, after an EH finger move, the endpoint of $\gamma$ is no longer in contact with $b$.
Given our choice during each EV finger move, the only way that an endpoint of $\gamma$ could end up on $b$ is if it is there as a result of truncating the snake path, as in \refsec{TruncateSnake}. 
However, every segment of $\gamma$ enters and exits $b$ on the edge where $b$ meets the finger (see \reffig{ClearBigon2}).
Suppose for a contradiction that there is an endpoint $x$ of $\gamma$ on this edge.
Let $\cover{x}$ be the lift of $x$ meeting $\cover{b}$.
Let $\cover{\gamma}$ be the lift of $\gamma$ containing $\cover{x}$.
Let $R$ be the region containing the finger meeting $\cover{x}$.
Thus the label of $R$ agrees with the label on $\cover{\gamma}$.
Since $\gamma$ reenters the finger, we find that $\gamma$ is not $L$--self-avoiding.
This contradicts the fact that 
the snake paths produced by finger moves (perhaps with truncations) are always $L$--self-avoiding.
\end{proof}

Therefore $\gamma$ meets $b$ along a disjoint collection of arcs with no endpoint of $\gamma$ meeting $b$, and with all endpoints of arcs meeting the same edge of $b$. 
Each of these arcs can be removed using a slide move $\textrm{A}^{-1}$ (see \refsec{MoveA}), applied to outermost arcs first.
Once $\gamma$ has been cleared off of $b$ we apply the 2-0 move.

\subsection{3-2 move}
\label{Sec:3-2Move}

Suppose that applying a 3-2 move to the triangular face $f$ of $\calF_i$ produces $\calF_{i+1}$.
Let $\cover{f}$ be a lift of $f$ to $\cover{M}$.
All regions around $\cover{f}$ have distinct labels because they are all in contact with each other in $\cover{\calF}_i$.
If $\gamma$ is disjoint from $f$ then we can apply the 3-2 move.
If not then we must clear $\cover{f}$ of snake paths.
The process is very similar to the process of clearing the bigon in preparation for a 2-0 move (see \refsec{2-0Move}). 
Let $e_1$, $e_2$, and $e_3$ be the edges of $f$.
We clear $\gamma$ off $e_1$ using the same process as in \refsec{2-3Move}.
Next, we perform finger moves along both $e_2$ and $e_3$, as shown in \reffig{ClearTriangle}.
We push endpoints of $\gamma$ off of $f$ if we make any EV finger moves.
(Note that $f$ is embedded because we are about to perform a 3-2 move.)

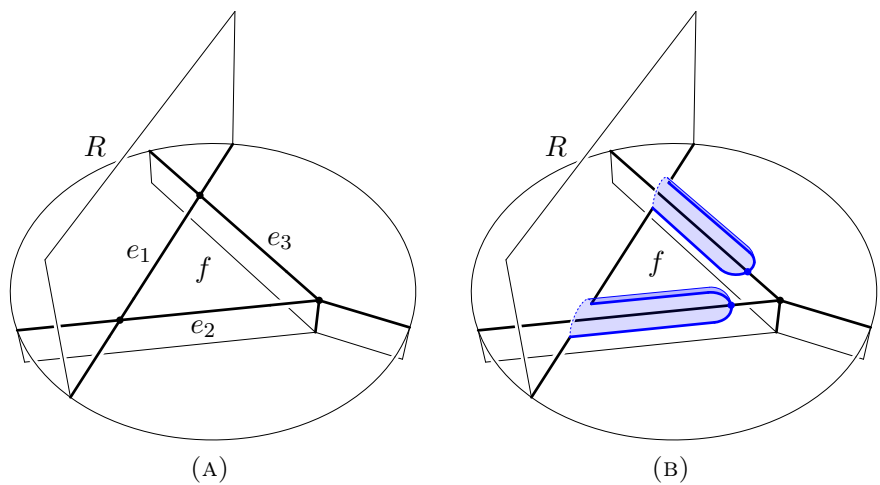
\begin{figure}[htbp]
\subfloat[]{
\labellist
\hair 2pt \small
\pinlabel $R$ at 65 220
\pinlabel $e_1$ [r] at 111 140
\pinlabel $f$ at 145 128
\pinlabel $e_2$ [t] at 145 95
\pinlabel $e_3$ [l] at 188 150
\endlabellist
\includegraphics[width = 0.43\textwidth]{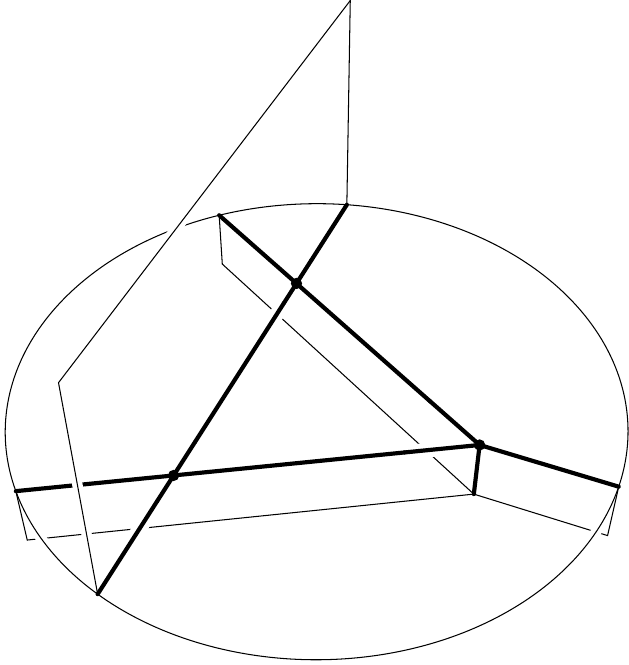}
\label{Fig:ClearTriangle1}
}
\quad
\subfloat[]{
\labellist
\hair 2pt \small
\pinlabel $R$ at 65 220
\pinlabel $f$ at 139 133
\endlabellist
\includegraphics[width = 0.43\textwidth]{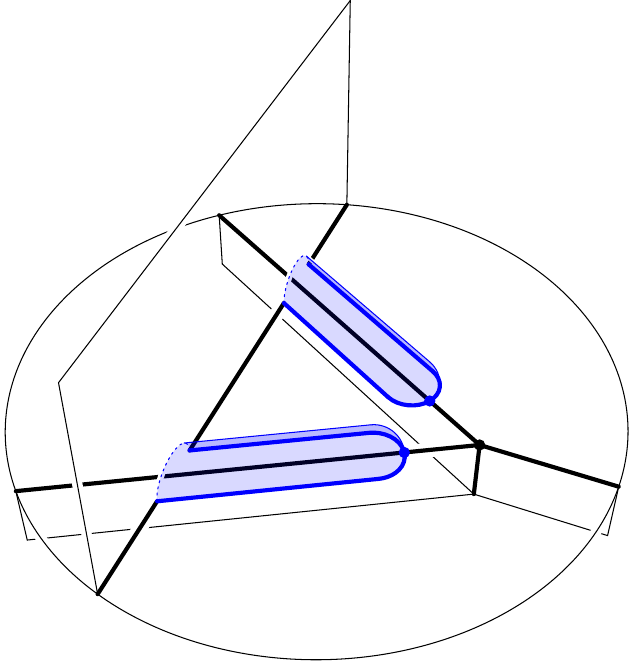}
\label{Fig:ClearTriangle2}
}
\caption{We use a sequence of finger moves to clear a triangle.}
\label{Fig:ClearTriangle}
\end{figure}

Once we have pushed the fingers past all intersections of $\gamma$ with $e_2$ and $e_3$, neither endpoint of $\gamma$ meets $f$, and all endpoints of intervals of $\gamma$ on $f$ meet the same edge of $f$.
We now repeatedly apply slide move $\textrm{A}^{-1}$ to clear $\gamma$ off of $f$.
We then apply the 3-2 move.

\subsection{Reverse bubble move}
\label{Sec:ReverseBubbleMove}

Suppose that applying a reverse bubble move to $\calF_i$ produces $\calF_{i+1}$.
Let $B$ be the region inside the bubble that is removed by the reverse bubble move.
We first assume that $i = q$, and so $B = Z$.
In this case let $e_1$ be the edge of $B$ that the snake path $\gamma$ terminates on. 
Let $x$ be the terminal point of $\gamma$.
Note that it is possible that the initial point of $\gamma$ is also on some edge of $B$.
However, in this case a prefix of $\gamma$ lies in $\bdy B$, since the source of $\gamma$ is peripheral.
Let $e_2$ and $e_3$ be the other edges of $B$.
Label the bigons of $B$ as $b_1$, $b_2$, and $b_3$, with $b_i$ opposite $e_i$.

We first clear $\gamma$ off of the bigon $b_1$ using the same process as in \refsec{2-0Move}.
Note that when pushing fingers along $\bdy B$, the finger is never part of the region $B$. 
This is because $\bdy B$ is embedded (because we are about to perform a reverse bubble move). 
Since the finger is never part of $B$, truncating (as in \refsec{TruncateSnake}) never replaces $\gamma$ with a prefix. 
Thus $x$ remains the terminal point of $\gamma$.

Next we push fingers in along $e_1$ from both sides towards $x$.
See \reffig{ClearBubble}.

\begin{figure}[htbp]
\subfloat[]{
\labellist
\hair 2pt \small
\pinlabel $e_1$ [b] at 103 81
\pinlabel $b_1$ at 153 197
\pinlabel $x$ [b] at 153 107
\endlabellist
\includegraphics[width = 0.43\textwidth]{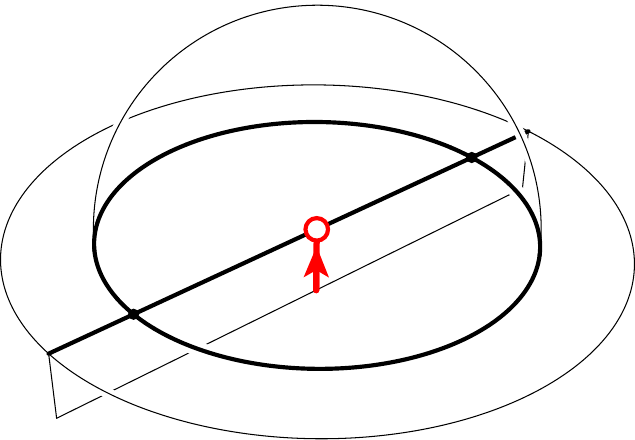}
\label{Fig:ClearBubble1}
}
\quad
\subfloat[]{
\labellist
\hair 2pt \small
\endlabellist
\includegraphics[width = 0.43\textwidth]{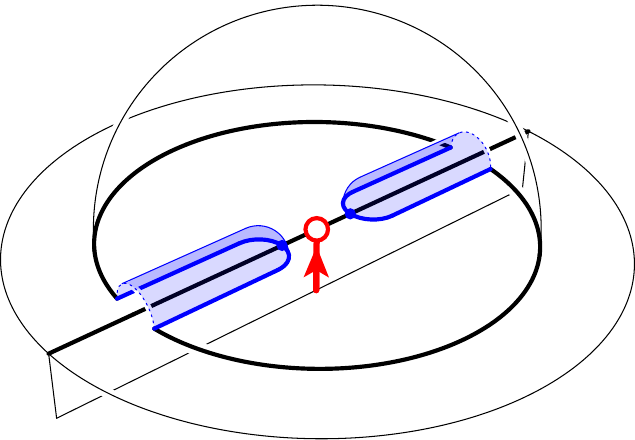}
\label{Fig:ClearBubble2}
}
\caption{We use a sequence of finger moves to clear a bubble.}
\label{Fig:ClearBubble}
\end{figure}

Again the fingers are not part of $B$, so $x$ remains the terminal point of $\gamma$.
Note that in pushing these fingers inwards, we never make an EV finger move (see \refsec{FingerMoves}). 
This is because exactly one end of the snake path $\gamma$ points into $B$, and it does so at $x$.
Therefore, if we meet the initial point of $\gamma$, we do so with an EH finger move. 
After we have applied the EH finger move, the initial point of $\gamma$ is no longer in contact with $B$.

Once we have pushed the fingers past all intersections with $\gamma$ (other than the terminal point $x$), the snake path $\gamma$ meets $B$ at $x$, as well as a disjoint collection of arcs.
These arcs fall into two subcollections.
One subcollection is contained in $b_2$, with endpoints only on $e_3$, while the 
other subcollection is contained in $b_3$, with endpoints only on $e_2$.
As in previous cases these can be removed by slide move $\textrm{A}^{-1}$, applied to outermost arcs first.

This done, $\gamma$ meets $B = Z$ only at $x$. 
We now remove the bubble and the snake by performing 2-0 moves, following \refdef{BuildSnake} in reverse.

Finally, we deal with the case that $p \leq i < q$, so $B \neq Z$.
This is very similar but slightly simpler than the case when $B = Z$. 
Here we choose the edge $e_1$ arbitrarily.
We clear the opposite bigon $b_1$ as before, then clear the edge $e_1$ again using two fingers, this time pointing towards an arbitrarily chosen point $x$ on $e_1 - \gamma$.
Again we clear $b_2$ and $b_3$ using slide move $\textrm{A}^{-1}$.
This done, $\gamma$ does not meet $\bdy B$.
We now perform the reverse bubble move.

\subsection{The proof is complete}

We have now dealt with all types of moves that can appear in the sequence $(\calF_i)$.
Following the outline given at the start of \refsec{TraversingWithSnakes}, 
this completes the proof of \refthm{ConnectivityWith0-2}. \qed

\appendix

\section{Many infinitely anchorable representations}
\label{Sec:Zariski}

We now recall a few definitions from the theory of character varieties over $\PSL(2, \CC)$, following~\cite[Section~3]{BoyerZhang98}.
Suppose that $M$ is a compact, connected, oriented three-manifold with non-empty boundary.
Suppose that the interior of $M$ admits a finite volume complete hyperbolic metric.  
We use $\rho_0$ to denote a discrete and faithful representation from $\pi_1(M)$ to $\PSL(2, \CC)$.

Suppose that $\rho \from \pi_1(M) \to \PSL(2,\CC)$ is any representation.
For any $\gamma$ in $\pi_1(M)$ we use $\tr^2_\rho(\gamma)$ to denote the square of the trace of $\rho(\gamma)$.
This gives a function $\tr^2_\rho$ from $\pi_1(M)$ to $\CC$;
we call this the \emph{character} of $\rho$. 
We use $X(M)$ to denote the resulting \emph{character variety}.
We use the notation $\tr^2(\gamma) \from X(M) \to \CC$ to denote the regular function that given the character of $\rho$ returns the squared trace $\tr^2_\rho(\gamma)$.

The goal of this appendix is to prove the following.

\begin{theorem}
\label{Thm:InfinitelyAnchorableZariskiNeighbourhood}
Suppose that $M$ is a compact, connected, oriented three-manifold with non-empty boundary.
Suppose that the interior of $M$ admits a finite volume complete hyperbolic metric.
Then there is a neighbourhood of $\rho_0$ (in the Zariski topology on $X(M)$) that consists solely of infinitely anchorable representations.
\end{theorem}

\begin{proof}
Our hypotheses imply that all boundary components of $M$ are tori.
Fix one such boundary component $C$.
Let $c \in \Delta_M$ be an elevation of $C$.
Fix a point $p \in C$.
Fix generators $\mu$ and $\lambda$ in $\pi_1(C, p)$.

Suppose that $\gamma$ lies in $\pi_1(M, p)$; 
let $\Gamma = \subgp{\mu, \lambda, \gamma}$ be the corresponding subgroup of $\pi_1(M)$.
Choosing $\gamma$ appropriately, and appealing to Klein's ``combination theorem''~\cite[III.16]{Klein83},
we may assume that all peripheral elements of $\pi_1(M, p)$ (in $\Gamma$) are conjugate (in $\Gamma$) into $\pi_1(C, p)$.
(For other, somewhat more modern, statements of the combination theorem see~\cite[page~210]{Fatou30} or~\cite[page~499]{Maskit65}.)


As we proceed we will impose (finitely many) algebraic inequalities of the form $\tr^2(w) \neq r$ where $w$ is an element in $\Gamma$ and where $r$ lies in $\CC$.  
We will always check that our inequalities hold at $\rho_0$.
Finally, we will prove that any representation $\rho$ satisfying our inequalities is infinitely anchorable.
\begin{equation}
\label{Eqn:NoInvolutions}
\tr^2(\mu) \neq 0, \quad \tr^2(\lambda) \neq 0, \quad \tr^2(\gamma) \neq 0
\end{equation}

An element of $\PSL(2, \CC)$ has trace equal to zero if and only if it is conjugate to an order two rotation~\cite[Theorem~4.3.1]{Beardon95}.  
Thus \refeqn{NoInvolutions} holds at $\rho_0$ because $\rho_0$ is faithful and because $\pi_1(M)$ is torsion free. 

From \refeqn{NoInvolutions} and from \reflem{NonAnchorable} we deduce that $\Fix_\rho(c)$ is non-empty. 
Therefore $\rho$ is anchorable at $c$.
Choosing framings for the finitely many other torus boundary components, and imposing \refeqn{NoInvolutions} for those meridians and longitudes, implies that $\rho$ is anchorable.
\begin{equation}
\label{Eqn:NoCommutation1}
\tr^2([\mu, \gamma]) \neq 4, \quad \tr^2([\lambda, \gamma]) \neq 4
\end{equation}

This holds at $\rho_0$ because the commutators $[\mu, \gamma]$ and $[\lambda, \gamma]$ are not conjugate in $\Gamma$ into $\pi_1(C, p)$
and because $\rho_0$ sends non-peripheral elements to loxodromics.

From \refeqn{NoCommutation1} we deduce that $\rho(\mu)$ and $\rho(\lambda)$ are not the identity.
Thus $\Fix_\rho(c)$ contains one or two points.
\begin{equation}
\label{Eqn:GammaNotParabolic}
\tr^2(\gamma) \neq 4
\end{equation}

This holds at $\rho_0$ because $\gamma$ is non-peripheral.
We impose this condition only to simplify our case analysis.

\begin{case*}
$\Fix_\rho(c)$ is a single point.
\end{case*}

Suppose that $\Fix_\rho(c) = \{z\}$.  
By~\cite[Theorem~4.3.6]{Beardon95}, since $\rho(\mu)$ and $\rho(\lambda)$ commute and are not of order two, they are parabolic.
Thus $\rho(\pi_1(C,p))$ is infinite.
\begin{equation}
\label{Eqn:NoCommutation2}
\tr^2([\mu, \gamma \mu \gamma^{-1}]) \neq 4, \qquad \tr^2([\lambda, \gamma \lambda \gamma^{-1}]) \neq 4
\end{equation}

This holds at $\rho_0$ because the commutators $[\mu, \gamma \mu \gamma^{-1}]$ and $[\lambda, \gamma \lambda \gamma^{-1}]$ are not conjugate in $\Gamma$ into $\pi_1(C, p)$.

By \refeqn{GammaNotParabolic} we have that $\rho(\gamma)$ is either elliptic or loxodromic.
We deduce from \refeqn{NoCommutation2} that $z$ is not a fixed point of $\rho(\gamma)$.
Let $A$ be the axis of $\rho(\gamma)$.
We deduce that $\rho(\pi_1(C,p)) \cdot A$ is infinite.
Thus the orbit $\rho(\Gamma) \cdot z$ is infinite.
Thus $\rho$ is infinitely anchorable, as desired.

\begin{case*}
$\Fix_\rho(c)$ is a pair of points.
\end{case*}

By \refeqn{GammaNotParabolic} we have that $\rho(\gamma)$ is not parabolic.
Thus $\Fix_\rho(\gamma)$ is also a pair of points.
By \refeqn{NoCommutation1}, by \refeqn{NoInvolutions}, and by~\cite[Theorem~4.3.6]{Beardon95} we have that $\Fix_\rho(c) \neq \Fix_\rho(\gamma)$.

\begin{subcase*}
$\rho(\gamma)$ is loxodromic.
\end{subcase*}

Suppose that $\Fix_\rho(c)$ = $\{z, w\}$ with, say, $z \notin \Fix_\rho(\gamma)$.
Thus the orbit $\rho(\subgp{\gamma}) \cdot z$ is infinite, and we are done.

\begin{subcase*}
$\rho(\mu)$ (or $\rho(\lambda)$) is loxodromic.
\end{subcase*}

Set $\delta = \gamma \mu \gamma^{-1}$.
Note that $\rho(\delta)$ is loxodromic.
By \refeqn{NoCommutation2} we have that $\delta$ does not commute with $\mu$.
Again applying~\cite[Theorem~4.3.6]{Beardon95} we have that $\Fix_\rho(c) \neq \Fix_\rho(\delta)$.
So we may apply the subcase immediately above with $\delta$ in place of $\gamma$.

\begin{subcase*}
$\rho(\mu)$, $\rho(\lambda)$, and $\rho(\gamma)$ are elliptic.
\end{subcase*}

Let $A$ be the common axis of $\rho(\mu)$ and $\rho(\lambda)$.
Let $B$ be the axis of $\rho(\gamma)$.

\begin{subsubcase*}
$A = B$.
\end{subsubcase*}

This implies that $\mu$ and $\gamma$ commute~\cite[Theorem~4.3.6]{Beardon95}, contradicting \refeqn{NoCommutation1}.

\begin{subsubcase*}
$A$ and $B$ meet at a point of $\HH^3$.
\end{subsubcase*}

Recall that $\Gamma = \subgp{\mu, \lambda, \gamma}$.
Thus $\rho(\Gamma)$ is conjugate into $\SO(3)$.
By \refeqn{NoCommutation1}, we have that $\rho(\Gamma)$ is not cyclic (or infinite cyclic).
By \refeqn{NoInvolutions}, we have that $\rho(\Gamma)$ is not dihedral (or infinite dihedral).
\begin{equation}
\label{Eqn:NoTorsion}
\tr^2(\mu), \tr^2(\lambda), \tr^2(\gamma) \neq 4\cos^2(\pi/n) \quad \text{for $n = 3, 4, 5$}
\end{equation}

This holds at $\rho_0$ because $\rho_0$ is faithful and $\pi_1(M)$ is torsion free. 

From \refeqn{NoTorsion} we deduce that $\rho(\Gamma)$ is not a platonic subgroup of $\SO(3)$.
We now work in $S^2 = \bdy_\infty\HH^3$ equipped with the round metric $d_{S^2}$.
We arrange matters so that $\rho(\gamma)$ acts on $S^2$ by isometries.
We further arrange matters so that $\Fix_\rho(c) = \{N, S\}$, the north and south poles.

\begin{claim*}
The orbit $\rho(\Gamma) \cdot N$ is infinite.
\end{claim*}

Given this, $\rho$ is infinitely anchorable, as desired.

\begin{proof}[Proof of Claim]
Suppose not.
Pick some $N'$ in the orbit that minimises $d_{S^2}(N, N')$.
Since $\rho(\Gamma)$ is neither cyclic nor dihedral, $\rho(\gamma)$ does not fix and does not interchange $N$ and $S$. 
Thus $N' \neq S$. 

Since the orbit is finite, we deduce that $\rho(\mu)$ is finite order.
By \refeqn{NoCommutation1} we have that $\rho(\mu)$ is not the identity.
By \refeqn{NoTorsion} we have that the order of $\rho(\mu)$ is at least six.
Taking $m =  \rho(\mu^k)$ for an appropriate power $k$, we find that 
\[
0 < d_{S^2}(N',m(N')) < d_{S^2}(N, N')
\]
This is a contradiction, and we are done.
\end{proof}

\begin{subsubcase*}
$A$ and $B$ meet, but only at a point of $\bdy_\infty \HH^3$.
\end{subsubcase*}

In this case, $\rho([\mu, \gamma])$ is parabolic, contradicting \refeqn{NoCommutation1}.

\begin{subsubcase*}
$A$ and $B$ are disjoint, even at $\bdy_\infty \HH^3$.
\end{subsubcase*}

In this case $\Gamma$ contains an element $\delta$ so that $\rho(\delta)$ is not elliptic (see~\cite[page~62]{Fatou30} or~\cite[page~1119]{LyndonUllman67}). 
In fact we may assume that $\rho(\delta)$ does not fix both endpoints of $A$.
We deduce that the orbit of $\Fix_\rho(c)$ under the action of $\rho(\delta)$ is infinite.
Thus $\rho$ is infinitely anchorable, as desired.
\end{proof}

\section{Constructing $L$--essential triangulations}
\label{App:HMP}

Our next result is inspired by, and used in, the work of \cite{PandeyWong24}, following \cite[Proposition~5.1]{HowieMathewsPurcell21}.
Recall that $\Fix_\rho$ is defined in \refdef{Anchorable}.

\begin{proposition}
\label{Prop:HMP}
Suppose that $M$ is a connected, compact, oriented 
three-manifold with boundary consisting of $m+1 \geq 2$ tori $T_0, T_1,\ldots,T_m$. 
Fix $k > 0$.
Suppose that $S = (s_k, \ldots, s_m)$ is a collection of slopes on $(T_k, \ldots, T_m)$.
Let $M(S)$ be the resulting Dehn filled manifold.
Let $\iota \from M \to M(S)$ be the resulting inclusion.
Suppose that $\rho_S \from \pi_1(M(S)) \to \PSL(2,\CC)$ is a representation. 
Let $\rho = \iota^*(\rho_S)$.
Let $c_0 \in \Delta_M$ be some fixed elevation of $T_0$.
Set $\Gamma = \rho(\pi_1(M))$ and $\Gamma_j = \rho(\pi_1(T_j))$.
\begin{enumerate}[label=(\Alph*)]
\item 
\label{Itm:InfinitelyAnchorableAtc_0}
Suppose that $\rho$ is infinitely anchorable at $c_0$. 
\item 
\label{Itm:AtLeast3}
Suppose that, for $j \geq k$, we have $|\Gamma_j| \geq 3$.  
\end{enumerate}

Then there exists an ideal triangulation $\calT$ of $M$ with the following properties.
\begin{enumerate}
\item 
\label{Itm:TwoTetrahedra}
In $M$, the cusp corresponding to $T_j$ for any $j \geq k$ meets exactly two ideal tetrahedra, $t_{j,1}$ and $t_{j,2}$. 
Each of these tetrahedra meets $T_j$ in exactly one ideal vertex.
\item 
\label{Itm:Fill}
We obtain an ideal triangulation $\calT(S)$ of the manifold $M(S)$ as follows.
For each filled cusp $T_j$, remove the ideal tetrahedra $t_{j,1}$ and $t_{j,2}$ and fold the resulting parallelogram in the boundary along its diagonal
as shown in~\cite[Figure~5]{PandeyWong24}.
\item 
\label{Itm:Basis}
There exists a choice of generators for $H_1(T_0;\ZZ)$, represented by curves $\mu_0$ and $\lambda_0$, such that $\mu_0$ and $\lambda_0$ meet the cusp triangulation inherited from $\calT$ in a sequence of arcs cutting off single vertices of triangles, without backtracking, and such that $\mu_0$ and $\lambda_0$ are disjoint from the tetrahedra $t_{j,1}$ and $t_{j,2}$, for all $j = k,\ldots, m$.
\item 
\label{Itm:RhoRegular}
There are solutions $Z$ and $Z(S)$ to the gluing equations for $\calT$ and $\calT(S)$ respectively, such that for $j \geq k$, deleting the coordinates $(j,1)$ and $(j,2)$ from the list $Z$ gives the list $Z(S)$.
Moreover, $\rho_Z$ and $\rho_{Z(S)}$ are conjugate to $\rho$ and $\rho_S$ respectively.
\end{enumerate}
\end{proposition}

\begin{remark}
By \refthm{InfinitelyAnchorableZariskiNeighbourhood}, if $M$ is hyperbolic and $\rho$ is close (in the Zariski, and therefore the euclidean, topology) to a discrete and faithful representation on $M$ then hypothesis 
\ref{Itm:InfinitelyAnchorableAtc_0} holds.
Furthermore, Equations~\ref{Eqn:NoInvolutions} and~\ref{Eqn:NoCommutation1} imply that hypothesis \ref{Itm:AtLeast3} holds.
\end{remark}

\begin{remark}
\label{Rem:EssentialsCollapse}
For ideal triangulations (with no material vertices), the concepts of weakly $L$--essential and $L$--essential (given in \refdef{LEssentialPartiallyIdeal}) agree.
\end{remark}

\begin{proof}[Proof of \refprop{HMP}]
As in the proof of \cite[Proposition~5.1]{HowieMathewsPurcell21}, we start with a material triangulation of $M$ where there is exactly one vertex in every boundary component $T_j$ for $j \geq 0$.
For each $j \geq k$ we layer tetrahedra onto the triangulation of $T_j$ until the three edges have slopes $p_j$, $q_j$, and $r_j$.
These are chosen so that the Farey sum of $p_j$ and $q_j$ that does not give $r_j$ instead gives $s_j$.
In $M(S)$ the curve $s_j$ is null-homotopic.
We deduce that $\rho(s_j)$ is trivial. 
Moreover, in $M(S)$ the curves $p_j$ and $q_j$ are identified.
Thus $\rho(r_j)$ is the square of $\rho(p_j) = \rho(q_j)$; 
furthermore these generate $\rho( \pi_1(T_j))$.
Thus $\Gamma_j = \rho( \pi_1(T_j))$ is cyclic.
Hypothesis~\ref{Itm:AtLeast3} now implies that $\rho(p_j)$ and $\rho(r_j)$ are non-trivial.

\begin{figure}[htbp]
\includegraphics[width = \textwidth]{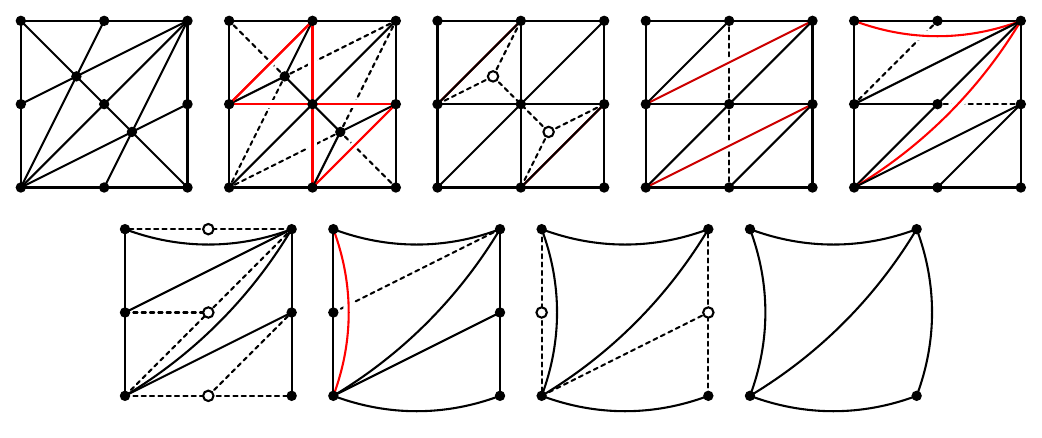}
\caption{Layering onto the barycentric subdivision to return it to the two triangle triangulation of the torus.
Each layered tetrahedron implements a bistellar move on the two-dimensional triangulation of the torus.
We draw the top edge of each 2-2 tetrahedron in red.
We draw the bottom edges of 2-2 and 3-1 tetrahedra with dotted lines.
From left to right along the top row followed by the bottom row we do six 2-2 moves, two 3-1 moves, two 2-2 moves, two 2-2 moves, two 3-1 moves, one 2-2 move, and one 3-1 move. }
\label{Fig:BarycentricTorusLayering}
\end{figure}

Now take the first barycentric subdivision of the triangulation.
For each $j \geq k$ we layer tetrahedra onto the boundary torus $T_j$ as shown in \reffig{BarycentricTorusLayering} to recover the two-triangle triangulation of $T_j$.
Let $v_j$ be the unique vertex in the triangulation of $T_j$.
The resulting triangulation has the property that the only material vertices with edge loops are the vertices $v_j$.

Following the proof of \cite[Proposition~5.1]{HowieMathewsPurcell21}, we now cone each $T_j$ (for $j \geq 0$) and its triangulation to a new ideal vertex $w_j$.
This gives a partially ideal triangulation, $\calS$ say, of $M$.

We now choose an anchoring $L$ of $\rho$, making sure that $L(c_0) = z_0$.
The triangulation $\calS$ is weakly $L$--essential (rather than $L$--essential), but only due to the loops based at the vertices $v_j$ for $j \geq k$.

\begin{figure}[htbp]
\labellist
\hair 2pt \small
\pinlabel $\mu_0$ [r] at 0 105
\pinlabel $\lambda_0$ [b] at 205 290
\endlabellist
\includegraphics[width = 0.4\textwidth]{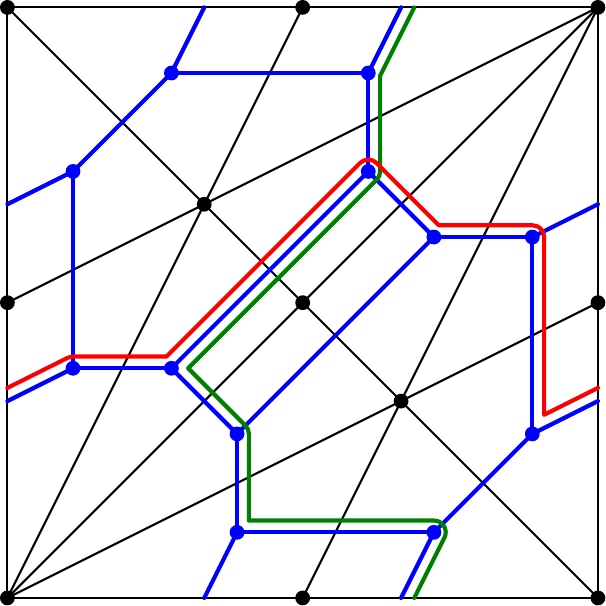}
\caption{The triangulation of $T_0$ after barycentric subdivision (drawn in black). 
The one-skeleton of $\bdy C_0$ (drawn in blue) is dual to the triangulation.  
The curves $\mu_0$ (drawn in red) and $\lambda_0$ (drawn in green) are carried by the one-skeleton of $\bdy C_0$.}
\label{Fig:BarycentricSubdivTorusCurves}
\end{figure}

Let $\calF_0$ be the foam dual to the triangulation $\calS$.
For each $j \geq k$ let $B_j$ be the material region dual to the material vertex $v_j$.
For each $j \geq 0$ let $C_j$ be the peripheral region dual to the ideal vertex $w_j$.
(There are many other material regions in the complement of $\calF_0$ that we do not name.)

We now choose curves $\mu_0$ and $\lambda_0$ in the one-skeleton of $\bdy C_0$, as shown in \reffig{BarycentricSubdivTorusCurves}.
These generate $H_1(T_0, \ZZ)$.
We set $\calA$ to be the union of the faces of $\bdy C_j$ (for $j \geq k$) together with the edges of $\bdy C_0$ that carry $\mu_0$ and $\lambda_0$.
Note that there are edges of $\bdy C_0$ that are not in $\calA$.
Also note that $\cover{\calF}_0 - \phi^{-1}\left(\closure{\calA}\right)$ is connected.

We proceed by induction.
As in the proof of \refthm{ExistsTriangulation}, we recursively build and inflate snake paths from elevations of $C_0$ to the material regions.
Suppose that $\calF_i$ is weakly $L$--essential.
Suppose that $\cover{\calF}_i - \phi^{-1}\left(\closure{\calA}\right)$ is connected.
If $\calF_i$ has no material regions then by \refrem{EssentialsCollapse}, $\calF_i$ is $L$--essential. 
In this case we are done with the recursive construction.

Suppose instead that $\cover{\calF}_i$ has a material region $D$.
We produce an $L$--self avoiding snake path $\gamma_i$ with target $D$.
There are two cases, as $\phi(D)$ is or is not one of the $B_j$.

Suppose first that $\phi(D)$ is not one of the $B_j$.
Then for any $\tau \in \pi_1(M)$ we have that $D$ and $\tau(D)$ do not share a face.
Applying hypothesis \ref{Itm:InfinitelyAnchorableAtc_0} we choose $g \in \pi_1(M)$ so that $\ell = \rho(g) \cdot z_0$ is distinct from the labels on all regions incident to $D$.
Thus we may apply \refalg{CreateSnakePath} with material region $D$, with label $\ell$, and with $\calA$ defined as above.
This gives us the $L$--self avoiding snake path $\gamma_i$ with target $D$.

Suppose instead that $\phi(D) = B_j$.
In this case we take the basepoint of $\pi_1(M)$ to be $v_j$.
Applying hypothesis \ref{Itm:InfinitelyAnchorableAtc_0} and recalling that $\rho(p_j)$ is non-trivial (so has at most two fixed points), we choose $g \in \pi_1(M, v_j)$ and define $\ell = \rho(g) \cdot z_0$ so that 
\begin{itemize}
\item $\ell$ is distinct from the labels on all regions incident to $D$ and
\item $\ell$ is not a fixed point of $\rho(p_j)$.
\end{itemize}

\begin{claim}
For each $\tau \in \Stab(\ell)$, there is no face with $D$ on one side and $\tau(D)$ on the other.
\end{claim}

\begin{proof}
We prove the contrapositive.
Suppose that for $\tau \in \pi_1(M, v_j)$ there is a face $f$ of $\cover{\calF}$ with $D$ on one side of $f$ and with $\tau(D)$ on the other.
We now must show that $\tau \notin \Stab(\ell)$.

Since $\phi(D) = B_j$, we have that $\phi(f)$ is dual to one of the three edges connecting $v_j$ to itself. 
We deduce that $\tau$ is one of the slopes $p_j$, $q_j$, or $r_j$ (or their inverses), thought of as elements of $\pi_1(T_j, v_j)$.
Our choice of $g$ implies that $\ell =  \rho(g) \cdot z_0$ is not fixed by $\rho(\tau)$.
Thus $\tau \notin \Stab(\ell)$, as desired.
\end{proof}

Thus we may apply \refalg{CreateSnakePath} with material region $D$, with label $\ell$, and with  $\calA$ defined as above.
Again this gives us the $L$--self avoiding snake path $\gamma_i$ with target $D$.

In either case, \reflem{EssentialSnake} now gives us that $\calF_{i+1} = \calF_i(\gamma_i)$ is weakly $L$--essential.
Also in either case, by our choice of $\calA$, for $j \geq k$ the snake path $\gamma_i$ does not meet any of the regions $C_j$, nor does it meet any of the edges of $\bdy C_0$ that carry $\mu_0$ or $\lambda_0$. 
Thus, for $j \geq k$ the combinatorics of $\bdy C_j$ in $\calF_{i+1}$ is the same as it is in $\calF_{i}$.
Similarly, the edges of $\bdy C_0$ that carry $\lambda_0$ or $\mu_0$ do not change. 
In particular, $\calA$ does not change from $\calF_i$ to $\calF_{i+1}$.

Because no faces incident to either end of the snake path $\gamma_i$ are in $\calA$, it follows that no faces on the boundary of the snake are in $\calA$. 
Recalling that $\cover{\calF}_i - \phi^{-1}\left(\closure{\calA}\right)$ is connected and consulting Figures~\ref{Fig:InflateSnakePath1} and~\ref{Fig:InflateSnakePath3}, we conclude that 
$\cover{\calF}_{i+1} - \phi^{-1}\left(\closure{\calA}\right)$ is connected.

The recursive process terminates because at each step we reduce the number of material regions by one.
Let $\calF_n$ be the final foam. 
Let $\calT$ be the triangulation dual to $\calF_n$.
By \refrem{EssentialsCollapse} the triangulation $\calT$ is $L$--essential.

For each $j \geq k$ the combinatorics of $C_j$ in $\calF_n$ is the same as it is in $\calF_0$. 
That is, there are two vertices on $\bdy C_j$. 
These are dual to two tetrahedra in $\calT$, giving us \refitm{TwoTetrahedra}.
These tetrahedra $t_{j,1}$ and $t_{j,2}$ are combinatorially identical in $\calT$ and in $\calS$. 
Thus, by the choice of slopes $p_j$, $q_j$, and $r_j$, removing $t_{j,1}$ and $t_{j,2}$ and folding along the slope $r_j$ for each $j \geq k$ produces $\calT(S)$, an ideal triangulation of $M(S)$.
Thus we have $\refitm{Fill}$.
By our choice of $\calA$, the curves $\mu_0$ and $\lambda_0$ are combinatorially identical in $\calT$ and in $\calS$. 
In particular, in $\calF_n$ the curves $\mu_0$ and $\lambda_0$ do not pass through the vertices dual to $t_{j,1}$ and $t_{j,2}$ for each $j \geq k$.
This gives \refitm{Basis}.

By \reflem{RecoverRho}, there is a solution $Z$ to the gluing equations on $\calT$ so that $\rho_Z$ is conjugate to $\rho$. 
Now consider $\calT(S)$.
The set $\Delta_{M(S)}$ is obtained from $\Delta_M$ in two steps.
\begin{enumerate}
\item Remove preimages under the covering map $\phi$ of the tori $T_j$ for $j \geq k$, giving a set $\Delta'_M$.
\item Quotient $\Delta'_M$ by the relationship induced by folding along the slopes $r_j$.
\end{enumerate}
Let $\Phi \from \Delta'_M \to \Delta_{M(S)}$ be this quotient map.
We now define $L_S$, a labelling on $\Delta_{M(S)}$, as follows.
Suppose that $c$ lies in $\Delta_{M(S)}$.
Choose any element $c'$ lying in $\Phi^{-1}(c)$.
We set $L_S(c) = L(c')$.
This is well-defined because $L$ is given by $\rho$, which is the pullback of $\rho_S$.

Suppose that $e$ is an edge of $\cover{\calT}(S)$.
Suppose that $e$ has endpoints $c$ and $d$ in $\Delta_{M(S)}$.
Pick $e'$ to be any edge of $\cover{\calT}$ sent to $e$ under the quotient map.
Let $c'$ and $d'$ in $\Delta_M$ be the endpoints of $e'$.
Since $\calT$ is $L$--essential we have that $L(c')$ and $L(d')$ are distinct.
Thus $L_S(c) = L(c')$ and $L_S(d) = L(d')$ are distinct.
Thus $\calT(S)$ is $L_S$--essential.
In fact, the shapes given to the tetrahedra of $\calT(S)$ given by $L_S$ are the same as those given to the corresponding tetrahedra of $\calT$ by $L$.
Thus by deleting the pairs of tetrahedron shapes with coordinates $(j, 1)$ and $(j, 2)$ for $j \geq k$ we obtain the desired solution $Z(S)$.
This gives \refitm{RhoRegular}.
\end{proof}

\renewcommand{\UrlFont}{\tiny\ttfamily}
\renewcommand\hrefdefaultfont{\tiny\ttfamily}
\bibliographystyle{plainurl}
\bibliography{avoid_inessential_edges.bib}
\end{document}

%% file: header_basic.tex
\usepackage{amsmath} 
\usepackage{amsthm} 
\usepackage{amssymb} 
\usepackage{mathtools}

\usepackage{cite}
\makeatletter
\newcommand{\citecomment}[2][]{\citen{#2}#1\citevar}
\newcommand{\citeone}[1]{\citecomment{#1}}
\newcommand{\citetwo}[2][]{\citecomment[,~#1]{#2}}
\newcommand{\citevar}{\@ifnextchar\bgroup{;~\citeone}{\@ifnextchar[{;~\citetwo}{]}}}
\newcommand{\citefirst}{\@ifnextchar\bgroup{\citeone}{\@ifnextchar[{\citetwo}{]}}}
\newcommand{\cites}{[\citefirst}
\makeatother

\usepackage{microtype} 
\usepackage{pinlabel} 

\newcommand{\sqdiamond}{\tikz [x=1.2ex,y=1.85ex,line width=.1ex,line join=round, yshift=-0.285ex] \draw  (0,.5) -- (0.75,1) -- (1.5,.5) -- (0.75,0) -- (0,.5) -- cycle;}%
\renewcommand{\Diamond}{$\sqdiamond$} 

\usepackage[scaled=0.9]{sourcecodepro} 

\usepackage{enumitem}
\usepackage{color}
\usepackage{subfig, caption}
 
\usepackage{wrapfig}
\usepackage{pdflscape}
\usepackage{array}
\usepackage{tikz-cd}
\usepackage{longtable, booktabs}

\usepackage[hidelinks, pagebackref]{hyperref}

\makeatletter
\define@key{href}{font}{#1}
\makeatother
\usepackage{xpatch}
\newcommand\hrefdefaultfont{\ttfamily}
\xpatchcmd\href{\setkeys{href}{#1}}{\setkeys{href}{font=\hrefdefaultfont,#1}}{}{\fail}

\renewcommand*{\backref}[1]{}
\renewcommand*{\backrefalt}[4]{
  \ifcase #1 
  [No citations.]
  \or [#2]
  \else [#2]
  \fi }

\let\originalleft\left
\let\originalright\right
\renewcommand{\left}{\mathopen{}\mathclose\bgroup\originalleft}
\renewcommand{\right}{\aftergroup\egroup\originalright}

\newcommand{\calA}{\mathcal{A}}

\newcommand{\calF}{\mathcal{F}}

\newcommand{\calL}{\mathcal{L}}

\newcommand{\calN}{\mathcal{N}}

\newcommand{\calS}{\mathcal{S}}
\newcommand{\calT}{\mathcal{T}}

\newcommand{\CC}{\mathbb{C}}

\newcommand{\HH}{\mathbb{H}}

\newcommand{\RR}{\mathbb{R}}

\newcommand{\ZZ}{\mathbb{Z}}

\newcommand{\st}{\mathbin{\mid}} 
\newcommand{\from}{\colon} 
 
\newcommand{\homeo}{\mathrel{\cong}} 

\newcommand{\isom}{\cong} 

\newcommand{\cover}[1]{{\widetilde{#1}}}

\newcommand{\closure}[1]{{\overline{#1}}}

\newcommand{\bdy}{\partial} 

\newcommand{\interior}{{\operatorname{interior}}}

\newcommand{\CP}{\mathbb{CP}} 

\newcommand{\group}[2]{{\langle #1 \st #2 \rangle}}
\newcommand{\subgp}[1]{{\langle #1 \rangle}}

\newcommand{\Stab}{\operatorname{Stab}}

\newcommand{\PSL}{\operatorname{PSL}} 
\newcommand{\SO}{\operatorname{SO}} 

\newcommand{\tr}{\operatorname{tr}} 



%% file: header_subtle.tex
\theoremstyle{plain}
\newtheorem{XXXtheoremQED}[equation]{Theorem} 
  {\pushQED{\qed}\begin{XXXtheoremQED}}
  {\popQED\end{XXXtheoremQED}}

\newcommand{\fakeenv}{} 

\newenvironment{restate}[2]  
{ 
 \renewcommand{\fakeenv}{#2} 
 \theoremstyle{plain} 
 \newtheorem*{\fakeenv}{#1~\ref{#2}} 
 \begin{\fakeenv}
}
{
 \end{\fakeenv}
}

\newenvironment{restated}[2]  
{ 
 \renewcommand{\fakeenv}{#2} 
 \theoremstyle{definition} 
 \newtheorem*{\fakeenv}{#1~\ref{#2}} 
 \begin{\fakeenv}
}
{
 \end{\fakeenv}
}

